\documentclass[a4paper]{article}
\usepackage{amssymb}
\usepackage{amsthm}
\usepackage{latexsym}
\usepackage{amsmath}
\usepackage{color}

\usepackage{comment}
\excludecomment{discuss}
\includecomment{detailed}
\excludecomment{note}

\newif\ifcol
\colfalse
\ifcol
\newcommand{\colorr}{\color[rgb]{0.8,0,0}}
\newcommand{\colorg}{\color[rgb]{0,0.5,0}}
\newcommand{\colord}{\color[rgb]{0.8,0.3,0}}
\newcommand{\colorlg}{\color[rgb]{0,0.8,0.3}}
\else
\newcommand{\colorr}{\color{black}}
\newcommand{\colorg}{\color{black}}
\newcommand{\colord}{\color{black}}
\newcommand{\colorlg}{\color{black}}
\fi

\newtheorem{definition}{Definition}[section]
\newtheorem{lemma}{Lemma}[section]
\newtheorem{proposition}{Proposition}[section]
\newtheorem{theorem}{Theorem}[section]
\newtheorem{remark}{Remark}[section]
\newtheorem{example}{Example}[section]
\newtheorem{corollary}{Corollary}[section]

\makeatletter
\@addtoreset{equation}{section}
\makeatother

\begin{document}
\title{
Parametric Inference for Nonsynchronously Observed Diffusion Processes in the Presence of Market Microstructure Noise
}
\author{Teppei Ogihara\\
\begin{small}The Institute of Statistical Mathematics, 10-3 Midori-cho, Tachikawa, Tokyo 190--8562, Japan \end{small}\\
\begin{small}School of Multidisciplinary Sciences, SOKENDAI (The Graduate University for Advanced Studies), \end{small}\\
\begin{small}Shonan Village, Hayama, Kanagawa 240-0193, Japan\end{small}
}
\date{}
\maketitle

\noindent
{\bf Abstract.}
We study parametric inference for diffusion processes when observations occur nonsynchronously and are contaminated by market microstructure noise.
We construct {\colorg a quasi-likelihood function and study asymptotic mixed normality of maximum-likelihood- and Bayes-type estimators based on it.}
We also prove the local asymptotic normality of the model and asymptotic efficiency of our estimator when the diffusion coefficients are constant and noise follows a normal distribution.
We conjecture that our estimator is asymptotically efficient even when the latent process is a general diffusion process.
{\colorg An estimator for the quadratic covariation of the latent process is also constructed. Some numerical examples show that 
this estimator performs better compared to existing estimators of the quadratic covariation. }
\\

\noindent
{\bf Keywords.}
asymptotic efficiency, Bayes-type estimation, diffusion processes, local asymptotic normality, nonsynchronous observations, 
parametric estimation, maximum-likelihood-type estimation, market microstructure noise

\section{Introduction}

Analysis of volatility and covariation is one of the most important subjects in the study of risk management of financial assets.
Studies of high-frequency financial data are increasingly significant as high-frequency financial data become increasingly available and computing technology develops.
While realized volatility has been studied as a consistent estimator of integrated volatility at high-frequency limits,
estimators of covariation of two securities are also important. The realized covariance, a natural extension of the realized volatility, 
is a consistent estimator of integrated covariation in ideal settings.

{\colorg However}, there are two significant problems in empirical analysis, one of which is the existence of observation noise.
When we model stock price data by a continuous stochastic process, we should assume that the observations are contaminated by additional noise as a way
to explain empirical evidence.
Consistent estimators of volatility under the presence of microstructure noise are investigated---for example, in Zhang, Mykland, and A${\rm \ddot{\i}}$t-Sahalia~\cite{zha-etal05},
Barndorff-Nielsen et al.~\cite{bar-etal08}, and Podolskij and Vetter~\cite{pod-vet09}---by using various data-averaging or resampling methods to reduce the influence of noise.
The other significant problem is that of nonsynchronous observation, namely, that we observe prices of different securities at different time points.
The realized covariance has serious bias under models of nonsynchronous observations,
though we can calculate the estimator by using some simple '{\it synchronization}' methods such as linear interpolation or the `previous tick' methods.
Hayashi and Yoshida~\cite{hay-yos05, hay-yos08, hay-yos11} and Malliavin and Mancino~\cite{mal-man02, mal-man09}
independently constructed consistent estimators for statistical models of diffusion processes with nonsynchronous observations.
There are also studies of covariation estimation under the simultaneous presence of microstructure noise and nonsynchronous observations. We refer interested readers to
Barndorff-Nielsen et al.~\cite{bar-etal11} for a kernel based method;
Christensen, Kinnebrock, and Podolskij~\cite{chr-etal10}, Christensen, Podolskij, and Vetter~\cite{chr-etal13} for a pre-averaged Hayashi--Yoshida estimator; 
A${\rm \ddot{\i}}$t-Sahalia, Fan, and Xiu~\cite{ait-etal10} for a method using the maximum likelihood estimator of a model with constant diffusion coefficients; 
and Bibinger et al.~\cite{bib-etal14} for a technique employing the local method of moments.

While the above studies concern estimators under non- or semi-parametric settings,
there are also studies about parametric inference of diffusion processes with high-frequency observations.
Genon-Catalot and Jacod~\cite{gen-jac94} constructed quasi-likelihood function and studied an estimator that maximizes it.
Gloter and Jacod~\cite{glo-jac01b} studied an estimator based on a quasi-likelihood function with noisy observations. Ogihara and Yoshida~\cite{ogi-yos14}
studied a maximum-likelihood-type estimator and a Bayes-type estimator on nonsynchronous observations without market microstructure noise.

One advantage of maximum-likelihood- and Bayes-type estimators is that they are asymptotically efficient in many models.
If a statistical model has the local asymptotic mixed normality (LAMN) property, then the results in Jeganathan~\cite{jeg82,jeg83} ensure that 
asymptotic variance of estimators cannot be smaller than a certain lower bound.
When some estimator attains this bound, it is called asymptotically efficient. 
For parametric estimation of diffusion processes on fixed intervals, Gobet~\cite{gob01} proved the LAMN property 
of the statistical model having equidistant observations, and an estimator in~\cite{gen-jac94} is asymptotically efficient.
Ogihara~\cite{ogi14} proved the LAMN property and asymptotic efficiency of estimators for the setting of~\cite{ogi-yos14}.
Gloter and Jacod~\cite{glo-jac01a} proved the local asymptotic normality (LAN) property 
for a statistical model with market microstructure noise when diffusion coefficients are deterministic,
and the estimator by Gloter and Jacod~\cite{glo-jac01b} is asymptotically efficient.
There are few studies about the efficiency of estimators that assume the presence of market microstructure noise and nonsynchronous observations. 
One exception is Bibinger et al.~\cite{bib-etal14}, who showed a lower bound of asymptotic variance of estimators in semi-parametric Cram\'er-Rao sense.
We need the LAN or LAMN property of the statistical model to obtain asymptotic efficiency of a parametric model.
To the best of our knowledge, this has not been studied for statistical models of noisy, nonsynchronous observations. 

This paper examines consistency and asymptotic mixed normality of a maximum-likelihood-type estimator and a Bayes-type estimator based on a quasi-likelihood function,
under the simultaneous presence of market microstructure noise and nonsynchronous observations.
We also study the LAN property of this model when diffusion coefficients are constants, as well as the asymptotic efficiency of our estimators.
We expect that our estimators are asymptotically efficient in the general cases. 
{\colord However, it is further difficult to obtain LAMN properties for {\colorg models of general diffusions}.}
This does not seem to have been obtained even for noisy, equidistant observations, and is left as future work. 
We will see by simulation that sample variance of the estimation error of our estimator is better than that of existing estimators for some examples in Section~\ref{simulation-section}. 
These results ensure that our estimator not only is the theoretical best for asymptotic behavior,
but also works well in practical finite samplings.

Our study has several advantages in addition to the above arguments regarding asymptotic efficiency.
\begin{enumerate}
\renewcommand{\labelenumi}{\roman{enumi})}
\item Our model also allows observation noise that follows a non-Gaussian distribution. 
We use a quasi-likelihood function for Gaussian noise, but our method is robust enough to allow misspecification of the noise distribution.
\item Since we obtain the results regarding asymptotic behaviors of the quasi-likelihood function as a byproduct,
many applications become available from the theory of maximum-likelihood-type estimation. 
For example, we can construct a theory of the likelihood ratio test and one-step estimators as an immediate application.
Further, the theory of information criteria is expected to follow from our results of quasi-likelihood functions. 
\item Our settings contain random sampling schemes where {\colord the maximum length of observation intervals is not bounded by any constant multiplication of the minimum length}.
{\colord This is the case} for some significant random sampling schemes, such as samplings based on Poisson or Cox processes. 
Our model encompasses such natural sampling schemes.
\end{enumerate}
\begin{discuss}
{\colorr previous studyでもポアソンを扱えるものは恐らくあるからPodolskijとかReiss et al.だけ見てunlike previous studiesとは言えない}
\end{discuss}

To obtain asymptotic mixed normality of our estimator, we {\colorg investigate} asymptotic behaviors
of a quasi-likelihood function of noisy, nonsynchronous observations.
{\colorg To this end, we need to specify the limit of some matrix trace related to a ratio of covariance matrices for two different values of parameters, as appearing in (\ref{Hn-diff-eq1}).}
The inverse of the covariance matrix of observation noise has nontrivial off-diagonal elements, 
and so the inverse of the covariance matrix of observations is far from a diagonal matrix.
This phenomenon is essentially different from the case of {\it synchronous} observations without noise (where the covariance matrix of observations is diagonal),
and the case of {\it nonsynchronous} observations without noise (where the inverse of the covariance matrix is not a diagonal matrix but is `{\it close}' to being one).

In a model of noisy, {\it synchronous} observations, the covariance matrix of a latent process is asymptotically equivalent 
to a unit matrix of the appropriate size, and is therefore simultaneously diagonalizable with the noise covariance.
Gloter and Jacod~\cite{glo-jac01a,glo-jac01b} used these facts and closed expressions for the eigenvalues of the noise covariance to identify the limit of the quasi-likelihood function,
but we cannot apply their idea because our sampling scheme is irregular and so not well approximated by a unit matrix.
Further, the sizes of the covariance matrices are different for different components of the process, which follows from nonsynchronousness.
In this paper, we deduce an asymptotically equivalent transform of the trace of the ratio of covariance matrices.
This transform changes sizes of matrices and matrix elements into local averages,
and arises from specific properties of the noise covariance matrix. We will see these results in Sections~\ref{tildeH-section} and~\ref{Hn-limit-section}. 

The remainder of this paper is organized as follows. In Section~\ref{results-section}, we describe our detailed settings and main results.
We propose a quasi-likelihood function for models with noisy, nonsynchronous observations, and construct a maximum-likelihood-type estimator based on it.
We introduce asymptotic mixed normality of our estimator and results about asymptotic efficiency in Section~\ref{mixed-normality-subsection}.
Section~\ref{LAMN-subsection} contains results about the LAN property of our model and the asymptotic efficiency of our estimator,
and Section~\ref{PLD-section} is devoted to results about Bayes-type estimators and convergence of moments of estimators.
Polynomial-type large deviation inequalities, introduced in Yoshida~\cite{yos06,yos11}, are key to deducing these results.
In Section~\ref{simulation-section} we will examine simulation results of our estimator for a simple example where the latent process is a Wiener process.
We also construct an estimator of the quadratic covariation and compare the performance of our estimator with that of other estimators.
The remaining sections are devoted to a proof of the main results.
Section~\ref{tildeH-section} introduces an asymptotically equivalent expression of the quasi-likelihood function. 
This expression is useful for deducing asymptotic properties of the quasi-likelihood function in Section~\ref{Hn-limit-section}.
We also need some results on identifiability of the model {\colorg to obtain consistency of the maximum-likelihood-type estimator}. 
These are discussed in Section~\ref{identifiability-section}.
Section~\ref{mixed-normality-section} shows asymptotic mixed normality of our estimator.
The LAN property of the model for constant diffusion coefficients is obtained in Section~\ref{LAN-section}.
Section~\ref{PLD-proof-section} contains a proof of results regarding the Bayes-type estimator and the convergence of moments of estimators.

\section{Main results}\label{results-section}

\subsection{Settings and construction of the estimator}\label{setting-subsection}

Let $(\Omega^{(0)}, \mathcal{F}^{(0)}, P^{(0)})$ be a probability space with a filtration ${\bf F}^{(0)}=\{\mathcal{F}_t^{(0)}\}_{0\leq t\leq T}$.
We consider a two-dimensional ${\bf F}^{(0)}$-adapted process $Y=\{Y_t\}_{0\leq t\leq T}$ satisfying the stochastic integral equation:
\begin{equation}\label{SRM}
Y_t=Y_0+\int^t_0\mu_sds+\int^t_0b(s,X_s,\sigma_{\ast})dW_s, \quad t\in[0,T],
\end{equation}
where $\{W_t\}_{0\leq t\leq T}$ is a $d_1$-dimensional standard ${\bf F}^{(0)}$-Wiener process, $b=(b^{ij})_{1\leq i\leq 2,1\leq j\leq d_1}$ is a Borel function,
$\mu=\{\mu_t\}_{0\leq t\leq T}$ is a locally bounded ${\bf F}^{(0)}$-adapted process with values in $\mathbb{R}^2$,
and $X=\{X_t\}_{0\leq t\leq T}$ is a continuous ${\bf F}^{(0)}$-adapted processes with values in $O$,
an open subset of $\mathbb{R}^{d_2}$ with $d_2\in\mathbb{N}$.
We consider market microstructure noise $\{\epsilon^{n,k}_i\}_{n\in\mathbb{N},i\in\mathbb{Z}_+,k=1,2}$ as an independent sequence of random variables on another probability space $(\Omega^{(1)},\mathcal{F}^{(1)},P^{(1)})$.
We assume that $\mathcal{F}^{(1)}=\mathfrak{B}((\epsilon^{n,k}_i)_{n,k,i})$ and that the distribution of $\epsilon^{n,k}_j$ does not depend on $j$,
{\colord where $\mathfrak{B}(S)$ denotes the minimal $\sigma$-field such that any element of $S$ is $\mathfrak{B}(S)$-measurable for a set $S$ of random variables.
We use the same notation $\mathfrak{B}(S)$ for a similarly defined $\sigma$-field for a set $S$ of measurable sets.} 
\begin{discuss}
{\colorr identical distributionであることはLemma~\ref{ZSZ-est}で使う. どちらにしろ二次モーメントは一定である必要がある.}
\end{discuss}
We consider a product probability space $(\Omega, \mathcal{F},P)$, where $\Omega=\Omega^{(0)}\times \Omega^{(1)}$, $\mathcal{F}=\mathcal{F}^{(0)}\otimes \mathcal{F}^{(1)}$, and $P=P^{(0)}\otimes P^{(1)}$.

We assume that the observations of processes occur in a nonsynchronous manner and are contaminated by market microstructure noise, that is,
we observe the vectors $\{\tilde{Y}^k_i\}_{0\leq i\leq {\bf J}_{k,n},k=1,2}$ and $\{\tilde{X}^k_j\}_{0\leq j\leq {\bf J}'_{k,n},1\leq k\leq d_2}$, 
where $\{S^{n,k}_i\}_{i=0}^{{\bf J}_{k,n}}$ and $\{T^{n,k}_j\}_{j=0}^{{\bf J}'_{k,n}}$ are random times in $(\Omega^{(0)},\mathcal{F}^{(0)})$,
$\{\eta^{n,k}_j\}_{j\in\mathbb{Z}_+,1\leq k\leq d_2}$ is a random sequence on $(\Omega, \mathcal{F})$, and
\begin{equation}
\tilde{Y}^k_i=Y^k_{S^{n,k}_i}+\epsilon^{n,k}_i, \quad \tilde{X}^k_j=X^k_{T^{n,k}_j}+\eta^{n,k}_j. 
\end{equation}
Our goal is to estimate the true value $\sigma_{\ast}$ of the parameter from nonsynchronous, noisy observations \\
$\{S^{n,k}_i\}_{0\leq i\leq {\bf J}_{k,n},k=1,2}$, 
$\{T^{n,k}_j\}_{0\leq j\leq {\bf J}'_{k,n},1\leq k\leq d_2}$, $\{\tilde{Y}^k_i\}_{0\leq i\leq {\bf J}_{k,n},k=1,2}$,
and $\{\tilde{X}^k_j\}_{0\leq j\leq {\bf J}'_{k,n},1\leq k\leq d_2}$.

\begin{discuss}
{\colorr
$\{\eta^{n,k}_i\}$をプロセスや$\epsilon$と独立に定義したい場合も$(\Omega^{(0)},\mathcal{F}^{(0)})$に含めればいいので問題ない.
Stable convergenceで任意の$W$と直交するマルチンゲールとのクロスバリデーションが$0$になるために$\epsilon$はプロセスやサンプリングのランダムタイムと別の空間に定義する必要がある.

Jacod型で$\{\epsilon_t\}_t$をdefすると, $\epsilon_{S^{n,k}_i}$が可測かどうかが分からない. (恐らく一般には可測にならないか）
一方で違う$n$に対し, $\epsilon^{n,k}_i$が独立でないと, Stable Convergenceで必要な$n$に依らないフィルとレーションをうまく作れず,
$f(\epsilon_{t_1},\cdots, \epsilon_{t_q})$と書いた時の$\mathcal{G}_t$-conditional expectationがどうなるかわからない.
}
\end{discuss}

By setting $d_2=2$, $X_t\equiv Y_t$, $\mu_t=\mu(t,Y_t)$, $S^{n,k}_i\equiv T^{n,k}_j$, and $\eta^{n,k}_j\equiv \epsilon^{n,k}_i$, 
our model contains the case where the latent process $Y$ is a diffusion process satisfying a stochastic differential equation
\begin{equation}\label{SDE}
dY_t=\mu(t,Y_t)dt+b(t,Y_t,\sigma_{\ast})dW_t, \quad t\in [0,T],
\end{equation}
and $Y$ is observed in a nonsynchronous manner with noise.
This model is of particular interest, but our results are also be applied to more general models (\ref{SRM}).

{\colord
\begin{remark}\label{SV-remark}
Stochastic volatility models are significant models for modeling stock prices. 
Unfortunately, our settings are not applied to hidden Markov models including stochastic volatility models 
because we require (possibly noisy) observations of process $X$.
However, we hope that our results give an essential idea to deal with noisy, nonsynchronous observations, 
and therefore we can construct an estimator for stochastic volatility models by replacing our quasi-likelihood function.
We have left it for future works.
\end{remark}
}

For a vector $x=(x_1,\cdots, x_k)$, we denote $\partial_x^l=(\frac{\partial^l}{\partial_{x_{i_1}}\cdots \partial_{x_{i_l}}})_{i_1,\cdots,i_l=1}^k$.
We assume the true value $\sigma_{\ast}$ of the parameter is contained in a bounded open set $\Lambda\subset \mathbb{R}^d$
that satisfies Sobolev's inequality; that is, for any $p>d$, there exists $C>0$ such that 
$\sup_{\sigma\in\Lambda}|u(x)|\leq C\sum_{k=0,1}(\int_{\Lambda}|\partial_x^ku(x)|_pdx)^{1/p}$ for any $u\in C^1(\Lambda)$.
This is the case when $\Lambda$ has a Lipshitz boundary. See Adams and Fournier~\cite{ada-fou03} for more details.

\begin{discuss}
{\colorr ソボレフの不等式を使うにはモーメント評価が示されている必要があるので漸近混合正規を示すときには使えない.
一方で, 漸近混合正規を示すときにも$\sup_{\sigma}$評価は必要. パラメータが一次元のopen intervalでない場合は$\sup_{\sigma}$評価は自明でない.}
\end{discuss}

Let {\colord $\Pi_n=(\{S^{n,k}_i\}_{n,k,i},\{T^{n,k}_j\}_{k,j})$} and $\{\mathcal{G}_t\}_{0\leq t\leq T}$ be a filtration of $(\Omega, \mathcal{F},P)$ given by
\begin{equation*}
\mathcal{G}_t=\mathcal{F}^{(0)}_t\bigvee \mathfrak{B}(\{\Pi_n\}_n) \bigvee \mathfrak{B}(A\cap \{S^{n,k}_i\leq t\} ; A\in \mathfrak{B}(\epsilon^{n,k}_i), m\in \mathbb{N}, k\in \{1,2\}, i\in\mathbb{Z}_+, n\in\mathbb{N}),
\end{equation*}
{\colord where $\mathcal{H}_1\bigvee \mathcal{H}_2$ denotes the minimal $\sigma$-field which contains $\sigma$-fields $\mathcal{H}_1$ and $\mathcal{H}_2$.}
We assume that $(X_t,Y_t,W_t,\mu_t)_t$ and $(\{S^{n,k}_i\}_{n,k,i},\{T^{n,k}_j\}_{n,k,j})$ are independent.
Moreover, we assume that there exist positive constants $v_{1,\ast}$ and $v_{2,\ast}$ such that $\eta^{n,k}_j1_{\{T^{n,k}_j\leq t\} }$ is $\mathcal{G}_t$-measurable,
\begin{equation*}
E[\epsilon^{n,k}_i1_{\{S^{n,k}_i>t\}}|\mathcal{G}_t]=0, \quad E[\epsilon^{n,k}_i\epsilon^{n,k'}_{i'}1_{\{S^{n,k}_i\wedge S^{n,k'}_{i'}>t\}}|\mathcal{G}_t]=v_{k,\ast}\delta_{ii'}\delta_{kk'}
\end{equation*}
for any $n,k,k',i,i',j,t$, where $\delta_{ij}$ is Kronecker's delta and $E_{\Pi}[{\bf X}]=E[X|\{\Pi_n\}_n]$ for a random variable ${\bf X}$.
{\colord We also assume that the distribution of $Y_0$ does not depend on $\sigma_{\ast}$, $v_{1,\ast}$, nor $v_{2,\ast}$.}
\begin{discuss}
{\colorr おそらくこの書き方をしておけば$\epsilon^{n,k}_i$の独立性ははずせて, 多少の$\epsilon$間の依存関係は許容できるだろうが, それをチェックするのは簡単ではないので保留しておく.
一番の問題は, $\tilde{Z}_{2,m}^{\top}{\bf S}\tilde{Z}_{2,m}$のモーメント評価をしているところとstable convergenceのところ.

$b$の中身は$X_t$ではなく, $(t,X_t)$に拡張する. そうすると, $[0,T]$の境界での微分可能性に
煩わされなくてすむし, $t$と$X_t$は要求される微分の階数がおそらく異なるから.}
\end{discuss}

Now we construct the quasi-likelihood function. 
We apply the idea of Gloter and Jacod~\cite{glo-jac01b} to our {\colorg construction} of a quasi-likelihood function; 
that is, we divide the whole observation interval $[0,T]$ into {\colorg equidistant} subdivisions and construct quasi-likelihood functions for each interval as follows.
Let $\{b_n\}_{n\in\mathbb{N}}$ and $\{k_n\}_{n\in\mathbb{N}}$ be sequences of positive numbers satisfying
$b_n\geq 1$, $k_n\leq b_n$, $b_n\to\infty$, $k_nb_n^{-1/2-\epsilon}\to \infty$, and {\colord $k_nb_n^{-2/3+\epsilon}\to 0$} as $n\to\infty$ for some $\epsilon>0$.
\begin{discuss}
{\colorr
$k_nb_n^{-1/2-\epsilon}\to 0$はLemma~\ref{Hn-lim-lemma2}での$\delta$の選択のために必要. $k_nb_n^{-9/13+\epsilon}\to 0$はLemma~\ref{ZSZ-est}の改善で改良できるかも.
}
\end{discuss}
{\colorg We will assume in Condition $[A2]$ a relation between $b_n$ and our sampling scheme, which implies that $b_n$ represents the order of observation frequency.}
Let $\ell_n=[b_nk_n^{-1}]$, $s_0=0$, $s_m=T[b_nk_n^{-1}]^{-1}m$, $b^k(t,x,\sigma)=(b^{kj}(t,x,\sigma))_{j=1}^{d_1}$, $K^k_0=-1$, and $K^k_m=\#\{i\in\mathbb{N};S^{n,k}_i<s_m\}$ for $k\in\{1,2\}$ and $1\leq m\leq \ell_n$.
Moreover, let $k^j_m=K^j_m-K^j_{m-1}-1$, $\bar{k}_n=\max_{m,j}k^j_m$, $\underbar{k}_n=\min_{m,j}k^j_m$, $J^k_m=\max\{1\leq j\leq {\bf J}'_{k,n};T^{n,k}_j\leq s_{m-1}\}$,
$I^k_{i,m}=[S^{n,k}_{i+K^k_{m-1}},S^{n,k}_{i+1+K^k_{m-1}})$, $\tilde{Y}^k(I^k_{i,m})=\tilde{Y}^k_{i+1+K^k_{m-1}}-\tilde{Y}^k_{i+K^k_{m-1}}$,
$\hat{X}_m=(\#\{j;T^{n,k}_j\in [s_{m-1},s_m)\}^{-1}\sum_{j;T^{n,k}_j\in [s_{m-1},s_m)}\tilde{X}^k_j)_{1\leq k\leq d_2}$, and
$b^j_m(\sigma)=b^j(s_{m-1},\hat{X}_m,\sigma)$ for $1\leq m\leq \ell_n$, $j\in\{1,2\}$ and $1\leq i\leq k^j_m$.
Then we have the following approximations of conditional covariance of observations:
\begin{eqnarray}\label{covariance-approximate}
E[\tilde{Y}^k(I^k_{i,m})\tilde{Y}^k(I^k_{i',m})|\mathcal{G}_{s_{m-1}}] &\sim& (|b^k_m|^2|I^k_{i,m}|+2v_k)\delta_{ii'}-v_11_{\{|i-i'|=1\} }, \nonumber \\
E[\tilde{Y}^1(I^1_{i'',m})\tilde{Y}^2(I^2_{i''',m})|\mathcal{G}_{s_{m-1}}] &\sim& b^1_m\cdot b^2_m|I^1_{i'',m}\cap I^2_{i''',m}| 
\end{eqnarray}
for any intervals $I^k_{i,m}, I^k_{i',m}, I^1_{i'',m},I^2_{i''',m}$.

Let ${\top}$ denotes the transpose operator for matrices (and vectors), $M(l)=\{2\delta_{i_1,i_2}-\delta_{|i_1-i_2|=1}\}_{i_1,i_2=1}^l$ for $l\in\mathbb{N}$,
$M_{j,m}=M(k^j_m)$ for $1\leq j\leq 2$.
Based on the relation (\ref{covariance-approximate}), we define a quasi-log-likelihood function $H_n(\sigma,v)$ by
\begin{equation}
H_n(\sigma,v)=-\frac{1}{2}\sum_{m=2}^{\ell_n}Z_m^{\top}S_m^{-1}(\sigma,v)Z_m-\frac{1}{2}\sum_{m=2}^{\ell_n}\log\det S_m(\sigma,v),
\end{equation}
where $Z_m=((\tilde{Y}^1(I^1_{i,m}))_{1\leq i\leq k^1_m}^{\top}, (\tilde{Y}^2(I^2_{i,m}))_{1\leq i\leq k^2_m}^{\top})^{\top}$ and 
\begin{equation}
S_m(\sigma,v)=\left(
\begin{array}{ll}
\{|b^1_m|^2|I^1_{i,m}|\delta_{ii'}\}_{ii'} & \{b^1_m\cdot b^2_m|I^1_{i,m}\cap I^2_{j,m}|\}_{ij} \\
\{b^1_m\cdot b^2_m|I^1_{i,m}\cap I^2_{j,m}|\}_{ji} & \{|b^2_m|^2|I^2_{j,m}|\delta_{jj'}\}_{jj'} \\
\end{array}
\right)+\left(
\begin{array}{ll}
v_1 M_{1,m} & 0 \\
0 & v_2 M_{2,m} 
\end{array}
\right).
\end{equation}

\begin{discuss}
{\colorr $m=1$を除いているのは, $b^l_m$の中の$\tilde{X}$がaverageされず不安定になっているから.}
\end{discuss}

\begin{remark}
Though such a local Gaussian quasi-log-likelihood function seems valid only when observation noise $\epsilon^{n,k}_i$ follows a Gaussian distribution,
asymptotic properties of the maximum likelihood estimator are robust enough to allow non-Gaussian noise.
We can use the same quasi-likelihood function for general noise.
\end{remark}

\begin{remark}
We used subdivisions of $[0,T]$ for the construction of $H_n$ because of technical issues related to deducing the limit of $H_n$. 
Since the diffusion coefficient $b$ in $S_m$ is fixed, matrix properties of $M_{j,m}$ introduced in Section~\ref{noise-cov-property-subsection} can be used to deduce the limit of $H_n$. 
On the other hand, such a construction of $H_n$ also contributes to reducing the calculation time of the maximum-likelihood-type estimator
because the size of $S_m$ is $O(k_n)$ while the size of the covariance matrix of all observations is $O(b_n)$.
\end{remark}

\begin{remark}
In~\cite{glo-jac01b}, $k_n$ is taken so that {\colord $n^{1/2}k_n^{-1}\to 0$} and $k_nn^{-3/4}\to 0$. 
Our rate {\colord $b_n^{2/3}$} for the upper bound of $k_n$ is a little bit worse because of some technical issue (for equidistance observations, we have $b_n\equiv n$). 
When we investigate asymptotic behaviors of the maximum-likelihood-type estimator, we deal with some supremum estimates for the $\sigma$ of quasi-likelihood ratios.
Unlike the one-dimensional settings of~\cite{glo-jac01b}, our multidimensional setting requires some properties to deal with the supremum.
We use Sobolev's inequality here for this purpose. Then we need an additional moment estimate for quasi-likelihood ratios,
which causes a worse rate of $k_n$. See the proofs of Lemmas~\ref{ZSZ-est} and~\ref{Hn-diff-lemma} for details.
\end{remark}

\begin{discuss}
{\colorr 
\begin{equation*}
P[k^k_m=0 \ {\rm or} \ -1]\leq P[r_n\geq k_nb_n^{-1}/3]=P[(b_n^{1-\epsilon}r_n)(b_n^{\epsilon}k_n^{-1})\geq 1/3]\to 0
\end{equation*}
より, $k^k_m=0$ or $-1$は無視できて問題ない. $H_n$のdefでも$0$としてdefされているし, 後ろの評価で問題にならないだろうからRemarkは必要ない.}
\end{discuss}

To construct the maximum-likelihood-type estimator $\hat{\sigma}_n$ for the parameter $\sigma$, 
we need estimators for the unknown noise variance $v_{\ast}=(v_{1,\ast},v_{2,\ast})$. We assume the following condition.

\begin{description}
\item{[$V$]} There exist estimators $\{\hat{v}_n\}_{n\in\mathbb{N}}$ of $v_{\ast}$ such that $\hat{v}_n\geq 0$ almost surely and $\{b_n^{1/2}(\hat{v}_n-v_{\ast})\}_{n\in\mathbb{N}}$ is tight.
\end{description}

\begin{discuss}
{\colorr $\hat{v}_n\geq 0$がないと$S_m(\sigma,\hat{v}_s)$が可逆かどうかわからない.}
\end{discuss}

For example, $\hat{v}_n=(\hat{v}_{n,k})_{k=1}^2$ with $\hat{v}_{n,k}=(2{\bf J}_{k,n})^{-1}\sum_i(\tilde{Y}^k_i-\tilde{Y}^k_{i-1})^2$ satisfies $[V]$
if $\{b_n{\bf J}_{k,n}^{-1}\}_n$ is tight for $k=1,2$, $\sup_{n,k,i}E[(\epsilon^{n,k}_i)^4]<\infty$ and $\sup_{n,k,i\neq j}b_n^2E[((\epsilon^{n,k}_i)^2-v_{k,\ast})(\epsilon^{n,k}_j)^2-v_{k,\ast})]<\infty$.
\begin{discuss}
{\colorr
\begin{eqnarray}
\hat{v}_{n,k}&=&\frac{1}{2{\bf J}_{k,n}}\sum_i(\tilde{Y}^k_i-\tilde{Y}^k_{i-1})^2 \nonumber \\
&=&\frac{1}{2{\bf J}_{k,n}}\sum_i\left\{(\epsilon^{n,k}_i-\epsilon^{n,k}_{i-1})^2-2(\epsilon^{n,k}_i-\epsilon^{n,k}_{i-1})(Y^k_{S^{n,k}_i}-Y^k_{S^{n,k}_{i-1}})+(Y^k_{S^{n,k}_i}-Y^k_{S^{n,k}_{i-1}})^2\right\} \nonumber \\
&=&\frac{1}{2{\bf J}_{k,n}}\sum_i(\epsilon^{n,k}_i-\epsilon^{n,k}_{i-1})^2+O_p\bigg({\bf J}_{k,n}^{-1}\bigg(\sum_i\Delta S^{n,k}_i\bigg)^{1/2}\bigg) + O_p\bigg({\bf J}_{k,n}^{-1}\sum_i\Delta S^{n,k}_i\bigg) \nonumber \\
&=&{\bf J}_{k,n}^{-1}\sum_i(\epsilon^{n,k}_i)^2-2^{-1}{\bf J}_{k,n}^{-1}((\epsilon^{n,k}_1)^2+(\epsilon^{n,k}_{{\bf J}_{k,n}})^2)-{\bf J}_{k,n}^{-1}\sum_i\epsilon^{n,k}\epsilon^{n,k}_{i-1}+O_p({\bf J}_{k,n}^{-1}). \nonumber
\end{eqnarray}
ゆえに
\begin{equation*}
b_n^{1/2}(\hat{v}_{n,k}-v_{k,\ast})=b_n^{1/2}{\bf J}_{k,n}^{-1}\sum_i((\epsilon^{n,k}_i)^2-v_{k,\ast})+O_p({\bf J}_{k,n}^{-1/2}b_n^{1/2}).
\end{equation*}
}
\end{discuss}

{\colorg Let ${\rm clos}(A)$ be the closure of a set $A$.}
A maximum-likelihood-type estimator $\hat{\sigma}_n$ is a random variable satisfying {\colorg $H_n(\hat{\sigma}_n,\hat{v}_n)=\max_{\sigma\in{\rm clos}(\Lambda)}H_n(\sigma,\hat{v}_n)$}. 
We study asymptotic mixed normality and asymptotic efficiency of the estimator in the following subsections.

\begin{remark}
We can also construct a simultaneous maximum-likelihood-type estimator $(\bar{\sigma}_n,\bar{v}_n)$ satisfying $H_n(\bar{\sigma}_n,\bar{v}_n)=\max_{\sigma,v}H_n(\sigma,v)$.
However, it is valid only when the observation noise $\epsilon^{n,k}_i$ follows a normal distribution.
Our interest is on estimating the parameter $\sigma$ of the latent process,
and so the assumptions for observation noise should be reduced as much as possible.
Therefore, the nonparametric estimator $\hat{v}_n$ is more suitable for our purpose.
\end{remark}

\subsection{Asymptotic mixed normality of the maximum-likelihood-type estimator}\label{mixed-normality-subsection}

In the rest of this section, we state our main theorems. Proofs of these results are left to Sections~\ref{tildeH-section}--\ref{PLD-proof-section}.
In this subsection, we describe the asymptotic mixed normality of the maximum-likelihood-type estimator $\hat{\sigma}_n$. 

We first describe assumptions for the theorem.  
Condition $[A1]$ is a sequence of assumptions on the latent processes $Y$ and $X$ and observation noise $\epsilon^{n,k}_i$ and $\eta^{n,k}_j$.
{\colorg We denote by $\mathcal{E}_l$ the unit matrix of size $l$.} 

\begin{description}
\item{[$A1$]} 1. For $0\leq 2i+j\leq {\colord 4}$ and $0\leq k\leq 4$, the derivatives $\partial^i_t\partial^j_x\partial^k_{\sigma}b(t,x,\sigma)$ exist on $[0,T]\times O\times \Lambda$ 
and have continuous extensions on $[0,T]\times O\times {\rm clos}(\Lambda)$.
\begin{discuss}
{\colorr 滑らかさの条件はLemma \ref{ZSZ-est} 2で使う. $\partial_t^2\partial_x,\partial_t\partial_x^2,\partial_x^4$があればいいか? 後ろを見てチェックする}
\end{discuss}
\\
2. $bb^{\top}(t,x,\sigma)$ is positive definite for $(t,x,\sigma)\in [0,T]\times O \times {\rm clos}(\Lambda)$.\\
3. $\sup_{n,k,i}E[(\epsilon^{n,k}_i)^q]<\infty$ for any $q>0$. 
\begin{discuss}
{\colorr $\Pi$と独立だから$E_{\Pi}$評価も問題ない. $\epsilon^{n,k}$をi.i.d.でなくす時はここの表現を変える.}
\end{discuss}
\\
4. $\mu_t$ is locally bounded (locally in time).\\
5. $\sup_n(\ell_n^{q/2}\max_{m,k}(\#\{j;T^{n,k}_j\in [s_{m-1},s_m)\}^{-1}E_{\Pi}[|\sum_{j;T^{n,k}_j\in [s_{m-1},s_m)}\eta^{n,k}_j|^q]))<\infty$ almost surely for any $q>0$. 
\begin{discuss}
{\colorr ソボレフの不等式を使うため, Lemma~\ref{Hn-diff-lemma}でこの$q$次モーメントを使う.}
\end{discuss}
\\
6. There exist progressively measurable processes $\{b^{(j)}_t\}_{0\leq t\leq T,0\leq j\leq 1}$ and $\{\hat{b}^{(j)}_t\}_{0\leq t\leq T, 0\leq j\leq 1}$ such that
{\colord $b^{(j)}_t$, $\hat{b}^{(j)}_t$, and $\sup_{u<s\leq t}((|b^{(j)}_s-b^{(j)}_u|\vee |\hat{b}^{(j)}_s-\hat{b}^{(j)}_u|)/|s-u|^{1/2})$ are locally bounded processes for $0\leq j\leq 1$}, and
\begin{equation*}
X_t=X_0+\int^t_0b^{(0)}_sds+\int^t_0b^{(1)}_sdW_s, \quad b^{(1)}_t=b^{(1)}_0+\int^t_0\hat{b}^{(0)}_sds+\int^t_0\hat{b}^{(1)}_sdW_s
\end{equation*}
for $t\in [0,T]$. 
\begin{discuss}
{\colorr 
モーメント条件から確率積分が定義できることもわかる.
$X_t$で$Y_t$と独立なブラウン運動の成分を定義できるように, $W$の次元を一般にしておく. $\tilde{W}$として定義すると, $S,T$との独立性などの議論がごちゃごちゃしやすい.
}
\end{discuss}
\end{description}
\begin{detailed}

Condition $[A1]$ captures somewhat standard assumptions and whether it holds can easily verified in practical settings.
Roughly speaking, point 5 of $[A1]$ is satisfied if the summation of $\eta^{n,k}_j$ is {\colord of an order equivalent to the square root of the number of $\eta^{n,k}_j$}.
This is satisfied under certain independency, martingale conditions or mixing conditions of $\eta^{n,k}_j$.
If $\{\eta^{n,k}_j\}_j$ is a sequence of independent and identically distributed values and the sequence has finite moments, 
then $E_{\Pi}[|\sum_{j;T^{n,k}_j\in [s_{m-1},s_m)}\eta^{n,k}_j|^q]=O_p(\#\{j;T^{n,k}_j\in [s_{m-1},s_m)\}^{-q/2})$. Then, point 5 of $[A1]$ is satisfied if sampling frequency of $\{T^{n,k}_j\}$ is of order $b_n$.
Decomposition of $X$ in point 6 of $[A1]$ is used to deduce asymptotically equivalent representation of $H_n$
where the diffusion coefficient $b(t,X_t,\sigma_{\ast})$ is replaced by $b(s_{m-1},X_{s_{m-1}},\sigma_{\ast})$.
Detailed semimartingale decomposition is required to estimate the difference $b(t,X_t,\sigma_{\ast})-b(s_{m-1},X_{s_{m-1}},\sigma_{\ast})$.

\end{detailed}
\begin{discuss}
{\colorr
Stable convergenceを示すときでもSobolevの不等式は必要なので, $\epsilon,\eta$のモーメント条件は必要.

localizationのためには$\bar{\Lambda}$まで全ての微分が連続拡張される必要がある.
微分が多項式増大を仮定すれば$\sqrt{x}$のような例はいずれにせよ除外され, $\partial_{\sigma}^4$が
境界付近で$\sin (1/x)$のような動きをするもの以外は差がないので問題ない.
}
\end{discuss}

In the following, we assume some conditions about our sampling scheme.
For $\eta\in (0,1/2)$, let $\mathcal{S}_{\eta}$ be the set of all sequences $\{[s'_{n,l},s''_{n,l})\}_{n\in\mathbb{N},1\leq l\leq L_n}$ of intervals on $[0,T]$ 
satisfying $\{L_n\}_n\subset\mathbb{N}$, $[s'_{n,l_1},s''_{n,l_2})\cap [s'_{n,l_2},s''_{n,l_2})=\emptyset$ for $n,l_1\neq l_2$, 
$\inf_{n,l}(b_n^{1-\eta}(s''_{n,l}-s'_{n,l}))>0$, and $\sup_{n,l}(b_n^{1-\eta}(s''_{n,l}-s'_{n,l}))<\infty$.
Let $r_n=\max_{i,k,m}|I^k_{i,m}|$ and $\underbar{r}_n=\min_{i,k,m}|I^k_{i,m}|$.
\begin{description}
\item{[A2]} There exist $\eta \in (0,1/2)$, {\colord $\dot{\eta}\in (0,1]$} and positive-valued functions $\{a_t^j\}_{t\in[0,T],j=1,2}$ such that 
{\colord $\sup_{t\neq s}(|a^j_t-a^j_s|/|t-s|^{\dot{\eta}})<\infty$ almost surely, $b_n^{-1/2}k_n(b_n^{-1}k_n)^{\dot{\eta}}\to 0$} and
\begin{eqnarray}\label{A2-conv}
k_nb_n^{-1/2}\max_{1\leq l\leq L_n}\bigg|b_n^{-1}(s''_{n,l}-s'_{n,l})^{-1}\#\{i;[S^{n,j}_{i-1},S^{n,j}_i)\subset (s'_{n,l},s''_{n,l})\}-a^j_{s'_{n,l}}\bigg|\to^p 0  
\end{eqnarray}
as $n\to\infty$ for $j=1,2$ and $\{[s'_{n,l},s''_{n,l})\}_{1\leq l\leq L_n,n\in\mathbb{N}}\in \mathcal{S}_{\eta}$.  
Moreover, $(r_nb_n^{1-\epsilon}) \vee (b_n^{-1-\epsilon}\underbar{r}_n^{-1})\to^p 0$ for any $\epsilon>0$.
\end{description}
In particular, Condition $[A2]$ implies $b_n^{-1}{\bf J}_{j,m}\to^p \int^T_0a^j_tdt$ and $\max_m|T^{-1}k_n^{-1}k^j_m-a^j_{s_{m-1}}|\to^p 0$ as $n\to\infty$.
Roughly speaking, $[A2]$ shows the law of large numbers for sampling schemes in any local time intervals.
In the proof of Lemma~\ref{Hn-lim-lemma2}, we will see that some properties of $M_{j,m}$ enable us to replace $|I^j_k|$ in $S_m$
by the local average in asymptotics. Then $[A2]$ leads to the limit of $H_n$.

\begin{discuss}
{\colorr

\begin{equation*}
b_n^{-1}{\bf J}_{j,n}=b_n^{-1}\sum_{k=1}^{[b_n^{1-\eta}]}\#\{i;[S^{n,j}_{i-1},S^{n,j}_i)\subset [s'_{n,l},s''_{n,l})\}=b_n^{-1}\sum_kb_na^j_{s'_{n,l}}\Delta s'_{n,l}(1+o_p(1))\to^p \int^T_0a^j_tdt.
\end{equation*}
また、$\{\tilde{s}^m_l\}$を$[s_{m-1},s_m)$の分割で$S_{\eta}$に入るものとすると,
\begin{equation*}
\max_m \bigg|\frac{b_n^{-1}k^j_m}{T[b_nk_n^{-1}]^{-1}}-a^j_{s_{m-1}}\bigg|
\leq \max_m \sum_l\frac{\tilde{s}^m_l-\tilde{s}^m_{l-1}}{T[b_nk_n^{-1}]^{-1}}\bigg|b_n^{-1}\frac{\sum_{S^{n,j}_i\in[\tilde{s}^m_{l-1},\tilde{s}^m_l)}1}{\tilde{s}^m_l-\tilde{s}^m_{l-1}}-a^j_{s_{m-1}}\bigg|\to^p 0.
\end{equation*}

ある$\eta'\in (\eta,1/2)$があって
\begin{eqnarray}
\{(b_nr_n)^3(b_n\underbar{r}_n)^{-7/2}(b_n^{\eta'}k_n^{-1})\}\vee\{r_n^3\underbar{r}_n^{-3/2}b_n^{3/2-\eta}\} 
\vee \{b_n^{\eta-\eta'}(b_nr_n)^3(b_n\underbar{r}_n)^{-3/2}\} \vee \{b_n^{-1+\eta+\eta'}(b_nr_n)^4(b_n\underbar{r}_n)^{-7/2}\} \to^p 0. \nonumber
\end{eqnarray}
があれば十分か.

$a_t^j$の連続性は, Proposition~\ref{Hn-lim}の証明で, $a_t^j$の値を$m$に関して一様に左端で近似するために使う.
$\max$で評価するのはきついか? Lem~\ref{D1-change}などで必要だが, 条件が弱められたら弱める.
最後の式はLemma~\ref{D1-change}で使う. $\{b_n^{-1}({\bf J}_{1,n}+{\bf J}_{2,n})\}_n$のtightnessもLemma~\ref{D1-change}で使う.
Lemma~\ref{D1-change}を示すためには, $\underbar{r}_n$は$m$毎の$\inf$ではなく, $[0,T]$全体でとる必要がある.
}
\end{discuss}

{\colord
\begin{example}
Let $\{N^k_t\}_{t\geq 0}$ be an exponential $\alpha$-mixing point process with stationary increments for $k=1,2$.
Set $S^{m,k}_i=\inf\{t\geq 0 ; N^k_{b_nt}\geq i\}$.
Then Rosenthal-type inequalities (Theorem 3 and Lemma 7 in Doukhan and Louhichi \cite{dou-lou99}, or Theorem 4 in \cite{ogi-yos14}) 
and a similar argument to the proof of Proposition 6 in \cite{ogi-yos14} ensure $[A2]$ with $a^j_t\equiv E[N^j_1]$ (constants).
\end{example}
}
\begin{discuss}
{\colorr
まず$E[N^j_t]=tE[N^j_1]$. (∵$t$が有理数ならOK. 右連続性から$\{t_n\}_n\subset \mathcal{Q}, t_n\to t$に対し, $E[N^j_{t_n}]=t_nE[N^j_1]$なのでOK)

よって,(\ref{A2-conv})の左辺の絶対値の中は
\begin{eqnarray}
b_n^{-1}(s''_{n,l}-s'_{n,l})^{-1}(N^j_{b_ns''_{n,l}}-N^j_{b_ns'_{n,l}})-E[N^j_1]
=b_n^{-1}(s''_{n,l}-s'_{n,l})^{-1}(N^j_{b_ns''_{n,l}}-N^j_{b_ns'_{n,l}}-E[N^j_{b_ns''_{n,l}}-N^j_{b_ns'_{n,l}}]) \nonumber
\end{eqnarray}
となるのでRosenthal-type inequalitiesから
\begin{eqnarray}
&&E_{\Pi}\bigg[\bigg(k_nb_n^{-1/2}\max_{1\leq l\leq L_n}\bigg|b_n^{-1}(s''_{n,l}-s'_{n,l})^{-1}\sum_{i;[S^{n,j}_{i-1},S^{n,j}_i)\subset (s'_{n,l},s''_{n,l})}1-a^j_{s'_{n,l}}\bigg|\bigg)^q\bigg] \nonumber \\
&\leq &k_n^qb_n^{-q/2}\sum_lE_{\Pi}\bigg[\bigg|\cdot \bigg|^q\bigg]
\leq k_n^qb_n^{-q/2}L_n\max_l(b_n^{-1}(s''_{n,l}-s'_{n,l})^{-1})^qE_{\Pi}\bigg[\bigg|\cdot \bigg|^q\bigg]. \nonumber
\end{eqnarray}
\begin{eqnarray}
E_{\Pi}\bigg[\bigg|\cdot \bigg|^q\bigg]
&\leq& C_q\bigg\{\bigg(\sum_{i=1}^{[b_n^{\eta}]}\int^1_0(\alpha^{-1}(u)\vee n)^{q-1}Q^q_{X'_i}(u)du\bigg)
\vee \bigg(\sum_{i=1}^{[b_n^{\eta}]}\int^1_0(\alpha^{-1}(u)\vee n)Q^2_{X'_i}(u)du\bigg)^{q/2}\bigg\} \nonumber \\
&\leq &C_q[b_n^{-\eta}]^{q/2}\int^1_0(\alpha^{-1}(u))^{q-1}Q^q_{X'_i}(u)du\leq C_qb_n^{\eta q/2}
\bigg(\int^1_0(\alpha^{-1}(u))^{2q-2}du\bigg)^{1/2}\bigg(\int^1_0Q^{2q}_{X'_i}(u)du\bigg)^{1/2}. \nonumber
\end{eqnarray}
ただし, $\alpha^{-1}(u)=\sum_{k=0}^{\infty}1_{\{\alpha_k\geq u\} }$ and $Q_{X'}(s)=\inf\{t\geq 0, P[|X'|>t]\leq s\}$.
$\int^1_0Q^{2q}_{X'_i}(u)du=E[|X'_i|^{2q}]$, $\int^1_0(\alpha^{-1}(u))^{q'}du\leq q'\sum_{k=0}^{\infty}(k+1)^{q'-1}\alpha_k^n$.
よって与式$=O((k_nb_n^{-1/2})^qb_n^{1-\eta}b_n^{-\eta q/2})\to 0$.
}
\end{discuss}

Under the above conditions, we can show convergence of the quasi-likelihood ratio $H_n(\sigma,\hat{v}_n)-H_n(\sigma_{\ast},\hat{v}_n)$.
The limit function is rather complicated, so we prepare some functions.
Let $b_t=b(t,X_t,\sigma)$, $b_{t,\ast}=b(t,X_t,\sigma_{\ast})$,
{\colorlg $\tilde{a}^j_t=a^j_t/v_{j,\ast}$} for $j=1,2$, {\colorlg $\varphi(x,y)=\sqrt{x+\sqrt{x^2-4y}}+\sqrt{x-\sqrt{x^2-4y}}$ for $0\leq 4y\leq x^2$, and
\begin{eqnarray}
\mathcal{Y}_1(\sigma)&=&\int^T_0\bigg\{
\frac{\sum_{j=1}^2(|b^j_t|^2-|b^j_{t,\ast}|^2)(|b^{3-j}_t|^2\sqrt{\tilde{a}^1_t\tilde{a}^2_t}+\tilde{a}^j_t\sqrt{\det(b_tb_t^{\top})})-2(b^1_t\cdot b^2_t-b^1_{t,\ast}\cdot b^2_{t,\ast})b^1_t\cdot b^2_t\sqrt{\tilde{a}_t^1\tilde{a}_t^2}}
{2\sqrt{2}\sqrt{\det(b_tb_t^{\top})}\varphi(\tilde{a}^1_t|b^1_t|^2+\tilde{a}^2_t|b^2_t|^2,\tilde{a}^1_t\tilde{a}^2_t\det(b_tb_t^{\top}))} \nonumber \\
&&\quad \quad -\frac{\varphi(\tilde{a}^1_t|b^1_t|^2+\tilde{a}^2_t|b^2_t|^2,\tilde{a}^1_t\tilde{a}^2_t\det(b_tb_t^{\top}))-\varphi(\tilde{a}^1_t|b^1_{t,\ast}|^2+\tilde{a}^2_t|b^2_{t,\ast}|^2,\tilde{a}^1_t\tilde{a}^2_t\det(b_{t,\ast}b_{t,\ast}^{\top}))}{2\sqrt{2}}\bigg\}dt. \nonumber
\end{eqnarray}
}
\begin{proposition}\label{Hn-lim}
Assume $[A1]$,$[A2]$ and $[V]$. Then $\sup_{\sigma\in\Lambda}|b_n^{-1/2}\partial_{\sigma}^k(H_n(\sigma,\hat{v}_n)-H_n(\sigma_{\ast},\hat{v}_n))-\partial_{\sigma}^k\mathcal{Y}_1(\sigma)|\to^p 0$ as $n\to \infty$ 
for $0\leq k\leq 3$.
\end{proposition}

To show consistency and asymptotic normality of $\hat{\sigma}_n$, the limit function $\mathcal{Y}_1(\sigma)$ of the quasi-likelihood ratio should have the unique maximum point at $\sigma=\sigma_{\ast}$.
More precisely, we use the following as a kind of identifiability condition: $\inf_{\sigma\neq \sigma_{\ast}}(-\mathcal{Y}_1(\sigma))/|\sigma-\sigma_{\ast}|^2>0$ almost surely.
Though it is difficult to directly check this condition in general,
we can {\colorg check} it under a more tractable sufficient condition. Let
\begin{equation*}
\mathcal{Y}_0(\sigma)=-\frac{1}{2}\int^T_0\bigg\{{\rm tr}((b_tb_t^{\top})^{-1}(b_{t,\ast}b_{t,\ast}^{\top})-\mathcal{E}_2)+\log\frac{\det (b_tb_t^{\top})}{\det (b_{t,\ast}b_{t,\ast}^{\top})}\bigg\}dt.
\end{equation*}
Then $\mathcal{Y}_0$ is the probability limit $n^{-1/2}(H^0_n(\sigma)-H^0_n(\sigma_{\ast}))$, where $H^0_n$ represents a quasi-likelihood function 
for a statistical model of equidistant observations without noise. 

\begin{description}
\item{[$A3$]} $\inf_{\sigma\neq \sigma_{\ast}}((-\mathcal{Y}_0(\sigma))/|\sigma-\sigma_{\ast}|^2)>0$ almost surely.
\end{description}
\begin{discuss}
{\colorr Gloter and Jacodのような条件で定義するのは今のところは難しい. 彼らの場合は１次元だからフィッシャーインフォメーションが退化していないかどうかを見るのに$bb^{\top}$の微分が$0$にならないことを見ればよかったが,
$2$次元以上ではこのようなシンプルな判定はできない. 彼らと同じ形の条件でもMLEの一致性は恐らくいえる.}
\end{discuss}

We will show in Proposition~\ref{separability-prop} that $[A3]$ is sufficient for the identifiability condition of our model.
Moreover, the following condition is a simple sufficient condition for $[A3]$
(see Remark 4 in Ogihara and Yoshida~\cite{ogi-yos14} for the details):
\begin{description}
\item{[$A3'$]} $\inf_{\sigma_1\neq \sigma_2}(|bb^{\top}(t,x,\sigma_1)-bb^{\top}(t,x,\sigma_2)|/|\sigma_1-\sigma_2|)>0$ for any $t\in[0,T]$ and $x\in O$.
\end{description}

We denote by $\to^{s\mathchar`-\mathcal{L}}$ the stable convergence of random variables. Let 
\begin{equation}\label{Gamma1-def}
{\colorlg \hat{\Gamma}_{1,n}=-b_n^{-1/2}\partial_{\sigma}^2H_n(\hat{\sigma}_n,\hat{v}_n), \quad} \Gamma_1=-\partial_{\sigma}^2\mathcal{Y}_1(\sigma_{\ast}).
\end{equation}
Let $\mathcal{N}$ be a $d$-dimensional random variable on some extension $(\tilde{\Omega},\tilde{\mathcal{F}},\tilde{P})$ of $(\Omega,\mathcal{F},P)$
satisfying the condition that $\mathcal{N}$ is independent of $\mathcal{F}$ and $\mathcal{N}$ follows the $d$-dimensional standard normal distribution.
We denote the expectation with respect to $\tilde{P}$ by the same notation $E$. 

The following theorem is one of our main results.
\begin{theorem}\label{main}
Assume $[A1]$--$[A3]$ and $[V]$. Then $\Gamma_1$ is positive definite almost surely and 
$b_n^{1/4}(\hat{\sigma}_n-\sigma_{\ast})\to^{s\mathchar`-\mathcal{L}} \Gamma_1^{-1/2}\mathcal{N}$ as $n\to \infty$.
{\colorlg Moreover, $\hat{\Gamma}_{1,n}\to^p \Gamma_1$, and therefore 
$b_n^{1/4}\hat{\Gamma}_{1,n}^{1/2}1_{\{\hat{\Gamma}_{1,n} {\rm is \ p.d.}\}}(\hat{\sigma}_n-\sigma_{\ast})\to^{s\mathchar`-\mathcal{L}} \mathcal{N}$ as $n\to\infty$.
}
\end{theorem}

\begin{corollary}
Assume $[A1]$, $[A2]$, $[A3']$ and $[V]$. Then the results in Theorem~\ref{main} hold true.
\end{corollary}

\subsection{On the LAMN property and asymptotic efficiency of the estimator}\label{LAMN-subsection}

In this subsection, we state some results on the so-called LAMN(LAN) property for our model and asymptotic efficiency of our estimator.
We also comment on some further studies.

Throughout this subsection, we assume that $X_t\equiv Y_t$, {\colorg $T^{n,k}_j\equiv S^{n,k}_i$}, 
{\colorg $\eta^{n,k}_j\equiv \epsilon^{n,k}_i$}, $\mu_t\equiv 0$ and $Y_0=\gamma$ for {\colord some known} $\gamma\in\mathbb{R}^2$.
Then the latent process $Y$ is a diffusion process satisfying the stochastic differential equation (\ref{SDE}) {\colorg with $\mu\equiv 0$}.
Let $P_{\sigma'_{\ast},v'_{\ast},n}$ be the distribution of $((S^{n,k}_i)_{k,i},(\tilde{Y}^k_i)_{k,i})$ with true values $(\sigma'_{\ast},v'_{\ast})$ of the parameters.
We denote 
\begin{equation*}
{\rm diag}(A,B)=\left(
\begin{array}{ll}
A & 0 \\
0 & B
\end{array}
\right)
\end{equation*} 
for square matrices $A$ and $B$. 
\begin{discuss}
{\colorr $\Gamma$のdef, Lemma~\ref{Sm-properties}の前のdefで使う. あとは${\bf K}_m$,$M_m$のdef}
\end{discuss}
Let $\mathcal{Y}_2(v)=-\int^T_0\sum_{j=1}^2a^j_t\{(v_{j,\ast}/v_j)-1+\log(v_j/v_{j,\ast})\}dt/2$, 
\begin{equation}\label{Gamma-def}
\Gamma_2=-\partial_v^2\mathcal{Y}_2(v_{\ast}) \quad {\rm and} \quad \Gamma={\rm diag}(\Gamma_1,\Gamma_2).
\end{equation}

{\colord We adopt the following definition of the LAMN property from Jeganathan \cite{jeg83}.}
\begin{definition}\label{LAMN-def}
Let $P_{\theta,n}$ be a probability measure on some measurable space $(\mathcal{X}_n,\mathcal{A}_n)$ for each $\theta\in{\bf \Theta}$ and $n\in\mathbb{N}$,
where $\Theta$ is a bounded open subset of $\mathbb{R}^d$.
Then the family $\{P_{\theta,n}\}_{\theta,n}$ satisfies the local asymptotic mixed normality (LAMN) property at $\theta=\theta_{\ast}$
if there exist a sequence $\{\delta_n\}_{n\in\mathbb{N}}$ of $d\times d$ positive definite matrices, 
$d\times d$ symmetric random matrices $\Gamma_n,\Gamma$ and $d$-dimensional random vectors $\mathcal{N}_n,\mathcal{N}$
such that $\Gamma$ is positive definite a.s.,  $P_{\theta_{\ast},n}[\Gamma_n \ {\rm is \ positive \ definite} ]=1 \ (n\in\mathbb{N})$, $\left\lVert \delta_n\right\rVert\to 0$, {\colorg and}
\begin{equation*}
\log \frac{dP_{\theta_{\ast}+\delta_n u,n}}{dP_{\theta_{\ast},n}}- \bigg(u^{\top}\sqrt{\Gamma_n}\mathcal{N}_n-\frac{1}{2}u^{\top}\Gamma_n u \bigg)\to 0
\end{equation*}
in $P_{\theta_{\ast},n}$-probability as $n\to \infty$ for any $u\in\mathbb{R}^d$. 
Moreover, $\mathcal{N}$ follows the $d$-dimensional standard normal distribution, $\mathcal{N}$ is independent of $\Gamma$ and
$\mathcal{L}(\mathcal{N}_n,\Gamma_n|P_{\theta_{\ast},n})\to \mathcal{L}(\mathcal{N},\Gamma)$ as $n\to\infty$.

If further the limit matrix $\Gamma$ is non-random, we say $\{P_{\theta,n}\}_{\theta,n}$ has the local asymptotic normality (LAN) property.
\end{definition}

{\colord To prove the LAMN property of our model, we assume the following additional condition.}

\begin{description}
\item{[$A1''$]} $[A1]$ is satisfied, $\mu_t\equiv 0$, $b(t,x,\sigma)$ does not depend on $(t,x)$
and $\epsilon^{n,k}_i$ follows a normal distribution for any $n,k,i$.
\end{description}
\begin{discuss}
{\colorr
$\mu$の条件($|\mu(t,x,\sigma)-\mu(t,y,\sigma)|\leq K|x-y|$とlocally bddくらいか)
$\mu_t=\mu(t,X_t,\sigma_{\ast})$は$P_{\sigma,v}$のdefで必要だから早めに書くか.
}
\end{discuss}

\begin{theorem}\label{LAN-theorem}
Assume $[A1'']$, $[A2]$ and $[A3]$. Then the family of distributions $\{P_{\sigma_{\ast},v_{\ast},n}\}_{\sigma_{\ast},v_{\ast},n}$
has the LAN property with $\Gamma$ in (\ref{Gamma-def}) and $\delta_n={\rm diag}(b_n^{-1/4}\mathcal{E}_d,b_n^{-1/2}\mathcal{E}_2)$.
\end{theorem}

\begin{remark}
Jeganathan~\cite{jeg82} studied lower bounds of estimation errors for any estimator of parameters.
They showed a version of H\'ajek's convolution theorem and that the optimal asymptotic variance of errors {\colord for regular estimators} is $\Gamma^{-1}$, where $\Gamma$ is in Definition~\ref{LAMN-def}.
Therefore, Theorems~\ref{main} and~\ref{LAN-theorem} ensures that our estimator {\colorg $\hat{\sigma}_n$} of the parameter $\sigma$ is asymptotically efficient {\colord in this sense}
under the assumptions of both theorems.
\end{remark}

\begin{remark}
The assumptions of Theorem~\ref{LAN-theorem} are rather strong conditions. 
We are also interested in the LAMN property in more general settings.
In particular, we are interested in the case that $\mu_t=\mu(t,X_t)$ 
and $\mu$ and $b$ are general functions with suitable conditions.
However, we need further analysis using Malliavin calculus to deal with the LAMN property of general diffusion processes,
as seen in Gobet~\cite{gob01} and Ogihara~\cite{ogi14}.
To the best of author's knowledge, such a result has not been obtained even for models with noisy, synchronous observations.
We have left this for future works.
\end{remark}

\begin{discuss}
{\colorr 
$\Gamma_n,\mathcal{N}_n$は$\mu$によらず, 確率収束は絶対連続なmeasureの交換で保たれる.
また, proof of Theorem \ref{main}の赤字の議論からstable convergenceも保たれるのでよい.

もしGirsanovでドリフトを消すときはLangevinの時と似たような問題が出る.
$\mu_t$がdeterministicで$x$によらないなら
\begin{equation*}
\mu(\sigma)=\left(
\begin{array}{l}
(\int_{I^1_i}\mu^1_t(\sigma)dt)_i \\
(\int_{I^2_j}\mu^2_t(\sigma)dt)_j
\end{array}
\right), \quad 
((\Delta Y^1_i)_i,(\Delta Y^2_j)_j)\sim N(\mu(\sigma_{\ast}),S_m)
\end{equation*}
対数尤度比は
\begin{eqnarray}
&&-\frac{1}{2}\{(Z-\mu(\sigma))^{\top}S(\sigma,v)^{-1}(Z-\mu(\sigma))-(Z-\mu_{\ast})^{\top}S_{\ast}^{-1}(Z-\mu_{\ast})\}-\frac{1}{2}\log\frac{\det S}{\det S_{\ast}} \nonumber \\
&=&-\frac{1}{2}\{(Z-\mu(\sigma))^{\top}(S(\sigma,v)-S_{\ast})^{-1}(Z-\mu(\sigma))+(2Z-\mu(\sigma)-\mu_{\ast})S_{\ast}^{-1}(\mu(\sigma)-\mu_{\ast})\}-\frac{1}{2}\log\frac{\det S}{\det S_{\ast}}. \nonumber 
\end{eqnarray}
$\mu_t\equiv 0$の時との差は
\begin{eqnarray}
&&-\frac{1}{2}\{\mu(\sigma)^{\top}(S(\sigma,v)-S_{\ast})^{-1}(-2Z+\mu(\sigma))+(2Z-\mu(\sigma)-\mu_{\ast})S_{\ast}^{-1}(\mu(\sigma)-\mu_{\ast})\} 
\end{eqnarray}
確率変数は$Z$だけなので, マルチんげーるパートとそれ以外の項に分けて評価すればできそうだが,
評価式が複雑になりそうだし, 一般のドリフトならOUが入るがそうでないのであまり得られるものがかわらないのでやめておくか.
Remarkで$\mu_t=\mu(t,\sigma)$ならできると書いてもよいが.

}
\end{discuss}

\subsection{A Bayes-type estimator and covergence of moments of estimation errors}\label{PLD-section}

Polynomial-type large deviation theory by Yoshida~\cite{yos06,yos11} enables us to address the asymptotic properties of a Bayes-type estimator
and the convergence of moments of estimation errors, which is a stronger result than asymptotic mixed normality.
Convergence of moments is useful when we investigate the theory of information criteria, minimax inequality and asymptotic expansion of estimators.
See Uchida~\cite{uch10} for a theory of contrast-based information criteria for ergodic diffusion processes with equidistant observations.
{\colord We also see asymptotic efficiency of our estimator in the sense of minimax inequality.}

We first assume following stronger conditions than $[A1]$--$[A3]$ and $[V]$.

\begin{description}
\item{[$B1$]}
\begin{enumerate}
\item $[A1]$ holds true with $O=\mathbb{R}^{d_2}$. 
\item There exists a positive constant $C$ such that
$\sup_{t\in [0,T],\sigma\in\Lambda}|\partial_t^i\partial_x^j\partial_{\sigma}^kb(t,x,\sigma)|\leq C(1+|x|)^C$
for $0\leq 2i+j\leq 4$, $0\leq k\leq 4$ and $x\in\mathbb{R}^{d_2}$.
\item $\inf_{t,x,\sigma}\det bb^{\top}(t,x,\sigma)>0$.
\item $E[|Y_0|^q]<\infty$ for any $q>0$.
\item $\sup_tE[|\mu_t|^q]<\infty$, $\sup_{s<t}(E[|\mu_t-\mu_s|^q]^{1/q}(t-s)^{-1/2})<\infty$ and $\sup_{s<t}E[(E[\mu_t-\mu_s|\mathcal{G}_s]/(t-s))^q]<\infty$.
\item For any $q>0$, $\max_j\sup_tE[|b^{(j)}_t|^q\vee |\hat{b}^{(j)}_t|^q]<\infty$ and $\max_j\sup_{s<t}(E[|b^{(j)}_t-b^{(j)}_s|^q \vee |\hat{b}^{(j)}_t-\hat{b}^{(j)}_s|^q]^{1/q}(t-s)^{1/2})<\infty$
for any $q>0$.

\end{enumerate}
\item{[$B2$]} There exist $\eta\in(0,1/2)$, {\colord $\dot{\eta}\in (0,1]$}, $\delta>0$ and positive-valued functions $\{a^j_t\}_{t\in [0,T],j=1,2}$ such that 
{\colord $b_n^{-1/2}k_n(b_n^{-1}k_n)^{\dot{\eta}}\to 0$ as $n\to\infty$, $E[\sup_{j,t>s}(|a^j_t-a^j_s|^q|t-s|^{-q\dot{\eta}})]<\infty$}, $E[\sup_{j,t}|a^j_t|^q] \vee E[\sup_{j,t}(|a^j_t|^{-q})]<\infty$, and
\begin{equation*}
\sup_n\sup_{\{[s'_{n,l},s''_{n,l})\}\in \mathcal{S}_{\eta}}E\bigg[\bigg(k_nb_n^{-1/2+\delta}\max_{1\leq l\leq L_n}\bigg|b_n^{-1}(s''_{n,l}-s'_{n,l})^{-1}\#\{i;[S^{n,j}_{i-1},S^{n,j}_i)\subset (s'_{n,l},s''_{n,l})\}-a^j_{s'_{n,l}}\bigg|\bigg)^q\bigg]<\infty
\end{equation*}
for any $q>0$. Moreover, there exists a positive constant $\gamma$ such that $k_nb_n^{-4/7+\gamma}\to 0$ and 
$E[((r_nb_n^{1-\epsilon}) \vee (\underbar{r}_n^{-1}b_n^{-1-\epsilon}))^q]\to 0$ as $n\to\infty$ for any $q>0$ and $\epsilon>0$.
\item{[$B3$]} For any $q>0$, there exists a positive constant $c_q$ such that $P[\inf_{\sigma\neq \sigma_{\ast}}((-\mathcal{Y}_0(\sigma))/|\sigma-\sigma_{\ast}|^2)\leq r^{-1}]\leq c_q/r^q$ for any $r>0$.
\item{[$B4$]} There exist estimators $\{\hat{v}_n\}_{n\in\mathbb{N}}$ of $v_{\ast}$ such that $\hat{v}_n> 0$ almost surely, {\colord $\limsup_nE[\hat{v}_n^{-q}]<\infty$,} and $\sup_nE[|b_n^{1/2}(\hat{v}_n-v_{\ast})|^q]<\infty$ for any $q>0$.
\end{description}
\begin{discuss}
{\colorr 例の$\hat{v}_n$では$\epsilon$で下からカットオフすればよい}
\end{discuss}
Though Condition $[B3]$ is rather difficult to check in a practical setting, Uchida and Yoshida~\cite{uch-yos13} investigated sufficient conditions for $[B3]$. 
The simplest condition is that $[B3]$ is satisfied if there exists $\epsilon>0$ such that
$|bb^{\top}(t,x,\sigma_1)-bb^{\top}(t,x,\sigma_2)|\geq \epsilon|\sigma_1-\sigma_2|$ for any $t\in [0,T]$, $x\in O$ and $\sigma_1,\sigma_2\in\Lambda$.
See Remark 4 in~\cite{ogi-yos14} for details.

Let $U_n=\{u\in \mathbb{R}^d;\sigma_{\ast}+b_n^{-1/4}u\in\Lambda\}$, $V_n(r)=\{|u|\geq r\}\cap U_n$, 
and ${\bf Z}_n(u)=\exp(H_n(\sigma_{\ast}+b_n^{-1/4}u,\hat{v}_n)-H_n(\sigma_{\ast},\hat{v}_n))$ for $u\in U_n$.
\begin{proposition}[Polynomial-type large deviation inequalities]\label{PLD-prop}
Assume $[B1]$--$[B4]$. Then for any $L>0$, there exists a positive constant $c_L$ such that
$P[\sup_{u\in V_n(r)}{\bf Z}_n(u)\geq e^{-r/2}]\leq c_L/r^L$
for any $n\in\mathbb{N}$ and $r>0$.
\end{proposition}

Since ${\bf Z}_n(0)=1$, Proposition~\ref{PLD-prop} immediately yields
\begin{equation}\label{PLD-ineq}
E[|b_n^{1/4}(\hat{\sigma}_n-\sigma_{\ast})|^p]=\int^{\infty}_0pt^{p-1}P[|b_n^{1/4}(\hat{\sigma}_n-\sigma_{\ast})|\geq t]dt
\leq \int^{\infty}_0pt^{p-1}P[\sup_{u\in V_n(t)}{\bf Z}_n(u)\geq e^{-t/2}]dt<\infty
\end{equation}
for any $p>0$. Moreover, we obtain the following convergence of moments of the estimation error.

\begin{theorem}\label{moment-conv-thm}
Assume $[B1]$--$[B4]$. Then $E[{\bf Y}f(b_n^{1/4}(\hat{\sigma}_n-\sigma_{\ast}))]\to E[{\bf Y}f(\Gamma_1^{-1/2}\mathcal{N})]$
as $n\to\infty$ for any bounded random variable ${\bf Y}$ on $(\Omega, \mathcal{F})$ and any continuous function $f$ of at most polynomial growth.
\end{theorem}

In particular, we obtain convergence of moments where $E[|b_n^{1/4}(\hat{\sigma}_n-\sigma_{\ast})|^q]\to E[|\Gamma_1^{-1/2}\mathcal{N}|^q]$ for any $q>0$.
This property is used when we study the theory of information criteria and asymptotic expansion of estimators.

We also obtain results for a Bayes type estimator.
Let a prior density $\pi:\Lambda \to (0,\infty)$ be a continuous function satisfying $0<\inf_{\sigma}\pi(\sigma)\leq \sup_{\sigma}\pi(\sigma)<\infty$.
Then a Bayes-type estimator $\tilde{\sigma}_n$ for the quadratic loss function is defined by
\begin{equation*}
\tilde{\sigma}_n=\bigg(\int_{\Lambda}\exp(H_n(\sigma))\pi(\sigma)d\sigma\bigg)^{-1}\int_{\Lambda}\sigma\exp(H_n(\sigma))\pi(\sigma)d\sigma.
\end{equation*}
Since the Bayes-type estimator $\tilde{\sigma}_n$ contains integrals with respect to $\sigma$, we need to deal with tail behaviors of likelihood ratio $H_n(\sigma)-H_n(\sigma_{\ast})$. 
{\colorg Hence} Proposition~\ref{PLD-prop} is essential to deduce asymptotic properties of a Bayes-type estimator.
Since the Bayes-type estimator can be calculated using Markov-Chain Monte Carlo methods, 
it is often easier to calculate than the maximum-likelihood-type estimator.
For the Bayes-type estimator $\tilde{\sigma}_n$, we obtain similar results to the ones for the maximum-likelihood-type estimator.

\begin{theorem}\label{bayes-thm}
Assume $[B1]$--$[B4]$. Then $E[{\bf Y}f(b_n^{1/4}(\tilde{\sigma}_n-\sigma_{\ast}))]\to E[{\bf Y}f(\Gamma_1^{-1/2}\mathcal{N})]$
as $n\to\infty$ for any bounded random variable ${\bf Y}$ on $(\Omega, \mathcal{F})$ and any continuous function $f$ of at most polynomial growth.
\end{theorem}

\begin{remark}
If the assumptions of Theorem~\ref{LAN-theorem} are satisfied, asymptotic minimax theorem (Theorem 4 in~\cite{jeg83}) holds for our model, so
\begin{equation*}
\lim_{\alpha\to\infty}\liminf_{n\to\infty}\sup_{|u|\leq \alpha}E_{\sigma_{\ast}+b_n^{-1/4}u}[l(|b_n^{1/4}(V_n-\sigma_{\ast}-b_n^{-1/4}u)|)]\geq E[l(|\Gamma_1\mathcal{N}|)]
\end{equation*}
for any estimators $\{V_n\}_n$ of the parameter and any function $l:[0,\infty)\to [0,\infty)$ which is nondecreasing and $l(0)=0$,
where $E_{\sigma}$ denotes expectation with respect to $P_{\sigma,v_{\ast},n}$. 
Using Theorems~\ref{moment-conv-thm} and~\ref{bayes-thm} and a similar argument in Theorem 2.2 of Ogihara~\cite{ogi14},
we can see that $\hat{\sigma}_n$ and $\tilde{\sigma}_n$ attain the lower bound of the above inequality
{\colord for continuous $l$ of at most polynomial growth, if further $[B2]$ and uniform versions of $[B3]$ and $[B4]$} with respect to the true value $(\sigma_{\ast},v_{\ast})$ {\colord are satisfied}.
{\colord Hence our estimators are asymptotically efficient in this sense as well.}
\end{remark}
\begin{discuss}
{\colorr 
真値fixに対するモーメント収束は上で示されている. $P[\sup_t|Y_t^{(\sigma_u)}-Y_t^{(\sigma_{\ast})}|]\leq Cb_n^{-1/2}|u|$も成り立つ.
最後の$b_n^{1/4}(\partial_{\sigma}H_n(\sigma_u;\sigma_u)-H_n(\sigma_{\ast};\sigma_{\ast}))$の評価も
$\sum_i(M_i(\sigma)-M_{i,\ast})\int_{I_i}\Delta W_tdW_t$の形になるからBurkholderを使えばよさそう.
一様なPLDも$H_n$の極限への収束と分離性のパラメータに関する一様性が証明を見直すと示せるのでOK
}
\end{discuss}

\section{Simulation results}\label{simulation-section}

In this section, we examine some simulation results of our estimator.

First, we consider the case where the latent process $Y$ is a Brownian motion, that is, 
$Y$ satisfies the following stochastic differential equation:
\begin{equation*}
\left\{
\begin{array}{ll}
dY^1_t & = \sigma_{1,\ast} dW_t^1 \\
dY^2_t & = \sigma_{3,\ast} dW_t^1 + \sigma_{2,\ast} dW_t^2, \\
\end{array}
\right.
\end{equation*}
where $\sigma_{\ast}=(\sigma_{1,\ast},\sigma_{2,\ast},\sigma_{3,\ast})\in (\epsilon,R)\times (-R,R)\times (\epsilon,R)$ {\colorg for some $0<\epsilon<R$}.
Moreover, let $\{N^1_t\}_{0\leq t\leq T}$ and $\{N^2_t\}_{0\leq t\leq T}$ be two independent Poisson processes with parameters $\lambda_1$ and $\lambda_2$, respectively.
We give sampling times by $S^{n,j}_i=\inf\{N^j_{nt}\geq j\} \wedge T$ for $j=1,2$.
Let $\{\epsilon^{n,j}_i\}_{i\in\mathbb{Z}_+,j=1,2}$ be independent normal random variables with $E[\epsilon^{n,j}_i]=0$ and $E[(\epsilon^{n,j}_i)^2]=v_{j,\ast}$. 

Then we can see that this example satisfies $[A1'']$, $[A2]$ and $[A3']$. So the maximum-likelihood-type estimator $\hat{\sigma}_n$ is asymptotically mixed normal 
and asymptotically efficient with asymptotic variance $\Gamma_1^{-1}$.
For the estimator $\hat{v}_n$ of $v_{\ast}$ we first use a simple estimator $\hat{v}_n=(2{\bf J}_{k,n})^{-1}\sum_i(\tilde{Y}^k_i-\tilde{Y}^k_{i-1})^2$,
which means that our estimator is calculated by $\hat{\sigma}_n={\rm argmax}_{\sigma}H_n(\sigma,\hat{v}_n)$.
We also consider a plug-in estimator $\hat{v}'_{k,n}=(\hat{v}_{k,n}-|b^k(\hat{\sigma}_n)|^2T/(2{\bf J}_{k,n})) \vee 0$ of $v_{k,\ast}$,
and $\hat{\sigma}'_n={\rm argmax}_{\sigma}H_n(\sigma,\hat{v}'_n)$.
Let $\hat{\sigma}''_n={\rm argmax}_{\sigma}H_n(\sigma,v_{\ast})$. Then $\hat{\sigma}''_n$ cannot be calculated by observed data,
but we can use it for comparison.
{\colorg Though} these estimators have the same asymptotic variance, their performances for finite samples are different.
In particular, we cannot ignore the bias of $\hat{v}_n$ since $v$ is relatively small compared with $\sigma$ in practical data.

\begin{table}
\caption{Simulation results for estimators of parameters}
\label{table1}
\begin{center}
\footnotesize
\begin{tabular}{c|c|c|c|c|c|c|c|c|c|c|c}
 &  & \multicolumn{5}{|c|}{Results with $v_{\ast}=(0.001,0.001)$}  & \multicolumn{5}{|c}{Results with $v_{\ast}=(0.005,0.005)$} \\ \cline{3-12}
$n$ &  & $\sigma_1$ & $\sigma_2$ & $\sigma_3$ & $v_1$ & $v_2$ & $\sigma_1$ & $\sigma_2$ & $\sigma_3$ & $v_1$ & $v_2$ \\ \hline
$1000$ & $(\hat{\sigma}_n,\hat{v}_n)$ & 0.897 & 0.776 & 0.451 & 0.001504 & 0.001500 & 0.957 & 0.818 & 0.481 & 0.005515 & 0.005501 \\
 &  & (0.040) & (0.042) & (0.062) & (0.000079) & (0.000080) & (0.086) & (0.143) & (0.094) & (0.000293) & (0.000296) \\
 & $(\hat{\sigma}'_n,\hat{v}'_n)$ & 0.971 & 0.840 & 0.487 & 0.001100 & 0.001094 & 0.991 & 0.850 & 0.498 & 0.005053 & 0.005035 \\
 &  & (0.046) & (0.047) & (0.067) & (0.000075) & (0.000078) & (0.092) & (0.139) & (0.098) & (0.000298) & (0.000306) \\
 & $\hat{\sigma}''_n$ & 0.999 & 0.863 & 0.501 & - & - & 0.997 & 0.861 & 0.499 & - & - \\
 &  & (0.045) & (0.046) & (0.068) & - & - & (0.069) & (0.070) & (0.096) & - & - \\ \hline
$5000$ & $(\hat{\sigma}_n,\hat{v}_n)$ & 0.964 & 0.833 & 0.481 & 0.001099 & 0.001099 & 0.990 & 0.854 & 0.495 & 0.005095 & 0.005096 \\
 &  & (0.028) & (0.029) & (0.040) & (0.000026) & (0.000026) & (0.044) & (0.044) & (0.061) & (0.000121) & (0.000123) \\
 & $(\hat{\sigma}'_n,\hat{v}'_n)$ & 0.997 & 0.862 & 0.498 & 0.001006 & 0.001006 & 0.999 & 0.862 & 0.499 & 0.004996 & 0.004998 \\
 &  & (0.031) & (0.031) & (0.041) & (0.000027) & (0.000027) & (0.045) & (0.045) & (0.062) & (0.000123) & (0.000125) \\
 & $\hat{\sigma}''_n$ & 0.999 & 0.864 & 0.499 & - & - & 0.998 & 0.862 & 0.499 & - & - \\
 &  & (0.029) & (0.030) & (0.041) & - & - & (0.043) & (0.044) & (0.062) & - & - \\ \hline
 \multicolumn{2}{c|}{true values} & 1 & 0.866 & 0.5 & 0.001 & 0.001 & 1 & 0.866 & 0.5 & 0.005 & 0.005 \\  \hline
\end{tabular} 
\end{center}
\end{table}

Table~\ref{table1} shows results of $1000$ estimations. Each cell represents the average of estimators,
with sample standard deviations given in parentheses. 
We set the values of parameters as {\colorg $k_n=[n^{5/8}]$,}
$T=1$, $(\lambda_1,\lambda_2)=(1,1)$, $(\sigma_{1,\ast},\sigma_{2,\ast},\sigma_{3,\ast})=(1,\sqrt{1-0.5^2},0.5)$,
and {\colord consider two cases of the noise variances : $v_{\ast}=(0.001,0.001)$ for the left-hand side of the table and $v_{\ast}=(0.005,0.005)$ for the right-hand side.
In both cases, we}
can see that $\hat{v}_n$ has an upper bias for $n=1000$, and causes a lower bias of $\hat{\sigma}_n$
because $\hat{v}_n$ contains variance of the latent process, which is always positive.
These biases can be moderated by using {\colorg the} plug-in estimator.
For $n=5000$, the plug-in estimator $\hat{\sigma}'_n$ performs as well as $\hat{\sigma}''_n$.
{\colord In the case of $v_{\ast}=(0.005,0.005)$,} the biases of $\hat{v}_n$ and $\hat{v}'_n$ are relatively small, 
so the performance of $\hat{\sigma}_n$ and $\hat{\sigma}'_n$ are better.

We can also construct an estimator $\hat{\sigma}'_{1,n}\hat{\sigma}'_{3,n}T$ of the quadratic covariation $\langle Y^1,Y^2\rangle_T=\sigma_{1,\ast}\sigma_{3,\ast}T$.
We see that {\colorlg 
\begin{equation}\label{qcv-conv}
n^{1/4}(\hat{\sigma}'_{1,n}\hat{\sigma}'_{3,n}T-\langle Y^1,Y^2\rangle_T)\to^d N(0,\sigma_{3,\ast}^2(\Gamma_1^{-1})_{11}+2\sigma_{1,\ast}\sigma_{3,\ast}(\Gamma_1^{-1})_{13}+\sigma_{1,\ast}^2(\Gamma_1^{-1})_{33})
\end{equation}
as $n\to\infty$
by the delta method, and the estimator is asymptotically efficient} since we can reparameterize the model using $\sigma_{1,\ast}\sigma_{3,\ast}$.
We therefore compared the performance of the estimator (MLE) with existing estimators of the quadratic covariation.
We used the pre-averaged Hayashi--Yoshida estimator (PHY) and modulated realized covariance (MRC) by Christensen, Kinnebrock, and Podolskij~\cite{chr-etal10},
the local method of moments (LMM) by Bibinger et al.~\cite{bib-etal14},
{\colord and an estimator based on maximum likelihood estimator of a model of constant diffusion coefficients (QMLE) by A${\rm \ddot{\i}}$t-Sahalia, Fan, and Xiu~\cite{ait-etal10} for comparison.}
Except LMM these estimators can be calculated using the `cce' function in the `yuima' R package (http://r-forge.r-project.org/projects/yuima).
We used the default values of the `cce' function or values used in corresponding papers for parameters of estimators 
($\theta=0.15$ for PHY, $\theta=1$ for MRC$_1$, $J=30$, $h^{-1}=10$ for LMM).
{\colord Here we use {\it the oracle estimator} defined in \cite{bib-etal14} for LMM to avoid a complicated calculation.
For the modulated realized covariance, we also examine an estimator MRC$_2$ with $\theta=1/3$ which is used in Jacod et al. \cite{jac-etal09}.}
Table~\ref{table3} shows the results of $1000$ estimations. We used the same parameter values as above.
Then the true value of the quadratic covariation becomes $\langle Y^1,Y^2\rangle_T=0.5$.
For both cases of observation noise variance, we can see that sample standard deviations of our estimator are the best in large samples. 
{\colorlg The theoretical (asymptotic) minimum of standard deviations for all estimators is calculated as 
$n^{-1/4}(\sigma_{3,\ast}^2(\Gamma_1^{-1})_{11}+2\sigma_{1,\ast}\sigma_{3,\ast}(\Gamma_1^{-1})_{13}+\sigma_{1,\ast}^2(\Gamma_1^{-1})_{33})^{1/2}$.
Table~\ref{table3} also shows that the sample standard deviations of MLE are close to the minima in large samples.
}

\begin{table}
\caption{Comparison of estimators of $\langle Y^1,Y^2\rangle_T$}
\label{table3}
\begin{center}
\footnotesize
\begin{tabular}{c|c|c|c|c|c|c|c}
  \multicolumn{8}{c}{Resuls with $v_{\ast}=(0.001,0.001)$}   \\
  $n$ & MLE & PHY & MRC$_1$ & MRC$_2$ & QMLE & LMM & Theoretical minimum \\ \hline
$1000$ & 0.474 & 0.499 & 0.508 & 0.501 & 0.501 & 0.463 & \\
 & (0.073) & (0.121) & (0.182) & (0.110) & (0.095) & (0.082) & (0.066) \\ \hline
$5000$ & 0.496 & 0.497 & 0.504 & 0.499 & 0.498 & 0.497 & \\
 & (0.046) & (0.081) & (0.124) & (0.073) & (0.056) & (0.069) & (0.044) \\
 \multicolumn{8}{c}{} \\
  \multicolumn{8}{c}{Results with $v_{\ast}=(0.005,0.005)$} \\
$n$ & MLE & PHY & MRC$_1$ & MRC$_2$ & QMLE & LMM & Theoretical minimum \\ \hline
$1000$ & 0.496 & 0.497 & 0.508 & 0.5000 & 0.5000 & 0.518 & \\
 & (0.109) & (0.148) & (0.185) & (0.124) & (0.120) & (0.112) & (0.099) \\ \hline
$5000$ & 0.499 & 0.497 & 0.505 & 0.499 & 0.499 & 0.514 & \\
 & (0.069) & (0.098) & (0.126) & (0.083) & (0.079) & (0.083) & (0.066) \\
\end{tabular} 
\end{center}
\end{table}

{\colord
$ $

In the next, we consider the model with random diffusion coefficients and non-Gaussian noise. 
As mentioned in Remark \ref{SV-remark}, we cannot directly apply our results to stochastic volatility models.
Here we consider the Cox-Ingersoll-Ross (CIR) process derived in \cite{cox-etal85} as a latent process with random diffusion coefficients.
Let the latent process $Y$ satisfy
\begin{equation*}
dY_t=\left(
\begin{array}{l}
\alpha_1-\beta_1Y^1_t\\
\alpha_2-\beta_2Y^2_t
\end{array}
\right)dt+\left(
\begin{array}{ll}
\sigma_{1,\ast}\sqrt{Y^1_t} & 0 \\
\sigma_{3,\ast}\sqrt{Y^2_t} & \sigma_{2,\ast}\sqrt{Y^2_t}
\end{array}
\right)dW_t,
\end{equation*}
where $\sigma_{\ast}=(\sigma_{1,\ast},\sigma_{2,\ast},\sigma_{3,\ast})\in (\epsilon',R')\times (-R',R')\times (\epsilon',R')$.
We assume Conditions $2\alpha_1>\sigma_{1,\ast}^2$ and $2\alpha_2>\sigma_{2,\ast}^2+\sigma_{3,\ast}^2$ which ensure $Y^1_t>0$ and $Y^2_t>0$ for $t\in [0,T]$ almost surely.
Let $\{\epsilon^{n,j}_i\}_{i\in\mathbb{Z}}$ be i.i.d. random variables following a centered Gamma distribution with a shape parameter $k_j$ and a scale parameter $\theta_j$ for $j=1,2$. 
We define $\{N^j_t\}$, $\hat{v}_n$, $\hat{v}'_n$, $\hat{\sigma}_n$, and $\hat{\sigma}'_n$ similarly to the first example.
We set the values of parameters as {\colorg $k_n=[n^{5/8}]$,} $T=1$, $(\lambda_1,\lambda_2)=(1,1)$, $(\sigma_{1,\ast},\sigma_{2,\ast},\sigma_{3,\ast})=(1,\sqrt{1-0.5^2},0.5)$,
$(\alpha_1,\alpha_2,\beta_1,\beta_2)=(1,1,1,1)$, and $(k_1,k_2,\theta_1,\theta_2)=(2,2,\sqrt{0.0005},\sqrt{0.0005})$ which implies $v_{\ast}=(0.001,0.001)$.
Table~\ref{cir-table} shows averages and sample standard deviations of $T_n-\langle Y^1,Y^2\rangle_T$
for each estimator $T_n$ of the quadratic covariation $\langle Y^1,Y^2\rangle_T$ in $1000$ simulations.
$\langle Y^1,Y^2\rangle_T$ is random in this model since the diffusion coefficients are random.
So we use extra-{\colorg high-frequency} observations $\{Y^l_{k/100000}\}_{k=0}^{100000}$ of $Y$ to calculate the approximated true value of $\langle Y^1,Y^2\rangle_T$.
In this model, we have not obtained the LAMN property nor asymptotic efficiency of our estimator though we expect to obtain them.
However, we still see that our estimator achieves the best error variance in large samples.
}

\begin{table}
\caption{Estimation errors of estimators of $\langle Y^1,Y^2\rangle_T$ for the CIR process}
\label{cir-table}
\begin{center}
\footnotesize
\begin{tabular}{c|c|c|c|c|c|c}
  $n$ & MLE & PHY & MRC$_1$ & MRC$_2$ & QMLE & LMM \\ \hline
1000 & -0.0267 & -0.0063 & -0.0058 & -0.0036 & -0.0008 & -0.0348 \\
 & (0.0733) & (0.1286) & (0.1867) & (0.1162) & (0.1013) & (0.0844) \\ \hline
5000 & -0.0023 & -0.0036 & -0.0022 & -0.0016 & -0.0005 & -0.0033 \\
 & (0.0456) & (0.0858) & (0.1305) & (0.0768) & (0.0580) & (0.0719) \\
\end{tabular} 
\end{center}
\end{table}

\section{Asymptotically equivalent representation of the quasi-likelihood function}\label{tildeH-section}

We will prove our main results in the rest of this paper.
In this section, we introduce an asymptotically equivalent representation $\tilde{H}_n(\sigma,v)$ of the quasi-likelihood function $H_n(\sigma,v)$,
and prove the equivalence. $\tilde{H}_n$ is a useful function for deducing the limit of $H_n$. 

\subsection{Some notations}

We denote $E_m$ as the $\mathcal{G}_{s_{m-1}}$-conditional expectation and $\bar{E}_m[{\bf X}]={\bf X}-E_m[{\bf X}]$ for a random variable ${\bf X}$.
We use the symbol $C$ for a generic positive constant that can vary from line to line.

For a sequence {\colorg $c_n$} of positive-valued $\mathfrak{B}(\Pi_n)$-measurable random variables, 
let us denote by $\{\bar{R}_n(c_n)\}_{n\in\mathbb{N}}$, $\{\underbar{R}_n(c_n)\}_{n\in\mathbb{N}}$ and $\{\dot{R}_n(c_n)\}_{n\in\mathbb{N}}$
sequences of random variables (which may depend on {\colorg $1\leq m\leq \ell_n$} and $\sigma$) satisfying
\begin{equation*}
E[(c_n^{-1}(r_n/b_n)^{-p_1}(b_n/\underbar{r}_n)^{-p_2}(\bar{k}_n/k_n)^{-p_3}(k_n/\underbar{k}_n)^{-p_4}b_n^{-\delta}\sup_{\sigma,m}E_{\Pi}[|\bar{R}_n(c_n)|^q]^{1/q})^{q'}]\to 0,
\end{equation*}
\begin{equation*}
E[(c_n^{-1}(r_n/b_n)^{q_1}(b_n/\underbar{r}_n)^{q_2}(\bar{k}_n/k_n)^{q_3}(k_n/\underbar{k}_n)^{q_4}b_n^{\delta'}\sup_{\sigma,m}E_{\Pi}[|\underbar{R}_n(c_n)|^q]^{1/q})^{q'}]\to 0,
\end{equation*}
and
\begin{equation*}
c_n^{-1}(r_n/b_n)^{q_1}(b_n/\underbar{r}_n)^{q_2}(\bar{k}_n/k_n)^{q_3}(k_n/\underbar{k}_n)^{q_4}\sup_{\sigma,m}|\dot{R}_n(c_n)|\to^p0,
\end{equation*}
respectively, as $n\to \infty$ for any $\delta,q,q',q_1,\cdots,q_4>0$ with some constants $\delta',p_1,\cdots,p_4\geq 0$.

Let $M_m(v)={\rm diag}(v_1M_{1,m},v_2M_{2,m})$ for $v=(v_1,v_2)$, $\tilde{b}^k_m=b^k(s_{m-1},X_{s_{m-1}},\sigma)$, $\tilde{b}^k_{m,\ast}=b^k(s_{m-1},X_{s_{m-1}},\sigma_{\ast})$,
\begin{equation*}
\tilde{Z}_m=(((\tilde{b}^1_{m,\ast}\cdot (W_{S^{n,1}_i}-W_{S^{n,1}_{i-1}})+\epsilon^{n,1}_i-\epsilon^{n,1}_{i-1})_{i=K^1_{m-1}+2}^{K^1_m})^{\top}, ((\tilde{b}^2_{m,\ast}\cdot (W_{S^{n,2}_j}-W_{S^{n,2}_{j-1}})+\epsilon^{n,2}_j-\epsilon^{n,2}_{j-1})_{j=K^2_{m-1}+2}^{K^2_m})^{\top})^{\top},
\end{equation*}
\begin{equation}
\tilde{S}_m(\sigma,v)=\left(
\begin{array}{ll}
{\rm diag}((|\tilde{b}^1_m|^2|I^1_{i,m}|)_i) & \{\tilde{b}^1_m\cdot \tilde{b}^2_m|I^1_{i,m}\cap I^2_{j,m}|\}_{ij} \\
\{\tilde{b}^1_m\cdot \tilde{b}^2_m|I^1_{i,m}\cap I^2_{j,m}|\}_{ji} & {\rm diag}((|\tilde{b}^2_m|^2|I^2_{j,m}|)_j) \\
\end{array}
\right)+M_m(v),
\end{equation}
and
\begin{equation*}
\tilde{H}_n(\sigma,v)=-\frac{1}{2}\sum_{m=2}^{\ell_n}\tilde{Z}_m^{\top}\tilde{S}_m^{-1}(\sigma,v)\tilde{Z}_m-\frac{1}{2}\sum_{m=2}^{\ell_n}\log \det \tilde{S}_m(\sigma,v).
\end{equation*}

The diffusion coefficients $b$ in $\tilde{Z}_m$ and $\tilde{S}_m$ are either $b(s_{m-1},X_{s_{m-1}},\sigma)$ or $b(s_{m-1},X_{s_{m-1}},\sigma_{\ast})$.
Hence we do not need to consider the {\colorg time-dependent} structure of $b$ when we study asymptotics of the summands in $\tilde{H}_n$.
In particular, we obtain $E_m[\tilde{Z}_m^{\top}\partial_{\sigma}\tilde{S}_m(\sigma_{\ast},v_{\ast})^{-1}\tilde{Z}_m+\partial_{\sigma}\log \det \tilde{S}_m(\sigma_{\ast},v_{\ast})]=0$ by $\partial_{\sigma}\log\det \tilde{S}_m(\sigma,v)=-{\rm tr}(\partial_{\sigma}\tilde{S}_m\tilde{S}_m^{-1})(\sigma,v)$.
We will prove the asymptotic equivalence of $H_n$ and $\tilde{H}_n$ and then investigate asymptotic properties of $\tilde{H}_n$ instead of $H_n$.

Similarly to the approach of Gloter and Jacod~\cite{glo-jac01b}, we first show our results under the following condition $[A1']$, which is stronger than $[A1]$.
Then localization techniques and Girsanov's theorem enable us to replace $[A1']$ with $[A1]$.
\begin{description}
\item{[$A1'$]} Condition $[A1]$ is satisfied, $O=\mathbb{R}^{d_2}$, $\sup_{t,x,\sigma}\left\lVert (bb^{\top})^{-1}\right\rVert(t,x,\sigma) <\infty$, 
$\mu_t\equiv 0$ and $Y_0$, $\sup_t|b^{(l)}_t|$, $\partial_t^i\partial_x^j\partial_{\sigma}^kb$, and
$\sup_{t>s}((|b^{(l)}_t-b^{(l)}_s|\vee |\hat{b}^{(l)}_t-\hat{b}^{(l)}_s|)/(t-s))$
are all bounded for $l=0,1$, $0\leq 2i+j\leq 4$ and $0\leq k\leq 4$. 
\end{description}
We can also see that $[A1']$ implies $[B1]$.
\begin{discuss}
{\colorr
統計モデルを一度変えたり, $X_n-X_n^{(p)}=O_p(1)$となる$X_n^{(p)}$をdefしたりするので, 一般的なLemとして「$[A1']$である命題が成り立つなら$[A1]$の下成り立つ」
というようなことを示すのは難しい.
}
\end{discuss}

\subsection{Fundamental properties of the noise covariance matrix}\label{noise-cov-property-subsection}

In the following subsection we will show the asymptotic equivalence of $\tilde{H}_n$ and $H_n$, 
namely that  
\begin{equation*}
{\colorg b_n^{-1/2}\sup_{\sigma\in\Lambda}|\partial_{\sigma}^j(H_n(\sigma,\hat{v}_n)-H_n(\sigma_{\ast},\hat{v}_n))-\partial_{\sigma}^j(\tilde{H}_n(\sigma,v_{\ast})-\tilde{H}_n(\sigma_{\ast},v_{\ast}))| \to^p 0}
\end{equation*}
{\colorg as $n\to\infty$ for $0\leq j\leq 3$.}
To that end, we first show fundamental properties of $S_m$ and $\tilde{S}_m$.
These matrices inherit some properties of $M_{j,m}$, that are necessary to deduce the limit of $H_n$ and $\tilde{H}_n$.
The first property (\ref{tr-lim-est}) concerns the trace of a matrix related to $M_{j,m}$ investigated by~\cite{glo-jac01b}.
In the one-dimensional model with noisy, equidistance observations, this property can be directly applied to the quasi-likelihood function
because the covariance matrix of the latent process is the unit matrix.
However, this is insufficient for our purpose because our covariance matrix $S_m-M_m(v)$ of the latent process is rather complicated.
Therefore, we investigate further matrix properties related to $M_{j,m}$.

First, we consider the results in~\cite{glo-jac01b}. 
For any positive constants {\colorlg $p$, $q$, $a$ and $b$}, {\colord eigenvalues of $(a\mathcal{E}+M_{j,m})^{-1}$ are $\{(a+2(1-\cos (i\pi(k^j_m+1)^{-1}))\}_{i=1}^{k^j_m}$ and }we obtain
\begin{equation}\label{tr-lim-est}
\pi^{-1}k^j_mI_p(a)-a^{-p}\leq {\rm tr}((a\mathcal{E}+M_{j,m})^{-p})\leq \pi^{-1}k^j_mI_p(a),
\end{equation}
\begin{equation}\label{tr-lim-est2}
{\colorlg \pi^{-1}k^j_mI_{p,q}(a,b)-a^{-p}b^{-q}\leq {\rm tr}((a\mathcal{E}+M_{j,m})^{-p}(b\mathcal{E}+M_{j,m})^{-q})\leq \pi^{-1}k^j_mI_{p,q}(a,b)},
\end{equation}
where $I_p(a)=\int^{\pi}_0(a+2(1-\cos x))^{-p}dx$ 
{\colorlg and $I_{p,q}(a,b)=\int^{\pi}_0(a+2(1-\cos x))^{-p}(b+2(1-\cos x))^{-q}dx$}. 
Simple calculations show that $I_1(a)=\pi/\sqrt{a(4+a)}$, $I_2(a)=\pi(2+a)a^{-3/2}(4+a)^{-3/2}$
{\colord and $\int^{\pi}_0\{\log(a+2(1-\cos x))-\log(b+2(1-\cos x))\}dx=2\pi(\log(\sqrt{a}+\sqrt{4+a})-\log(\sqrt{b}+\sqrt{4+b}))$}. 
See Section 4.1 in~\cite{glo-jac01b} for the details.
{\colord Moreover, differentiation with respect to $a$ yields 
\begin{equation*}
I_p(a)=\frac{(-1)^{p-1}}{(p-1)!}\left(\frac{d}{da}\right)^{p-1}\left(\frac{\pi}{\sqrt{a(4+a)}}\right).
\end{equation*}
In particular, if $a={\bf X}_nb_n^{-1}$ for some tight random variables $\{{\bf X}_n\}_n$, then we have
\begin{equation*}
I_p({\bf X}_nb_n^{-1})=\frac{\pi(2p-3)!!}{2^p(p-1)!}({\bf X}_nb_n^{-1})^{-p+1/2}+O_p(b_n^{p-3/2}).
\end{equation*}
}

For $\epsilon\geq 0$, let $\{p_j(\epsilon)\}_{j\in\mathbb{N}}$ and $\{p'_j(\epsilon)\}_{j\in\mathbb{N}}$ be sequences of positive numbers satisfying $p_1(\epsilon)=2+\epsilon$, {\colorg $p'_1(\epsilon)=1+\epsilon$},
$p_{j+1}(\epsilon)=2+\epsilon-1/p_j(\epsilon)$, and $p'_{j+1}(\epsilon)=2+\epsilon-1/p'_j(\epsilon)$ for $j\in\mathbb{N}$.
Let $E_{i,j}(a)$ be a $k^j_m\times k^j_m$ matrix satisfying $(E_{i,j}(a))_{k,l}=\delta_{k,l}+a\delta_{(i,j)}(k,l)$ for $a\in\mathbb{R}$.
Then we have 
\begin{equation*}
E_{k^j_m,k^j_m-1}(p_{k^j_m-1}(\epsilon)^{-1})\cdots E_{2,1}(p_1(\epsilon)^{-1})(\epsilon \mathcal{E}+M_{j,m})E_{1,2}(p_1(\epsilon)^{-1})\cdots E_{k^j_m-1,k^j_m}(p_{k^j_m-1}(\epsilon)^{-1})
={\rm diag}((p_j(\epsilon))_{j=1}^{k^j_m}),
\end{equation*}
\begin{equation*}
(\epsilon \mathcal{E}+M_{j,m})^{-1}=E_{1,2}(p_1(\epsilon)^{-1})\cdots E_{k^j_m-1,k^j_m}(p_{k^j_m-1}(\epsilon)^{-1}){\rm diag}((p_j(\epsilon)^{-1})_{j=1}^{k^j_m})E_{k^j_m,k^j_m-1}(p_{k^j_m-1}(\epsilon)^{-1})\cdots E_{2,1}(p_1(\epsilon)^{-1}),
\end{equation*}
and hence 
\begin{equation}\label{dotD-eq}
(\epsilon \mathcal{E}+M_{j,m})^{-1}=\bigg\{\prod_{k+1\leq i\leq l}p_{i-1}(\epsilon)^{-1}1_{\{k\leq l\}}\bigg\}_{k,l}{\rm diag}((p_j(\epsilon)^{-1})_{j=1}^{k^j_m})
\bigg\{\prod_{l+1\leq i\leq k}p_{i-1}(\epsilon)^{-1}1_{\{l\leq k\}}\bigg\}_{k,l}.
\end{equation}
\begin{discuss}
{\colorr
$\ddot{D}$の逆行列の成分は, Ait-Sahalia, Mykland, Zhang 2005のP368にある. $\eta \sim -1-b_n^{-1/2}$.
\begin{eqnarray}
v^{ij}&=&(1-\eta^2)^{-1}(1-\eta^{2N+2})^{-1}\Big\{(-\eta)^{|i-j|}-(-\eta)^{i+j}-(-\eta)^{2N-i-j+2} \nonumber \\
&&-(-\eta)^{2N+|i-j|+2}+(-\eta)^{2N+i-j+2}+(-\eta)^{2N-i+j+2}\Big\}. \nonumber 
\end{eqnarray}
これはオーダーを見たり, averaging効果を見るには少し見づらい形か.
}
\end{discuss}
Moreover, we have the following lemma.
\begin{lemma}\label{ddotD-properties}
Let $\epsilon\in [0,1)$ and $p_+(\epsilon)=1+\epsilon/2+\sqrt{\epsilon+\epsilon^2/4}$. Then
\begin{enumerate}
\item $1\leq p'_j(\epsilon)\leq p_+(\epsilon)<p_j(\epsilon)\leq 1+1/j+j\epsilon$ for $j \in \mathbb{N}$, $\{p_j(\epsilon)\}_j$ is monotone decreasing,
and $\{p'_j(\epsilon)\}_j$ is monotone nondecreasing.
\item $\{((\epsilon\mathcal{E}+M_{j,m})^{-1})_{kk}\}_{k=1}^{[k^j_m/2]}$ is monotone increasing.
\item $p_j-p_+\leq (1+\sqrt{\epsilon})^{-(j-2)}$ and $p_+-p'_j\leq \sqrt{\epsilon}(1+\sqrt{\epsilon})^{-(j-2)}$ for $j\geq 2$.
\item $\prod_{j=1}^kp'_j(\epsilon)=(p_k(\epsilon)-1)\prod_{j=1}^{k-1}p_j(\epsilon)$ for any $k\geq 2$.
\end{enumerate}
\end{lemma}
\begin{proof}
{\it 1}. We simply denote $p_j=p_j(\epsilon)$. We will prove $p_+(\epsilon)<p_j(\epsilon)\leq 1+1/j+j\epsilon$ for $j \in \mathbb{N}$ by induction.
The results obviously hold for $j=1$. Assume the results hold for all values in $\mathbb{N}$ up to $j$. Then since $p_+=2+\epsilon-1/p_+$, 
we obtain $p_{j+1}-p_+=1/p_+-1/p_j>0$, and 
\begin{equation*}
p_{j+1}\leq 2+\epsilon-(j/(j+1))(1+j^2\epsilon/(j+1))^{-1}\leq 2+\epsilon-(j/(j+1))(1-j^2\epsilon/(j+1))\leq 1+1/(j+1)+(j+1)\epsilon.
\end{equation*}
Hence, we have $p_+(\epsilon)<p_j(\epsilon)\leq 1+1/j+j\epsilon$ for $j \in \mathbb{N}$.
Moreover, we can inductively deduce 
$p_{j+1}-p_j=1/p_{j-1}-1/p_j>0$.
The results for $\{p'_j(\epsilon)\}_j$ are obtained similarly.	

\noindent 
{\it 2}. By considering the cofactor matrix and (\ref{dotD-eq}), we have
\begin{equation}\label{invD-diagonal-est}
((\epsilon\mathcal{E}+M_{j,m})^{-1})_{kk}=\frac{\det (\epsilon\mathcal{E}_{k-1}+M(k-1))\det(\epsilon\mathcal{E}_{k^j_m-k}+M(k^j_m-k))}{\det (\epsilon\mathcal{E}+M_{j,m})}
=\frac{\prod_{l=1}^{k-1}p_l\prod_{l=1}^{k^j_m-k}p_l}{\prod_{l=1}^{k^j_m}p_l}.
\end{equation}  
Therefore we obtain the result by monotonicity of $p_j$.

\noindent
{\it 3}. This is easy since $p_j-p_+=(p_{j-1}-p_+)/p_+p_{j-1}\leq (p_1-p_+)/p_+^{j-1}\leq p_+^{-j+2}$.

\noindent
{\it 4}. 
\begin{equation*}
\prod_{j=1}^kp'_j(\epsilon)=\det(\epsilon\mathcal{E}+M(k)-(E_{11}(1)-\mathcal{E}))=\det(\epsilon\mathcal{E}+M(k)-(E_{kk}(1)-\mathcal{E}))=(p_k(\epsilon)-1)\prod_{j=1}^{k-1}p_j(\epsilon).
\end{equation*}

\end{proof}

\begin{discuss}
{\colorr ここの結果は一応(\ref{tr-lim-est})の評価をすぐ使うからこの位置に書いておく}
\end{discuss}

Let $\bar{\rho}=\sup_{t,\sigma}(|b^1\cdot b^2||b^1|^{-1}|b^2|^{-1})(t,X_t,\sigma)$, $\tilde{D}_m=(\tilde{D}_{1,m},\tilde{D}_{2,m})$, $\tilde{D}_{j,m}={\rm diag}((|\tilde{b}^j_m|^2|I^j_{i,m}|)_i)+v_{j,\ast}M_{j,m}$,
$D'_m=(D'_{1,m},D'_{2,m})$, $D'_{j,m}={\rm diag}((|I^j_{i,m}|)_i)$, and $\check{D}_{j,m}=|\tilde{b}^j_m|^2r_n\mathcal{E}+v_{j,\ast}M_{j,m}$.

\begin{lemma}\label{Sm-properties}
Assume $[B1]$. 
Then ${\rm tr}(\tilde{S}_m^{-1}(\sigma,v_{\ast}))=\bar{R}_n(b_n^{1/2}k_n)$.
\end{lemma}
\begin{discuss}
{\colorr (\ref{Sm-properties-lemma-eq1})は後ろでも使うからこの証明はAppには回さない.}
\end{discuss}

\begin{proof}
Let $D''_m={\rm diag}(|\tilde{b}^1_m|^2D'_{1,m},|\tilde{b}^2_m|^2D'_{2,m})$ and $D'''_m=(D''_m)^{-1/2}\tilde{D}_m(D''_m)^{-1/2}$, then we have
\begin{equation*}
\tilde{S}_m=(D''_m)^{1/2}(D'''_m)^{1/2}(\mathcal{E}+(D'''_m)^{-1/2}(D''_m)^{-1/2}(\tilde{S}_m-\tilde{D}_m)(D''_m)^{-1/2}(D'''_m)^{-1/2})(D'''_m)^{1/2}(D''_m)^{1/2}.
\end{equation*} 
Moreover, Lemma~\ref{inverse-norm-est}, $[B1]$, and Lemma 2 in~\cite{ogi-yos14} yield
\begin{eqnarray}
&&\left\lVert (D'''_m)^{-1/2}(D''_m)^{-1/2}(\tilde{S}_m-\tilde{D}_m)(D''_m)^{-1/2}(D'''_m)^{-1/2}\right\rVert \nonumber \\
&\leq &\left\lVert (D''_m)^{-1/2}(\tilde{S}_m-\tilde{D}_m)(D''_m)^{-1/2}\right\rVert
\leq \bar{\rho}\bigg\{\bigg\lVert\bigg\{\frac{|I^1_{i,m}\cap I^2_{j,m}|}{|I^1_{i,m}|^{1/2}|I^2_{j,m}|^{1/2}}\bigg\}_{i,j}\bigg\rVert \vee \bigg\lVert\bigg\{\frac{|I^1_{i,m}\cap I^2_{j,m}|}{|I^1_{i,m}|^{1/2}|I^2_{j,m}|^{1/2}}\bigg\}_{j,i}\bigg\rVert\bigg\}\leq \bar{\rho}<1. \nonumber
\end{eqnarray}
Therefore, we obtain
\begin{eqnarray}
{\rm tr}(\tilde{S}_m^{-1})&\leq &{\rm tr}((D'''_m)^{-1/2}(D''_m)^{-1}(D'''_m)^{-1/2})\rVert (\mathcal{E}+(D'''_m)^{-1/2}(D''_m)^{-1/2}(\tilde{S}_m-\tilde{D}_m)(D''_m)^{-1/2}(D'''_m)^{-1/2})^{-1}\lVert \nonumber \\
&\leq &{\rm tr}(\tilde{D}_m^{-1})/(1-\bar{\rho})\leq \sum_{j=1}^2{\rm tr}(\check{D}_{j,m}^{-1})r_n\underbar{r}_n^{-1}(1-\bar{\rho})^{-1}, \nonumber
\end{eqnarray}
by Lemma~\ref{tr-est}, the equation $\tilde{D}_{j,m}=\check{D}_{j,m}^{1/2}(\mathcal{E}-\check{D}_{j,m}^{-1/2}(\check{D}_{j,m}-\tilde{D}_{j,m})\check{D}_{j,m}^{-1/2})\check{D}_{j,m}^{1/2}$, and that
\begin{equation}\label{Sm-properties-lemma-eq1}
\lVert \check{D}_{j,m}^{-1/2}(\check{D}_{j,m}-\tilde{D}_{j,m})\check{D}_{j,m}^{-1/2}\lVert \leq (|\tilde{b}^j_m|^2r_n)^{-1}|\tilde{b}^j_m|^2(r_n-\underbar{r}_n)=1-\underbar{r}_n/r_n.
\end{equation}
We thus obtain the results by (\ref{tr-lim-est}).
\end{proof}
\begin{discuss}
$b$:constなら$\bar{\rho}$, $\check{D}_{j,m}$, $\tilde{D}_{j,m}$は$Y$によらないから$\sigma_{\ast}$の一様な評価も言える.
\end{discuss}

\subsection{Asymptotic equivalence of $H_n$ and $\tilde{H}_n$}

In this section, we prove the asymptotic equivalence of $H_n$ and $\tilde{H}_n$.
We provide the following lemma about estimates of moments of the quantities related to $H_n$ and $\tilde{H}_n$.
The proof is given in the appendix; it is obtained based on the properties of $M_{j,m}$ in Section~\ref{noise-cov-property-subsection},
standard It${\rm \hat{o}}$ calculus, and some results from linear algebra.

Let $\tilde{S}_{m,\ast}=\tilde{S}_m(\sigma_{\ast},v_{\ast})$ and 
\begin{equation*}
{\bf S}(t,x,\sigma,v)=\left(
\begin{array}{ll}
\{|b^1(t,x,\sigma)|^2|I^1_{i,m}|\delta_{ii'}\}_{ii'} + v_1M_{1,m} & \{b^1\cdot b^2(t,x,\sigma)|I^1_{i,m}\cap I^2_{j,m}|\}_{ij} \\
\{b^1\cdot b^2(t,x,\sigma)|I^1_{i,m}\cap I^2_{j,m}|\}_{ji} & \{|b^2(t,x,\sigma)|^2|I^2_{j,m}|\delta_{jj'}\}_{jj'} + v_2M_{2,m}
\end{array}
\right).
\end{equation*}

\begin{lemma}\label{ZSZ-est}
Assume $[B1]$.
Let $\sigma\in \Lambda$, $k_1,k_2,k_3\in \mathbb{Z}_+$, $k_1+k_2\geq 1$, {\colorg $k_1\leq 4$, $k_2\leq 4$,} ${\bf X}_m$ be a $\mathcal{G}_{s_{m-1}}$-measurable random variable,
and ${\bf S}'=\partial_{\sigma}^{k_1}\partial_x^{k_2}\partial_v^{k_3} {\bf S}^{-1}(s_{m-1},{\bf X}_m,\sigma,v_{\ast})$. Then
\begin{enumerate}
\item 
$E_m[(\tilde{Z}_m^{\top}{\bf S}'\tilde{Z}_m)^2]=2{\rm tr}(({\bf S}'\tilde{S}_{m,\ast})^2)+{\rm tr}({\bf S}'\tilde{S}_{m,\ast})^2+\bar{R}_n(1)$,
$E_m[(\tilde{Z}_m^{\top}{\bf S}'\tilde{Z}_m)^4]=\bar{R}_n((b_n^{-4}k_n^7)\vee (b_n^{-2}k_n^4))$ and $E_m[(\tilde{Z}_m^{\top}{\bf S}'\tilde{Z}_m)^q]=\bar{R}_n(b_n^{-q}k_n^{2q})$ for $q>4$.
\item $E_{\Pi}[|\sum_m(Z_m-\tilde{Z}_m)^{\top}{\bf S}'(Z_m+\tilde{Z}_m)|^q]={\colord \bar{R}_n((b_n^{-3}k_n^7)^{q/4})}$ {\colord for $q\geq 4$}.
\item 
$E_{\Pi}[|\sum_m(Z_m-\tilde{Z}_m)^{\top}{\bf S}'(Z_m+\tilde{Z}_m)|^2]=\bar{R}_n((b_n^{-1}k_n^2)\vee (b_n^{-2}k_n^{7/2}))$.
\end{enumerate}
\end{lemma}

\begin{proof}
See the appendix.
\end{proof}

\begin{discuss}
{\colorr
$k_n=o_p(b_n^{-3/4})$の下ではきついオーダーの評価が必要なので$q=2$の特殊な性質を使ってやるから分ける.
}
\end{discuss}

Now we obtain the {\colorg asymptotic} equivalence of $H_n$ and $\tilde{H}_n$.

\begin{lemma}\label{Hn-diff-lemma}
Assume $[B1]$, $[A2]$, and $[V]$. Then
\begin{equation*}
b_n^{-1/2}\sup_{\sigma\in\Lambda}|\partial_{\sigma}^j(H_n(\sigma,\hat{v}_n)-H_n(\sigma_{\ast},\hat{v}_n))-\partial_{\sigma}^j(\tilde{H}_n(\sigma,v_{\ast})-\tilde{H}_n(\sigma_{\ast},v_{\ast}))| \to^p 0,
\end{equation*}
and $b_n^{-1/4}(\partial_{\sigma}H_n(\sigma_{\ast},\hat{v}_n)-\partial_{\sigma}\tilde{H}_n(\sigma_{\ast},v_{\ast}))\to^p 0$
as $n\to\infty$ for $0\leq j\leq 3$. If $[B4]$ holds as well, then 
\begin{equation*}
E_{\Pi}\bigg[\bigg(b_n^{-1/2}\sup_{\sigma\in\Lambda}|\partial_{\sigma}^j(H_n(\sigma,\hat{v}_n)-H_n(\sigma_{\ast},\hat{v}_n))-\partial_{\sigma}^j(\tilde{H}_n(\sigma,v_{\ast})-\tilde{H}_n(\sigma_{\ast},v_{\ast}))|\bigg)^q\bigg]={\colord \bar{R}_n((b_n^{-5}k_n^7)^{q/4})}
\end{equation*}
for any $0\leq j\leq 3$ and $q>0$.
\end{lemma}

\begin{discuss}
{\colorr
$\partial_v^i$は, LANを示すときに使うが, 代入している$v$の値が違うし, 時間ずらしは必要ないのでここに書くよりは後ろに書いて
このLemmaの一部と同様に示せると書く.
}
\end{discuss}

\begin{proof}
We first obtain
\begin{eqnarray}
&&H_n(\sigma,v_{\ast})-H_n(\sigma_{\ast},v_{\ast})-(\tilde{H}_n(\sigma,v_{\ast})-\tilde{H}_n(\sigma_{\ast},v_{\ast})) \nonumber \\
&=&-\frac{\sigma-\sigma_{\ast}}{2}\sum_m\bigg\{(Z_m-\tilde{Z}_m)^{\top}\int^1_0\partial_{\sigma}S^{-1}_m(\sigma_t,v_{\ast})dt(Z_m+\tilde{Z}_m)
+\tilde{Z}_m\int^1_0(\partial_{\sigma}S^{-1}_m(\sigma_t,v_{\ast})-\partial_{\sigma}\tilde{S}^{-1}_m(\sigma_t,v_{\ast}))dt\tilde{Z}_m \nonumber \\
&&\quad +\int^1_0\partial_{\sigma}\log\frac{\det S_m(\sigma_t,v_{\ast})}{\det \tilde{S}_m(\sigma_t,v_{\ast})}dt\bigg\} 
=:\hat{\Psi}_{1,n}(\sigma)+\hat{\Psi}_{2,n}(\sigma)+\hat{\Psi}_{3,n}(\sigma). \nonumber
\end{eqnarray}

We will give estimates for these quantities. Point 2 of Lemma~\ref{ZSZ-est} yields $\sup_{\sigma}E_{\Pi}[|b_n^{-1/2}\partial_{\sigma}^j\hat{\Psi}_{1,n}|^q]={\colord \bar{R}_n((b_n^{-5}k_n^7)^{q/4})}$ for $0\leq j\leq 4$ and $q>0$,
and consequently by Sobolev's inequality $E_{\Pi}[\sup_{\sigma}|b_n^{-1/2}\partial_{\sigma}^j\hat{\Psi}_{1,n}|^q]={\colord \bar{R}_n((b_n^{-5}k_n^7)^{q/4})}$ as $n\to\infty$ for $0\leq j\leq 3$ and $q>0$.

Let ${\bf S}^{(1)}=\int^1_0\int^1_0\partial_x\partial_{\sigma}{\bf S}^{-1}(s_{m-1},s\hat{X}_m+(1-s)X_{s_{m-1}},\sigma_t,v_{\ast})dsdt$,
${\bf S}^{(2)}=\int^1_0\partial_x\partial_{\sigma}{\bf S}^{-1}(s_{m-1},X_{s_{m-2}},\sigma_t,v_{\ast})dt$ and ${\bf S}^{(3)}={\bf S}(s_{m-1},X_{s_{m-2}},\sigma_{\ast},v_{\ast})$.
Then we obtain
\begin{eqnarray}
E_{\Pi}[|\hat{\Psi}_{2,n}|^q] &=&2^{-q}E_{\Pi}\bigg[\bigg(\sum_m\tilde{Z}_m^{\top}{\bf S}^{(1)}\tilde{Z}_m(\hat{X}_m-X_{s_{m-1}})(\sigma-\sigma_{\ast})\bigg)^q\bigg] \nonumber \\
&\leq & CE_{\Pi}\bigg[\bigg(\sum_m{\rm tr}({\bf S}^{(1)}E_m[\tilde{Z}_m\tilde{Z}_m^{\top}])(\hat{X}_m-X_{s_{m-1}})\bigg)^q\bigg] \nonumber \\
&&+CE_{\Pi}\bigg[\bigg(\sum_m{\rm tr}({\bf S}^{(1)}\bar{E}_m[\tilde{Z}_m\tilde{Z}_m^{\top}])^2(\hat{X}_m-X_{s_{m-1}})^2\bigg)^{q/2}\bigg] \nonumber \\
&\leq & CE_{\Pi}\bigg[\bigg(\sum_m{\rm tr}({\bf S}^{(2)}{\bf S}^{(3)})(\hat{X}_m-X_{s_{m-1}})\bigg)^q\bigg]+\bar{R}_n((\ell_nb_n^{-1/2}k_n\ell_n^{-1})^q) +\bar{R}_n(\ell_n^{q/2}b_n^{-q/2}k_n^{q}\ell_n^{-q/2}) \nonumber \\
&=&\bar{R}_n(\ell_n^{q/2}b_n^{-q/2}k_n^q\ell_n^{-q/2})+\bar{R}_n(\ell_n^qb_n^{-q/2}k_n^q\ell_n^{-q})+\bar{R}_n(b_n^{-q/2}k_n^q)=\bar{R}_n(b_n^{-q/2}k_n^q) \nonumber
\end{eqnarray}
for any $q>0$, by the Burkholder--Davis--Gundy {\colorg inequality}, $E_{m-1}[\hat{X}_m-X_{s_{m-1}}]=\bar{R}_n(\ell_n^{-1})$ and $E_{\Pi}[|\hat{X}_m-X_{s_{m-1}}|^q]^{1/q}=\bar{R}_n(\ell_n^{-1/2})$.
\begin{discuss}
{\colorr ${\rm tr}({\bf S}^{(1)}\tilde{Z}_m\tilde{Z}_m^{\top})=\tilde{Z}_m^{\top}{\bf S}^{(1)}\tilde{Z}_m$も使う.}
\end{discuss}
Similar estimates for $\partial_{\sigma}^j\hat{\Psi}_{2,n}$ and Sobolev's inequality yield $E_{\Pi}[\sup_{\sigma}|b_n^{-1/2}\partial_{\sigma}^j\hat{\Psi}_{2,n}|^q]=\bar{R}_n(b_n^{-q}k_n^q)$ for $0\leq j\leq 3$ and $q>0$.

Similarly, we have $E_{\Pi}[\sup_{\sigma}|b_n^{-1/2}\partial_{\sigma}^j\hat{\Psi}_{3,n}|^q]=\bar{R}_n(b_n^{-q}k_n^q)$,
and therefore we obtain $E_{\Pi}[(b_n^{-1/2}\sup_{\sigma}|\partial_{\sigma}^j(H_n(\sigma,v_{\ast})-H_n(\sigma_{\ast},v_{\ast}))-\partial_{\sigma}^j(\tilde{H}_n(\sigma,v_{\ast})-\tilde{H}_n(\sigma_{\ast},v_{\ast}))|)^q]={\colord \bar{R}_n((b_n^{-5}k_n^7)^{q/4})}$ for $0\leq j\leq 3$.

Taylor's formula yields
\begin{eqnarray}
&&H_n(\sigma,\hat{v}_n)-H_n(\sigma_{\ast},\hat{v}_n)-(H_n(\sigma,v_{\ast})-H_n(\sigma_{\ast},v_{\ast})) \nonumber \\
&=&-\frac{1}{2}\sum_m\int^1_0\bigg\{Z_m^{\top}(\partial_{\sigma}S_m^{-1}(\sigma_t,\hat{v}_n)-\partial_{\sigma}S_m^{-1}(\sigma_t,v_{\ast}))Z_m+\partial_{\sigma}\log\frac{\det S(\sigma_t,\hat{v}_n)}{\det S(\sigma_t,v_{\ast})}\bigg\}dt(\sigma-\sigma_{\ast}) \nonumber \\
&=&-\frac{1}{2}\sum_m\sum_{j=1}^3\int^1_0\bigg\{Z_m^{\top}\partial_v^j\partial_{\sigma}S_m^{-1}(\sigma_t,v_{\ast})Z_m+\partial_v^j\partial_{\sigma}\log \det S(\sigma_t,v_{\ast})\bigg\}dt(\sigma-\sigma_{\ast})\frac{(\hat{v}_n-v_{\ast})^j}{j!} \nonumber \\
&&-\frac{1}{2}\sum_m\int^1_0\int^1_0\bigg\{Z_m^{\top}\partial_v^4\partial_{\sigma}S_m^{-1}(\sigma_t,v_s)Z_m+\partial_v^4\partial_{\sigma}\log \det S(\sigma_t,v_s)\bigg\}dsdt(\sigma-\sigma_{\ast})\frac{(\hat{v}_n-v_{\ast})^4}{4!}, \nonumber 
\end{eqnarray}
where $v_s=s\hat{v}_n+(1-s)v_{\ast}$.

Then we obtain 
\begin{eqnarray}
&&b_n^{-1/4}|H_n(\sigma,\hat{v}_n)-H_n(\sigma_{\ast},\hat{v}_n)-(H_n(\sigma,v_{\ast})-H_n(\sigma_{\ast},v_{\ast}))| \nonumber \\
&=&\bar{R}_n(b_n^{-1/4}b_n^{-1}k_n^2\ell_n)\times O_p(b_n^{-1/2})+\bar{R}_n(b_n^{-1/4}(b_nk_n\ell_n))\times O_p(b_n^{-2}) \to^p0, \nonumber
\end{eqnarray}
by Lemma~\ref{ZSZ-est}, $[V]$, and the equation $\partial_{\sigma}\log\det S=-{\rm tr}(\partial_{\sigma}SS^{-1})$.
Similarly, we obtain \\
$b_n^{-1/4}|\partial_{\sigma}^jH_n(\sigma,\hat{v}_n)-\partial_{\sigma}^jH_n(\sigma,v_{\ast})|\to^p 0$ for $1\leq j\leq 4$. 
Sobolev's inequality yields \\ $\sup_{\sigma}(b_n^{-1/4}|\partial_{\sigma}^j(H_n(\sigma,\hat{v}_n)-H_n(\sigma_{\ast},\hat{v}_n))-\partial_{\sigma}^j(H_n(\sigma,v_{\ast})-H_n(\sigma_{\ast},v_{\ast}))|) \to^p0$ for $0\leq j\leq 3$ and consequently we obtain
$\sup_{\sigma}(b_n^{-1/2}|\partial_{\sigma}^j(H_n(\sigma,\hat{v}_n)-H_n(\sigma_{\ast},\hat{v}_n))-\partial_{\sigma}^j(\tilde{H}_n(\sigma,v_{\ast})-\tilde{H}_n(\sigma_{\ast},v_{\ast}))|) \to^p0$ for $0\leq j\leq 3$.

Moreover, point 3 of Lemma~\ref{ZSZ-est} yields $E_{\Pi}[|b_n^{-1/4}\partial_{\sigma}\hat{\Psi}_{1,n}(\sigma_{\ast})|^2]=\bar{R}_n((b_n^{-3/2}k_n^2)\vee (b_n^{-5/2}k_n^{7/2}))\to^p0$,
and consequently $b_n^{-1/4}(\partial_{\sigma}H_n(\sigma_{\ast},\hat{v}_n)-\partial_{\sigma}\tilde{H}_n(\sigma_{\ast},v_{\ast}))\to^p0$.

If further $[B4]$ is satisfied, then for any $q>0$, we obtain 
\begin{eqnarray}
&&\sup_{\sigma}E_{\Pi}[b_n^{-q/2}|H_n(\sigma,\hat{v}_n)-H_n(\sigma_{\ast},\hat{v}_n)-(H_n(\sigma,v_{\ast})-H_n(\sigma_{\ast},v_{\ast}))|^q] \nonumber \\
&=&\bar{R}_n(b_n^{-q/2}b_n^{-q}k_n^{2q}\ell_n^qb_n^{-q/2})+\bar{R}_n(b_n^{-q/2}(b_nk_n\ell_n)^qb_n^{-2q}) =\bar{R}_n(b_n^{-q}k_n^q), \nonumber
\end{eqnarray}
by Lemma~\ref{ZSZ-est}, $[V]$, and the equation $\partial_{\sigma}\log\det S=-{\rm tr}(\partial_{\sigma}SS^{-1})$.
Similarly, we obtain \\
$\sup_{\sigma}E_{\Pi}[b_n^{-q/2}|\partial_{\sigma}^jH_n(\sigma,\hat{v}_n)-\partial_{\sigma}^jH_n(\sigma,v_{\ast})|^q]=\bar{R}_n(b_n^{-q}k_n^q)$ for $1\leq j\leq 4$. 
\begin{discuss}
{\colorr
$(\hat{v}_n-v_{\ast})^j/j!$は$\sigma$に依らないから祖母レスを使うときのモーメント評価をする必要がない.
$\partial_v\partial_{\sigma}S_m^{-1}$にも$\hat{v}_n^{-1}$がでてくるが, 収束するから$O_p(1)$なのでOK. PLDでは$\hat{v}_n^{-1}$のモーメント評価が必要.
}
\end{discuss}
Sobolev's inequality yields \\ $E_{\Pi}[\sup_{\sigma}(b_n^{-1/2}|\partial_{\sigma}^j(H_n(\sigma,\hat{v}_n)-H_n(\sigma_{\ast},\hat{v}_n))-\partial_{\sigma}^j(H_n(\sigma,v_{\ast})-H_n(\sigma_{\ast},v_{\ast}))|)^q] =\bar{R}_n(b_n^{-q}k_n^q)$ for $0\leq j\leq 3$, which completes the proof.
\end{proof}
\begin{discuss}
{\colorr
$b$: constの時は$\sigma_{\ast}$に対する一様評価も得られる$(H\equiv \tilde{H})$
}
\end{discuss}

\section{The limit of the quasi-likelihood function}\label{Hn-limit-section}

We complete the proof of Proposition~\ref{Hn-lim} in this section.
To do so, it is essential to specify the asymptotic behavior of some functions of approximate covariance matrix $\tilde{S}_m$, as seen in (\ref{Hn-diff-eq1}).
Unlike previous studies by Gloter and Jacod~\cite{glo-jac01a,glo-jac01b}, the eigenvalues of the diagonal blocks $\tilde{D}_{1,m}$ and $\tilde{D}_{2,m}$ of $\tilde{S}_m$ are not identified because of the irregular sampling,
and even the sizes of $\tilde{D}_{1,m}$ and $\tilde{D}_{2,m}$ are different.
These problems make it difficult to deduce asymptotic behaviors of the right-hand side of (\ref{Hn-diff-eq1}).
To solve these problems, in Lemma~\ref{D1-change}, we approximate $\tilde{D}_{j,m}$ by $\dot{D}_{j,m}$, which is a kind of local averaged versions of $\tilde{D}_{j,m}$ and has similar properties to the covariance matrix of equidistant sampling scheme. 
Moreover, we can also change the sizes of $\tilde{D}_{j,m}$ using some specific properties of $\dot{D}_{j,m}$. 
We deal with this in Lemma~\ref{Hn-lim-lemma2}, and show convergence of some trace functions that appear in a decomposition of $H_n$.
The decomposition (\ref{dotD-eq}) and the nice properties of $p_i$ in Lemma~\ref{ddotD-properties} are essential in the proofs.

Lemma~\ref{Hn-diff-lemma} yields
\begin{eqnarray}
&&b_n^{-1/4}\partial_{\sigma}^jH_n(\sigma,\hat{v}_n) \nonumber \\
&=&b_n^{-1/4}\partial_{\sigma}^j\tilde{H}_n(\sigma,v_{\ast})+o_p(1) \nonumber \\
&=&-\frac{1}{2}b_n^{-\frac{1}{4}}\sum_m(E_m[\tilde{Z}_m^{\top}\partial_{\sigma}^j\tilde{S}_m^{-1}\tilde{Z}_m]+\partial_{\sigma}^j\log \det \tilde{S}_m)
-\frac{1}{2}b_n^{-\frac{1}{4}}\sum_m(\tilde{Z}_m^{\top}\partial_{\sigma}^j\tilde{S}_m^{-1}\tilde{Z}_m-E_m[\tilde{Z}_m^{\top}\partial_{\sigma}^j\tilde{S}_m^{-1}\tilde{Z}_m])+o_p(1) \nonumber 
\end{eqnarray}
for $1\leq j\leq 4$. Together with the relation $E_m[\tilde{Z}_m^{\top}\partial_{\sigma}^j\tilde{S}_m^{-1}\tilde{Z}_m]={\rm tr}(\partial_{\sigma}^j\tilde{S}_m^{-1}\tilde{S}_{m,\ast})$, we obtain
\begin{eqnarray}\label{Hn-diff-eq1}
b_n^{-\frac{1}{2}}\partial_{\sigma}^j(H_n(\sigma,\hat{v}_n)-H_n(\sigma_{\ast},\hat{v}_n))
&=&-\frac{1}{2}b_n^{-\frac{1}{2}}\sum_m\partial_{\sigma}^j\bigg({\rm tr}(\tilde{S}_m^{-1}\tilde{S}_{m,\ast}-\mathcal{E})+\log \frac{\det \tilde{S}_m}{\det \tilde{S}_{m,\ast}}\bigg)+o_p(1), 
\end{eqnarray}
since the residual terms are $o_p(1)$ by Lemma~\ref{ZSZ-est}.
\begin{discuss}
{\colorr
\begin{eqnarray}
E[|(\mbox{residual \ terms})|^2]&\leq &Cb_n^{-1}\sum_mE[(\tilde{Z}_m^{\top}\partial_{\sigma}^j(\tilde{S}_m^{-1}-\tilde{S}^{-1}_{m,\ast})\tilde{Z}_m)^2] \nonumber \\
&=&3Cb_n^{-1}\sum_mE[{\rm tr}(\partial_{\sigma}^j(\tilde{S}_m^{-1}-\tilde{S}^{-1}_{m,\ast})\tilde{S}_{m,\ast}\partial_{\sigma}^v(\tilde{S}_m^{-1}-\tilde{S}^{-1}_{m,\ast})\tilde{S}_{m,\ast})]=o(1). \nonumber
\end{eqnarray}
}
\end{discuss}

We first investigate asymptotics of ${\rm tr}(\tilde{S}_m^{-1}\tilde{S}_{m,\ast}-\mathcal{E})$.
Let $\tilde{L}=\{\tilde{b}^1_m\cdot \tilde{b}^2_m|I^1_{i,m}\cap I^2_{j,m}|\}_{i,j}$ and $\tilde{G}=\{|I^1_{i,m}\cap I^2_{j,m}|\}_{i,j}$. Then since
\begin{eqnarray}
\tilde{S}_m^{-1}&=&
\tilde{D}_m^{-1/2}\sum_{p=0}^{\infty}(-1)^p\left(
\begin{array}{ll}
0 & \tilde{D}_{1,m}^{-1/2}\tilde{L}\tilde{D}_{2,m}^{-1/2} \\
\tilde{D}_{2,m}^{-1/2}\tilde{L}^{\top}\tilde{D}_{1,m}^{-1/2} & 0 \\
\end{array}
\right)^p\tilde{D}_m^{-1/2} \nonumber \\
&=&\sum_{p=0}^{\infty}\left(
\begin{array}{ll}
\tilde{D}_{1,m}^{-1/2}(\tilde{D}_{1,m}^{-1/2}\tilde{L}\tilde{D}_{2,m}^{-1}\tilde{L}^{\top}\tilde{D}_{1,m}^{-1/2})^p\tilde{D}_{1,m}^{-1/2} & -\tilde{D}_{1,m}^{-1}\tilde{L}\tilde{D}_{2,m}^{-1/2}(\tilde{D}_{2,m}^{-1/2}\tilde{L}^{\top}\tilde{D}_{1,m}^{-1}\tilde{L}\tilde{D}_{2,m}^{-1/2})^p\tilde{D}_{2,m}^{-1/2} \\
-\tilde{D}_{2,m}^{-1}\tilde{L}^{\top}\tilde{D}_{1,m}^{-1/2}(\tilde{D}_{1,m}^{-1/2}\tilde{L}\tilde{D}_{2,m}^{-1}\tilde{L}^{\top}\tilde{D}_{1,m}^{-1/2})^p\tilde{D}_{1,m}^{-1/2} & \tilde{D}_{2,m}^{-1/2}(\tilde{D}_{2,m}^{-1/2}\tilde{L}^{\top}\tilde{D}_{1,m}^{-1}\tilde{L}\tilde{D}_{2,m}^{-1/2})^p\tilde{D}_{2,m}^{-1/2}
\end{array} 
\right), \nonumber
\end{eqnarray}
we have
\begin{eqnarray}\label{Hn-diff-eq2}
&&{\rm tr}(\tilde{S}_m^{-1}\tilde{S}_{m,\ast}-\mathcal{E}) \nonumber \\
&=&\sum_{p=0}^{\infty}\big\{(|\tilde{b}^1_{m,\ast}|^2-|\tilde{b}^1_m|^2){\rm tr}((\tilde{D}_{1,m}^{-1}\tilde{L}\tilde{D}_{2,m}^{-1}\tilde{L}^{\top})^p\tilde{D}_{1,m}^{-1}D'_{1,m})
+(|\tilde{b}^2_{m,\ast}|^2-|\tilde{b}^2_m|^2){\rm tr}((\tilde{D}_{2,m}^{-1}\tilde{L}^{\top}\tilde{D}_{1,m}^{-1}\tilde{L})^p\tilde{D}_{2,m}^{-1}D'_{2,m}) \nonumber \\
&&-2(\tilde{b}^1_{m,\ast}\cdot \tilde{b}^2_{m,\ast}-\tilde{b}^1_m\cdot \tilde{b}^2_m){\rm tr}(\tilde{D}_{1,m}^{-1}\tilde{L}\tilde{D}_{2,m}^{-1}(\tilde{L}^{\top}\tilde{D}_{1,m}^{-1}\tilde{L}\tilde{D}_{2,m}^{-1})^p\tilde{G}^{\top})\big\}.
\end{eqnarray}
{\colord Note that $\lVert \tilde{G}\rVert \vee \lVert \tilde{G}^{\top}\rVert\leq r_n$ by Lemma 2 in~\cite{ogi-yos14}.}

We will see the limit of each term on the right-hand side.
Let $\hat{a}_m^j=a_{s_{m-1}}^j$ and $\dot{D}_{j,m}=|\tilde{b}^j_m|^2b_n^{-1}(\hat{a}_m^j)^{-1}\mathcal{E}+v_{j,\ast}M_{j,m}$.
It is difficult to calculate each element or eigenvalue of $\tilde{D}_{j,m}^{-1}$.
However, we can apply (\ref{dotD-eq}) to $\dot{D}_{j,m}^{-1}$, and hence we can calculate its elements.
Therefore, we replace $\tilde{D}_{j,m}$ by $\dot{D}_{j,m}^{-1}$ {\colorg using} the following lemma.

\begin{lemma}\label{D1-change}
Let $j\in\{1,2\}$ and $A_{n,m}$ be a $k^j_m\times k^j_m$ matrix for $1\leq m\leq \ell_n$. 
Assume $[B1]$, $[A2]$ and that all elements of $A_{n,m}$ are nonnegative 
and $\lVert A_{n,m}\rVert\leq 1$ for any $m$. Then
\begin{equation*}
{\rm tr}(\partial_{\sigma}^k\tilde{D}_{j,m}^{-1}A_{n,m})={\rm tr}(\partial_{\sigma}^k\dot{D}_{j,m}^{-1}A_{n,m})+\dot{R}_n(b_n^{3/2}\ell_n^{-1})
\end{equation*}
for $0\leq k\leq 3$. {\colord If further $[B2]$ is satisfied, then
\begin{equation*}
\sup_{\sigma}|{\rm tr}(\partial_{\sigma}^k\tilde{D}_{j,m}^{-1}A_{n,m})-{\rm tr}(\partial_{\sigma}^k\dot{D}_{j,m}^{-1}A_{n,m})|=\underbar{R}_n(b_n^{3/2}\ell_n^{-1})
\end{equation*}
for $0\leq k\leq 3$.}
\end{lemma}

\begin{proof}
We first consider the case where $k=0$. (\ref{dotD-eq}),
(\ref{Sm-properties-lemma-eq1}), and the equation above it yield
\begin{eqnarray}\label{trD-eq}
&&{\rm tr}(\tilde{D}_{j,m}^{-1}A_{n,m}) \nonumber \\
&=&\sum_{p=0}^{\infty}{\rm tr}\left(\check{D}_{j,m}^{-1/2}(\check{D}_{j,m}^{-1/2}(\check{D}_{j,m}-\tilde{D}_{j,m})\check{D}_{j,m}^{-1/2})^p\check{D}_{j,m}^{-1/2}A_{n,m}\right) \nonumber \\
&=&\frac{1}{v_{j,\ast}^{p+1}}\sum_{p=0}^{\infty}\sum_{i_1,\cdots,i_{p+1}}\frac{1}{p_{i_1}\cdots p_{i_{p+1}}(|\tilde{b}^j_m|^2r_nv_{j,\ast}^{-1})} 
\sum_{\substack{l_q\leq i_q\wedge i_{q+1}\\ l'\leq i_{p+1},l''\leq i_1}}\frac{(A_{n,m})_{l',l''}}{P_{l',l'',i_{p+1},i_1}}\prod_{q=1}^p\frac{(\check{D}_{j,m}-\tilde{D}_{j,m})_{l_q,l_q}}{P_{l_q,l_q,i_q,i_{q+1}}}, 
\end{eqnarray}
where $P_{k_1,k_2,l_1,l_2}=\prod_{m_1;k_1\leq m_1\leq l_1-1} p_{m_1} \prod_{m_2;k_2\leq m_2\leq l_2-1} p_{m_2}(|\tilde{b}^j_m|^2r_nv_{j,\ast}^{-1})$.

Then the nice properties of $p_i$ in Lemma~\ref{ddotD-properties} will lead us to the desired results.
Roughly speaking, we have $1\leq p_i(|\tilde{b}^j_m|^2r_nv_{j,\ast}^{-1})\sim 1+Cb_n^{-1/2}$ for sufficiently large $i$.
This means that $P_{k_1,k_2,l_1,l_2}$ and $P_{k'_1,k'_2,l_1,l_2}$ are asymptotically equivalent if $|k_1-k'_1|$ and $|k_2-k'_2|$ are of order less than $b_n^{1/2}$. Then we can replace $\tilde{D}_{j,m}$ in the right-hand side of (\ref{trD-eq}) by $\dot{D}_{j,m}$ 
since the diagonal elements of both matrices have the same local average.
We will verify these rough sketches by the following.

We first see that terms containing small $l_q$ can be ignored.
Let $\eta$ be the one in $[A2]$, $\eta'\in (\eta,1/2)$, 
$t_{\tilde{l},m}=s_{m-1}+T[b_nk_n^{-1}]^{-1}b_n^{\eta'}/k_n + T[b_nk_n^{-1}]^{-1}(k_n-b_n^{\eta'})[(k_n-b_n^{\eta'})b_n^{-\eta}]^{-1}\tilde{l}/k_n$ for $0\leq \tilde{l}\leq [(k_n-b_n^{\eta'})b_n^{-\eta}]$,
$\mathcal{I}_m(\tilde{l})=\{l;I^j_{l,m}\subset [t_{\tilde{l}-1,m},t_{\tilde{l},m})\}$, and $\mathcal{E}'=\{\delta_{i_1,i_2}1_{\{\inf I^j_{i_1,m}<t_0\} }\}_{1\leq i_1,i_2\leq k^j_m}$.
Then the absolute value of summation involving terms with $l_q$ satisfying $\inf I^j_{l_q,m}<t_0$ for some $1\leq q\leq p$ on the right-hand side of the above equation
is less than 
\begin{eqnarray}\label{D1-change-lemma-eq2}
&&\lVert A_{n,m}\rVert \sum_{p=1}^{\infty}p\lVert \check{D}_{j,m}^{-1/2}(\check{D}_{j,m}-\tilde{D}_{j,m})\check{D}_{j,m}^{-1/2}\rVert^{p-1}
\lVert \check{D}_{j,m}^{-1/2}\rVert^2|\tilde{b}^j_m|^2{\rm tr}(\check{D}_{j,m}^{-1/2}(r_n-\underbar{r}_n)\mathcal{E}'\check{D}_{j,m}^{-1/2})
\leq r_n^2\underbar{r}_n^{-2}{\rm tr}(\check{D}_{j,m}^{-1}\mathcal{E}'), \nonumber \\
\end{eqnarray}
\begin{discuss}
{\colorr 
\begin{equation*}
\mbox{途中式}=\sum_{p=1}^{\infty}p(1-\underbar{r}_n/r_n)^{p-1}(|\tilde{b}^j_m|^2r_n)^{-1}|\tilde{b}^j_m|^2{\rm tr}(\check{D}_{j,m}^{-1/2}(r_n-\underbar{r}_n)\mathcal{E}'\check{D}_{j,m}^{-1/2})
\end{equation*}
}
\end{discuss}
by (\ref{Sm-properties-lemma-eq1}), Lemma~\ref{tr-est}, and the assumptions.
Moreover, point 2 of Lemma~\ref{ddotD-properties} ensures that ${\rm tr}(\check{D}_{j,m}^{-1}\mathcal{E}')$ is less than $[[k^j_m/2]/[Tb_n^{-1+\eta'}\underbar{r}_n^{-1}/2]]^{-1}{\rm tr}(\check{D}_{j,m}^{-1})$ 
{\colord if $[k^j_m/2]\geq 2[Tb_n^{-1+\eta'}\underbar{r}_n^{-1}]$},
and hence the right-hand side of (\ref{D1-change-lemma-eq2}) is $\bar{R}_n(\mathcal{V}_n)$
by {\colorlg ${\rm tr}(\check{D}_{j,m}^{-1})=\bar{R}_n(r_n^{-1/2}k_n)$,
where $\mathcal{V}_n=b_n^{1/2+\eta'}1_{\{[\bar{k}_n/2]\geq 2[Tb_n^{-1+\eta'}\underbar{r}_n^{-1}]\}}+b_n^{1/2}k_n1_{\{[\underline{k}_n/2]< 2[Tb_n^{-1+\eta'}\underbar{r}_n^{-1}]\}}$}.

Then for $\tilde{l}$, $i_q,i_{q+1}$ and $l_q$ satisfying $l_q\in\mathcal{I}_m(\tilde{l})$ and $\max \mathcal{I}_m(\tilde{l})\leq i_q\wedge i_{q+1}$, Lemma~\ref{ddotD-properties} yields that\\
$P_{\max \mathcal{I}_m(\tilde{l}),\max \mathcal{I}_m(\tilde{l}),i_q,i_{q+1}}/P_{l_q,l_q,i_q,i_{q+1}}$ is less than $1$ and greater than
\begin{equation*}
(1+C^{-1}b_n^{1-\eta'}\underbar{r}_n+C|\tilde{b}^j_m|^2b_n^{-1+\eta'}r_n\underbar{r}_n^{-1})^{-Cb_n^{-1+\eta}\underbar{r}_n^{-1}}\geq 1-\bar{R}_n(b_n^{\eta-\eta'}+b_n^{-2+\eta'+\eta}\underbar{r}_n^{-2}r_n)
\end{equation*}
{\colord for sufficiently large $n$}. 
\begin{discuss}
{\colorr $P[\max_{j,m}|\tilde{b}^j_m|>b_n^{\epsilon}]<cb_n^{-q}$ for any $q$と, $\epsilon,x\epsilon$:十分小の時$(1+\epsilon)^{-x}\geq 1-2\epsilon x$を使う.}
\end{discuss}
Moreover, $\max_{\tilde{l}}|\sum_{l\in \mathcal{I}_m(\tilde{l})}(\tilde{D}_{j,m}-\dot{D}_{j,m})_{l,l}|$ is {\colord $\dot{R}_n(b_n^{-1+\eta})$ by $[A2]$}.	
{\colord We also have $\sup_{\sigma,m}\max_{\tilde{l}}|\sum_{l\in \mathcal{I}_m(\tilde{l})}(\tilde{D}_{j,m}-\dot{D}_{j,m})_{l,l}|=\underbar{R}_n(b_n^{-1+\eta})$ if $[B2]$ is satisfied.}
\begin{discuss}
{\colorr
\begin{eqnarray}
&&\max_{\tilde{l}}|\sum_{l\in \mathcal{I}_m(\tilde{l})}(\tilde{D}_{j,m}-\dot{D}_{j,m})_{l,l}| \nonumber \\
&=&|\tilde{b}^j_m|^2\max_{\tilde{l}}|\sum_{l\in \mathcal{I}_m(\tilde{l})}(|I^j_{l,m}|-b_n^{-1}(\hat{a}_m^j)^{-1})| 
\leq |\tilde{b}^j_m|^2\max_{\tilde{l}}|t_{\tilde{l}}-t_{\tilde{l}-1}-\sum_{l\in \mathcal{I}_m(\tilde{l})}b_n^{-1}(\hat{a}_m^j)^{-1}|+C|\tilde{b}^j_m|^2r_n \nonumber \\
&=&|\tilde{b}^j_m|^2(\hat{a}_m^j)^{-1}(t_{\tilde{l}}-t_{\tilde{l}-1})
|(t_{\tilde{l}}-t_{\tilde{l}-1})^{-1}\sum_{l\in \mathcal{I}_m(\tilde{l})}b_n^{-1}-\hat{a}_m^j|+C|\tilde{b}^j_m|^2r_n=\dot{R}_n(b_n^{-1+\eta}). \nonumber 
\end{eqnarray}
$a^j_t$の時刻ずらしに$\dot{\eta}$-H${\rm \ddot{o}}$lder連続性を使う. $\sup_t (a^j_t)^{-1}<\infty$ a.s.も使う.
}
\end{discuss}

Therefore we obtain
\begin{eqnarray}\label{averaging-eq}
&&{\rm tr}(\tilde{D}_{j,m}^{-1}A_{n,m}) \nonumber \\
&=&\sum_{p=0}^{\infty}\sum_{i_1,\cdots,i_{p+1}}\frac{1}{p_{i_1}\cdots p_{i_{p+1}}(|\tilde{b}^j_m|^2r_n)} 
\sum_{\substack{l_q\leq i_q\wedge i_{q+1},\inf I^j_{l_q,m}\geq t_0\\l'\leq i_{p+1},l''\leq i_1}}\frac{(A_{n,m})_{l',l''}}{P_{l',l'',i_{p+1},i_1}}\prod_{q=1}^p\frac{(\check{D}_{j,m}-\tilde{D}_{j,m})_{l_q,l_q}}{P_{l_q,l_q,i_q,i_{q+1}}}+\bar{R}_n(\mathcal{V}_n) \nonumber \\
&=&\sum_{p=0}^{\infty}\mathcal{T}^{n,p}_{m,1}\sum_{i_1,\cdots,i_{p+1}}\frac{1}{p_{i_1}\cdots p_{i_{p+1}}(|\tilde{b}^j_m|^2r_n)}\sum_{l'\leq i_{p+1},l''\leq i_1}\frac{(A_{n,m})_{l',l''}}{P_{l',l'',i_{p+1},i_1}} \nonumber \\
&&\times \prod_{q=1}^p\sum_{1\leq \tilde{l}_q\leq [b_n^{\eta}]^{-1}(k_n-[b_n^{\eta'}])}\frac{\sum_{l_q\in\mathcal{I}_m(\tilde{l}_q)}(\check{D}_{j,m}-\dot{D}_{j,m}+\mathcal{T}^{n,p}_{m,3}\mathcal{E})_{l_q,l_q}}{P_{\max \mathcal{I}_m(\tilde{l}_q),\max \mathcal{I}_m(\tilde{l}_q),i_q,i_{q+1}}}+\bar{R}_n(\mathcal{V}_n) \nonumber \\
&=&\sum_{p=0}^{\infty}\mathcal{T}^{n,p}_{m,1}(\mathcal{T}^{n,p}_{m,2})^{-1}{\rm tr}(\check{D}_{j,m}^{-1/2}(\check{D}_{j,m}^{-1/2}(\check{D}_{j,m}-\dot{D}_{j,m}+\mathcal{T}^{n,p}_{m,3}\mathcal{E})\check{D}_{j,m}^{-1/2})^p\check{D}_{j,m}^{-1/2}A_{n,m}) + \dot{R}_n(b_n^{3/2}\ell_n^{-1}), 
\end{eqnarray}
where $\mathcal{T}^{n,p}_{m,i}$ is a random variable which does not depend on $l_q,\tilde{l}_q,i_q,i_{q+1}$ and satisfies 
\begin{equation*}
(1-\bar{R}_n(b_n^{\eta-\eta'}+b_n^{-2+\eta'+\eta}\underbar{r}_n^{-2}r_n))^p\leq \mathcal{T}^{n,p}_{m,i}\leq 1
\end{equation*}
for $i=1,2$ and {\colord $\sup_p|\mathcal{T}^{n,p}_{m,3}|=\dot{R}_n(b_n^{-1})$}.
\begin{discuss}
{\colorr $\mathcal{T}^{n,p}_{m,3}$は各成分に分配しているから, $\max_{\tilde{l}}|\sum_{l\in \mathcal{I}_m(\tilde{l})}(\tilde{D}_{j,m}-\dot{D}_{j,m})_{l,l}|$の評価を
$\sum_{l\in \mathcal{I}_m(\tilde{l})}1$で割ったもので評価できる. $\sum 1$は$[A2]$より$O_p(b_n^{\eta})$.
こういう$\mathcal{T}$が取れるのは, すべての要素がnonnegativeで$\check{D}-\dot{D}+t\mathcal{E}$の対角成分がすべて等しいから. 
$b_n(r_n-(\hat{a}^j_m)^{-1}b_n^{-1})$が小さいときは$\mathcal{T}$の値を連続的に動かせば対角成分が等しいことから右辺を$0$にできる. 小さくないときも$\mathcal{T}$を小さくしていけば, 
左辺より小さくできる. また左辺より各項を大きくすることもできるので途中のどこかでイコールになる.
$P_{\max \mathcal{I}_m,\max \mathcal{I}_m,,,}$をもとに戻せることも右辺の各対角が（すべて等しいため）正なので明らか. $\inf I^j_{i,m}<t_0$をもとに戻すのも(\ref{D1-change-lemma-eq2})に
$\mathcal{T}^{n,p}_{m,1}(\mathcal{T}^{n,p}_{m,2})^{-1}$を付けたものが引き続き成り立つからよい.
（$\bar{R}_n(b_n^{\eta-\eta'}+b_n^{-1+\eta+\eta'})$は$r_n\underbar{r}_n^{-1}$よりも速く$0$に収束.）
}
\end{discuss}

Let $F_p(t)={\rm tr}(\check{D}_{j,m}^{-1/2}(\check{D}_{j,m}^{-1/2}(\check{D}_{j,m}-\dot{D}_{j,m}+t\mathcal{T}^{n,p}_{m,3}\mathcal{E})\check{D}_{j,m}^{-1/2})^p\check{D}_{j,m}^{-1/2}A_{n,m})$.
Then 
\begin{equation*}
|F_p(1)-F_p(0)|\leq \int^1_0|F'_p(t)|dt\leq p|\tilde{b}^j_m|^{-2}|\mathcal{T}^{n,p}_{m,3}|(1-b_n^{-1}(\hat{a}^j_m)^{-1}r_n^{-1}+|\tilde{b}^j_m|^{-2}\mathcal{T}^{n,p}_{m,3}r_n^{-1})^{p-1}\bar{R}_n(r_n^{-1/2}b_n^2\ell_n^{-1}),
\end{equation*}
and hence $\sum_{p=0}^{\infty}|\mathcal{T}^{n,p}_{m,1}(\mathcal{T}^{n,p}_{m,2})^{-1}||F_p(1)-F_p(0)|\leq \sup_p|\mathcal{T}^{n,p}_{m,3}|\cdot \dot{R}_n(b_n^{3/2}k_n)=\dot{R}_n(b_n^{3/2}\ell_n^{-1})$.
\begin{discuss}
{\colorr $\bar{R}_n(b_n^{\eta-\eta'}+b_n^{-1+eta+\eta'})$は$b_n\underbar{r}_n^{-1}$より速く$0$に収束するから$(1+x)^p(1-y)^p\leq (1-y/2)^p$とできる.}
\end{discuss}
Therefore we obtain the desired conclusion by
\begin{eqnarray}
&&\sum_{p=0}^{\infty}|\mathcal{T}^{n,p}_{m,1}(\mathcal{T}^{n,p}_{m,2})^{-1}-1|F_p(0) \nonumber \\
&\leq &C\sum_pp\bar{R}_n(b_n^{\eta-\eta'}+b_n^{-1+\eta+\eta'})(1+\bar{R}_n(b_n^{\eta-\eta'}+b_n^{-1+\eta+\eta'}))^{p-1}(1-b_n^{-1}(\hat{a}^j_m)^{-1}r_n^{-1})^pr_n^{-1/2}k^j_m=\dot{R}_n(b_n^{3/2}\ell_n^{-1}). \nonumber 
\end{eqnarray}
\begin{discuss}
{\colorr $(1+x)^p-1=\int^1_0px(1+tx)^{p-1}dt\leq px(1+x)^{p-1}$より, 
}
\end{discuss}

{\colord For the case $k=1$, we have 
${\rm tr}(\partial_{\sigma}\tilde{D}_{j,m}^{-1}A_{n,m})={\rm tr}(\dot{D}_{j,m}^{-1}\partial_{\sigma}\tilde{D}_{j,m}\dot{D}_{j,m}^{-1}A_{n,m})+\dot{R}_n(b_n^{3/2}\ell_n^{-1})$
by using the result for $k=0$, $\partial_{\sigma}\tilde{D}_{j,m}^{-1}=\tilde{D}_{j,m}^{-1}\partial_{\sigma}\tilde{D}_{j,m}\tilde{D}_{j,m}^{-1}$, $\lVert \tilde{D}_{j,m}^{-1}\partial_{\sigma}\tilde{D}_{j,m}\rVert=\bar{R}_n(1)$
and all elements of $\tilde{D}_{j,m}^{-1}$ are nonnegative by a similar argument to (\ref{dotD-eq}).
\begin{discuss}
{\colorr 各要素が非負であることを言うには, $p_j$に対応する量が$1$より大きいことを言えばよいが, これは$\epsilon$が$j$に依存していても帰納的に言える.}
\end{discuss}
Then a similar argument to (\ref{averaging-eq}) enables us to replace $\partial_{\sigma}\tilde{D}_{j,m}$ by $\partial_{\sigma}\dot{D}_{j,m}$.
Similarly, we obtain ${\rm tr}(\partial_{\sigma}^k\tilde{D}_{j,m}^{-1}A_{n,m})={\rm tr}(\partial_{\sigma}^k\dot{D}_{j,m}^{-1}A_{n,m})+\dot{R}_n(b_n^{3/2}\ell_n^{-1})$ for $k=2,3$.

If further $[B2]$ is satisfied, then similarly we have $\sup_{\sigma}|{\rm tr}(\partial_{\sigma}^k\tilde{D}_{j,m}^{-1}A_{n,m})-{\rm tr}(\partial_{\sigma}^k\dot{D}_{j,m}^{-1}A_{n,m})|=\underbar{R}_n(b_n^{3/2}\ell_n^{-1})$.
}


\end{proof}

{\colorlg
\begin{remark}
The proof shows that there are upperbounds of the absolute values of residual terms 
in the statement of Lemma \ref{D1-change} which do not depend on $A_{n,m}$.
\end{remark}
}
	
\begin{discuss}
{\colorr マルチンゲール性などから評価を出していた四章とは違って, pathwiseの議論から評価しているのでソボレフの不等式を使うことなくsupの評価が得られる.}
\end{discuss}

{\colorlg
Let $c_j=|\tilde{b}^j_m|^2b_n^{-1}(\hat{a}_m^j)^{-1}/v_{j,\ast}$ and $c'_j=c_j(\hat{a}_m^j/\hat{a}^{3-j}_m)^2$ for $j=1,2$. 
\begin{lemma}\label{Hn-lim-lemma2}
Assume $[B1]$ and $[A2]$. Then 
\begin{eqnarray}
\sup_{\sigma,m}\bigg|\partial_{\sigma}^k{\rm tr}((\tilde{D}_{1,m}^{-1}\tilde{G}\tilde{D}_{2,m}^{-1}\tilde{G}^{\top})^p\tilde{D}_{1,m}^{-1}D'_{1,m})
-\frac{T\hat{a}_m^1k_n}{\pi}\frac{(\hat{a}_m^2)^p\partial_{\sigma}^kI_{p+1,p}(c_1,c'_2)}{b_n^{2p+1}(\hat{a}_m^1)^{3p+1}v_{1,\ast}^{p+1}v_{2,\ast}^p}\bigg|&=&o_p(b_n^{1/2}\ell_n^{-1}), \label{tr-eq1} \\
\sup_{\sigma,m}\bigg|\partial_{\sigma}^k{\rm tr}((\tilde{D}_{2,m}^{-1}\tilde{G}^{\top}\tilde{D}_{1,m}^{-1}\tilde{G})^p\tilde{D}_{2,m}^{-1}D'_{2,m})
-\frac{T\hat{a}_m^1k_n}{\pi}\frac{(\hat{a}_m^2)^{p+1}\partial_{\sigma}^kI_{p,p+1}(c_1,c'_2)}{b_n^{2p+1}(\hat{a}_m^1)^{3p+2}v_{1,\ast}^pv_{2,\ast}^{p+1}}\bigg|&=&o_p(b_n^{1/2}\ell_n^{-1}),\label{tr-eq2} \\
\sup_{\sigma,m}\bigg|\partial_{\sigma}^k{\rm tr}((\tilde{D}_{1,m}^{-1}\tilde{G}\tilde{D}_{2,m}^{-1}\tilde{G}^{\top})^{p+1})
-\frac{T\hat{a}_m^1k_n}{\pi}\frac{(\hat{a}_m^2)^{p+1}\partial_{\sigma}^kI_{p+1,p+1}(c_1,c'_2)}{(b_n^2(\hat{a}_m^1)^3v_{1,\ast}v_{2,\ast})^{p+1}}\bigg|&=&o_p(b_n^{1/2}\ell_n^{-1}) \label{tr-eq3}
\end{eqnarray}
for $0\leq k\leq 3$ and $p\in \mathbb{Z}_+$. {\colord If further $[B2]$ is satisfied, then $o_p(b_n^{1/2}\ell_n^{-1})$ in (\ref{tr-eq1})-(\ref{tr-eq3}) can be replaced by $\underbar{R}_n(b_n^{1/2}\ell_n^{-1})$.}
\end{lemma}

\begin{discuss}
{\colorr 残差は$p$に関して一様に抑えられているわけではない.
$\dot{R}_n$ではなく, $o_p$を使っているのは, 
$n^{\delta}$の隙間をとることができないから$r_nb_n$の評価ができないからか.
できるような気もするが、とりあえずこのまま保留
}
\end{discuss}
}

\begin{proof}

For any $p\in\mathbb{N}$, Lemma~\ref{D1-change} yields 
\begin{equation}\label{Hn-lim-lemma2-eq0}
{\colorlg b_n^{-1/2}\partial_{\sigma}^k{\rm tr}((\tilde{D}_{1,m}^{-1}\tilde{G}\tilde{D}_{2,m}^{-1}\tilde{G}^{\top})^p) 
=b_n^{-1/2}\partial_{\sigma}^k{\rm tr}((\dot{D}_{1,m}^{-1}\tilde{G}\dot{D}_{2,m}^{-1}\tilde{G}^{\top})^p) + \dot{R}_n(\ell_n^{-1})}. 
\end{equation}

{\colorlg Moreover, we have}
\begin{equation}\label{Hn-lim-lemma2-eq1}
b_n^{-1/2}{\rm tr}((\dot{D}_{1,m}^{-1}\tilde{G}\dot{D}_{2,m}^{-1}\tilde{G}^{\top})^p) 
=b_n^{-1/2}\frac{1}{v_{1,\ast}^pv_{2,\ast}^p}\sum_{\substack{i_1,\cdots,i_p \\ j_1,\cdots,j_p}} \sum_{\substack{\alpha_{2q-1}\leq i_q, \beta_{2q-1}\leq j_q \\ \alpha_{2q}\leq i_{q+1}, \beta_{2q}\leq j_q \ (1\leq q\leq p)}}
\prod_{q=1}^p\frac{\tilde{G}_{\alpha_{2q-1},\beta_{2q-1}}}{\tilde{P}_{\alpha_{2q-1},\beta_{2q-1},i_q+1,j_q+1}}\frac{\tilde{G}_{\alpha_{2q},\beta_{2q}}}{\tilde{P}_{\alpha_{2q},\beta_{2q},i_{q+1},j_q}}
\end{equation}
by (\ref{dotD-eq}), where $i_{p+1}=i_1$ and $\tilde{P}_{\alpha,\beta,i,j}=\prod_{k_1=\alpha}^{i-1}p_{k_1}(c_1)\prod_{k_2=\beta}^{j-1}p_{k_2}(c_2)$.

We will apply (\ref{tr-lim-est}) to obtain the limit of the traces.
To do so, we need to change the size of matrices $\tilde{G}$ and $\dot{D}_{2,m}^{-1}$.
This is again achieved by the nice properties of $p_i$.
The essential idea is that point 3 of Lemma~\ref{ddotD-properties} ensures $p_i\sim p_+$ for sufficiently large $i$,
and therefore $\tilde{P}_{\alpha,\beta,i,j}\sim p_+(c_1)^{i-\alpha}p_+(c_2)^{j-\beta}\sim \exp(\sqrt{c_1}(i-\alpha)+\sqrt{c_2}(j-\beta))\sim \acute{P}_{k\alpha,k\beta,ki,kj}$,
{\colord where $k\in\mathbb{N}$ and $\acute{P}_{\alpha',\beta',i',j'}=\prod_{k_1=\alpha'}^{i'-1}p_{k_1}(c_1/k^2)\prod_{k_2=\beta'}^{j'-1}p_{k_2}(c_2/k^2)$}.
The size of $\dot{D}_{2,m}^{-1}$ decides the ranges of summation of $j_1,\cdots, j_p$ in (\ref{Hn-lim-lemma2-eq1}). 
By changing these ranges using the above relation on $\tilde{P}_{\alpha,\beta,i,j}$ {\colord and $\acute{P}_{k\alpha,k\beta,ki,kj}$}, we can change the size of matrices $\tilde{G}$ and $\dot{D}_{2,m}^{-1}$.

\begin{discuss}
{\colorr 最初から$p_i\sim p_+$の近似を使ってLemma~\ref{D1-change}も示せば良かったように見えるが, $\Lambda_1$の評価は$\dot{D}^{-1}_{j,m}$じゃないとできない}
\end{discuss}

Now we verify the above idea. First, we see that the terms involving small $\alpha_q$ or $\beta_q$ in (\ref{Hn-lim-lemma2-eq1}) can be ignored.
Let $\eta\in (0,1/2)$ be the one in $[A2]$, $\delta\in (1/2,1)$ such that $b_n^{\delta}k_n^{-1}\to 0$, $\tilde{s}_{l'}=s_{m-1}+T[b_nk_n^{-1}]^{-1}[k_nb_n^{-\eta}]^{-1}((l'+[b_n^{\delta -\eta}])\wedge [k_nb_n^{-\eta}])$ for $0\leq l'\leq ([k_nb_n^{-\eta}]-[b_n^{\delta-\eta}])\vee 0$, 
$\dot{D}_{3,m}=(c_1\wedge c_2)(v_{1,\ast}\wedge v_{2,\ast})\mathcal{E}_{k^1_m\vee k^2_m}+(v_{1,\ast}\wedge v_{2,\ast})M(k^1_m\vee k^2_m)$, $G'=\{|I^1_{i,m}\cap I^2_{j,m}|1_{\{\inf I^1_{i,m}\wedge \inf I^2_{j,m}<\tilde{s}_0\} }\}_{1\leq i,j\leq k^1_m\vee k^2_m}$,
$\hat{G}=\{|I^1_{i,m}\cap I^2_{j,m}|1_{\{i\leq k^1_m \ {\rm and } \ j\leq k^2_m\} }\}_{1\leq i,j\leq k^1_m\vee k^2_m}$, and
$\mathcal{E}''=\{\delta_{ij}1_{\{\inf I^1_{i,m}\wedge \inf I^2_{i,m}<\tilde{s}_0\} }\}_{1\leq i,j\leq k^1_m\vee k^2_m}$.
Similarly to the proof of Lemma~\ref{D1-change}, the absolute value $\Lambda_1$ of a summation involving the terms with $(\alpha_q,\beta_q)$ satisfying $\inf I^1_{\alpha_q,m}\wedge \inf I^2_{\beta_q,m}<\tilde{s}_0$
is less than  
$pb_n^{-1/2}{\rm tr}(\dot{D}_{3,m}^{-1}(G'\dot{D}_{3,m}^{-1}\hat{G}^{\top}+\hat{G}\dot{D}_{3,m}^{-1}(G')^{\top})(\dot{D}_{3,m}^{-1}\hat{G}\dot{D}_{3,m}^{-1}\hat{G}^{\top})^{p-1})$.
\begin{discuss}
{\colorr この評価は(\ref{Hn-lim-lemma2-eq1})の表現を$\dot{D}_{3,m}^{-1}$の場合で考えると分子が全てnonnegativeであることからわかる.}
\end{discuss}
Lemma 3 in~\cite{ogi-yos14} implies $\lVert G'+(G')^{\top}\rVert\leq 2r_n$ and hence all the eigenvalues of $G'+(G')^{\top}$ are greater than or equal to $-2r_n$. Therefore 
$G'+(G')^{\top}+2r_n\mathcal{E}''$ is nonnegative definite, and hence Lemma~\ref{tr-est} yields
\begin{eqnarray}
\Lambda_1&\leq &p(r_n/\underbar{r}_n)^{2p-1}b_n^{-1/2}{\rm tr}(\dot{D}_{3,m}^{-1}(G'+(G')^{\top}+2r_n\mathcal{E}''))
\leq 2p(r_n/\underbar{r}_n)^{2p-1}b_n^{-1/2}({\rm tr}(\dot{D}_{3,m}^{-1}G')+r_n{\rm tr}(\dot{D}_{3,m}^{-1}\mathcal{E}'')). \nonumber 
\end{eqnarray}
\begin{discuss}
{\colorr モーメント評価の時は$\sum_p$の評価をして$b_n^{-\epsilon}$の隙間を作れることを示す必要があるが, 後ろの議論と同様にして, $(\dot{D}_{3,m}^{-1}\hat{G}\dot{D}_{3,m}^{-1}\hat{G}^{\top})^{p-1}$の中の
$\hat{G}$を$\tilde{G}$に戻すことで評価の中の$(r_n/\underbar{r}_n)^p$を消せる. ($i',j',\alpha',i(\alpha'),j(\alpha')$などはそのまま採用するが,スタートの時刻は$b_n^{\delta}$よりも早い)}
\end{discuss}
Let ($\dot{G})_{i,j}=(\sum_{l\leq i}G'_{l,i}+\sum_{m<i}G'_{i,m})\delta_{i,j}$ and $\dot{k}=\max\{i;\dot{G}_{ii}>0\}$. Then Lemma~\ref{ddotD-properties} yields
\begin{eqnarray}
{\rm tr}(\dot{D}_{3,m}^{-1}G')&=&\frac{1}{v_{1,\ast}\wedge v_{2,\ast}}\sum_i\sum_{\alpha,\beta\leq i}\frac{G'_{\alpha,\beta}}{p_{\alpha}\cdots p_ip_{\beta}\cdots p_{i-1}} 
\leq \frac{1}{v_{1,\ast}\wedge v_{2,\ast}}\sum_i\sum_{\alpha\leq i}\frac{\dot{G}_{\alpha,\alpha}}{p_{\alpha}\cdots p_ip_{\alpha}\cdots p_{i-1}} \nonumber \\
&=&{\rm tr}(\dot{D}_{3,m}^{-1}\dot{G})\leq (\dot{D}_{3,m}^{-1})_{\dot{k},\dot{k}}(\tilde{s}_0-s_{m-1}+r_n) 
\leq \frac{Cb_n^{-1+\delta}+r_n}{[(k^1_m\vee k^2_m)/2]-\dot{k}}{\rm tr}(\dot{D}_{3,m}^{-1}). \nonumber
\end{eqnarray}
Therefore we obtain $\Lambda_1=\dot{R}_n(\ell_n^{-1})$, {\colord and $\Lambda_1=\underbar{R}_n(\ell_n^{-1})$ if $[B2]$ is satisfied}.

Let $\ddot{D}_{2,m}={\colorlg v_{2,\ast}}c_2(\hat{a}_m^2/\hat{a}_m^1)^2\mathcal{E}+v_{2,\ast}M_{1,m}$, $i(\alpha')=\min\{i;S^{n,1}_i\geq \tilde{s}_{\alpha'-1}\}$, and $j(\alpha')=\min\{j;S^{n,2}_j\geq \tilde{s}_{\alpha'-1}\}$. 
{\colord We will show that $b_n^{-1/2}{\rm tr}((\dot{D}_{1,m}^{-1}\tilde{G}\dot{D}_{2,m}^{-1}\tilde{G}^{\top})^p)$ is approximated by $b_n^{-1/2-2p}(\hat{a}_m^2)^p(\hat{a}_m^1)^{-3p}{\rm tr}((\dot{D}_{1,m}^{-1}\ddot{D}_{2,m})^p)$.}
A similar argument to the proof of Lemma~\ref{D1-change} yields $|\tilde{P}_{\alpha,\beta,i,j}/\tilde{P}_{i(\alpha'),j(\alpha'),i(i'),j(j')}-1|=\dot{R}_n(1)$
for $i(\alpha')\leq \alpha<i(\alpha'+1)$, $j(\alpha')\leq \beta<j(\alpha'+1)$, $i(i')\leq i<i(i'+1)$, and $j(j')\leq j<j(j'+1)$. 
Therefore repeated use of $[A2]$ yields 
\begin{eqnarray}\label{Lambda1-equ}
&&b_n^{-1/2}{\rm tr}((\dot{D}_{1,m}^{-1}\tilde{G}\dot{D}_{2,m}^{-1}\tilde{G}^{\top})^p) \nonumber \\
&=&b_n^{-1/2}\frac{\mathcal{T}^{n,p}_{m,4}}{v_{1,\ast}^pv_{2,\ast}^p}\sum_{\substack{i'_1,\cdots,i'_p \\ j'_1,\cdots,j'_p}} \sum_{\substack{\alpha'_{2q-1}\leq i'_q\wedge j'_q \\ \alpha'_{2q}\leq i'_{q+1}\wedge j'_q \ (1\leq q\leq p)}}
\prod_{q=1}^p\frac{(Tb_n^{-1+\eta})^2\# \{i_q;I^1_{i_q,m}\subset [\tilde{s}_{i'_q-1},\tilde{s}_{i'_q})\}\# \{j_q;I^2_{j_q,m}\subset [\tilde{s}_{j'_q-1},\tilde{s}_{j'_q})\}}{\tilde{P}_{i(\alpha'_{2q-1}),j(\alpha'_{2q-1}),i(i'_q),j(j'_q)}\tilde{P}_{i(\alpha'_{2q}),j(\alpha'_{2q}),i(i'_{q+1}),j(j'_q)}} \nonumber \\
&&+\dot{R}_n(l_n^{-1}) \nonumber \\
&=&b_n^{-1/2}\frac{\mathcal{T}^{n,p}_{m,5}(Tb_n^{-1+\eta})^{2p}(\hat{a}_m^2)^p}{v_{1,\ast}^pv_{2,\ast}^p(\hat{a}_m^1)^p} \nonumber \\
&&\times \sum_{\substack{i'_1,\cdots,i'_p \\ j'_1,\cdots,j'_p}} \sum_{\substack{\alpha'_{2q-1}\leq i'_q\wedge j'_q \\ \alpha'_{2q}\leq i'_{q+1}\wedge j'_q \ (1\leq q\leq p)}}
\prod_{q=1}^p\frac{\#\{\alpha;I^1_{\alpha,m}\subset [\tilde{s}_{\alpha'_{2q-1}-1},\tilde{s}_{\alpha'_{2q-1}})\}\#\{\alpha;I^1_{\alpha,m}\subset [\tilde{s}_{\alpha'_{2q}-1},\tilde{s}_{\alpha'_{2q}})\} }{\tilde{P}_{i(\alpha'_{2q-1}),j(\alpha'_{2q-1}),i(i'_q),j(j'_q)}\tilde{P}_{i(\alpha'_{2q}),j(\alpha'_{2q}),i(i'_{q+1}),j(j'_q)}}
+\dot{R}_n(l_n^{-1}), \nonumber \\
\end{eqnarray}
{\colord where $\mathcal{T}^{n,p}_{m,i}$ is a random variable satisfying} $\sup_{\sigma,m}|\mathcal{T}^{n,p}_{m,i}-1|=\dot{R}_n(1)$ for $i=4,5$.
\begin{discuss}
{\colorr $l,m$が異なるブロックで$\tilde{G}_{l,m}>0$となるものがあるが, $\Lambda_1$の議論と同様m $b_n^{\eta}$間隔で$0$以外が現れる行列を$b_n^{-\eta}$で上から抑えればいい.
($\{(\dot{D}_{j,m})_{k,k}\}_{k=[k_m^j/2]+1}^{k_m^j}$の単調減少性も使う.)
もしくは分子が高々$r_n$で分母の違いも無視できるから$\mathcal{T}^{n,p}_{m,4}$に入れられる.
また, $l'=i'\wedge j'$の時, 項の数は分子と等しくならないが, 分子で抑えられていて, $i'\wedge j'-[b_n^{1/4-\eta/2}]\leq l'\leq i'\wedge j'$となる$l'$に対し,
分母分子が同じオーダーなので, $l'=i'\wedge j'$の項は無視できる. $a^j_t$の係数ずらしがヘルダー連続性より$\dot{R}_n$で言える. 

モーメント評価の時は二つ目の等式を得る時に$(1+x)^p-1\leq px(1+x)^{p-1}$を使う. $x$には[A2]の収束量が入る. $\sup|\mathcal{T}^{n,p}_{m,i}-1|\leq p(1+x)^{p-1}\dot{R}_n(1)$となる.
}
\end{discuss}

Since Lemma~\ref{ddotD-properties} and $[A2]$ yield
\begin{eqnarray}
p_{j(\alpha)}\cdots p_{j(\beta)-1}(c_2)&=&(p_+(c_2))^{j(\beta)-j(\alpha)}(1+\dot{R}_n(1))=\exp((b_n\hat{a}^2_m(\tilde{s}_{\beta}-\tilde{s}_{\alpha})+\dot{R}_n(b_n^{1/2}))\log p_+(c_2))(1+\dot{R}_n(1)) \nonumber \\
&=&\exp(\hat{a}^2_m(\hat{a}^1_m)^{-1}(i(\beta)-i(\alpha)))\log p_+(c_2))(1+\dot{R}_n(1))
=p_{i(\alpha)}\cdots p_{i(\beta)-1}({\colorlg c'_2})(1+\dot{R}_n(1)), \nonumber
\end{eqnarray}
\begin{discuss}
{\colorr $j(\beta)-j(\alpha)\sim \sum_{\gamma=\alpha+1}^{\beta}\hat{a}^2_m(\tilde{s}_{\gamma-1})b_n(\tilde{s}_{\gamma}-\tilde{s}_{\gamma-1})$なので, 
$\sum_{\gamma}(\hat{a}^2_m(\tilde{s}_{\gamma-1})-\hat{a}^2_m(\tilde{s}_{\alpha}))b_n(\tilde{s}_{\gamma}-\tilde{s}_{\gamma-1})=o_p(b_n^{1/2})$が必要より, $b_n^{-1/2}k_n(b_n^{-1}k_n)^{\dot{\eta}}\to 0$を使う.
また, $\tilde{s}_{\beta}-\tilde{s}_{\alpha}$は最大で$k_n$オーダーだからここで$[A2]$の$k_nb_n^{-1/2}$の条件を使う.
$c\log(1+\sqrt{c_2})=\log(1+\sqrt{c^2c_2})+O_p(b_n^{-1})$. $a_t$のヘルダー連続性も使う.}
\end{discuss}
we may replace $\tilde{P}_{i(\alpha'_{2q-1}),j(\alpha'_{2q-1}),i(i'_q),j(j'_q)}$ and $\tilde{P}_{i(\alpha'_{2q}),j(\alpha'_{2q}),i(i'_{q+1}),j(j'_q)}$ in the right-hand side of (\ref{Lambda1-equ}) by \\
$\hat{P}_{i(\alpha'_{2q-1}),i(\alpha'_{2q-1}),i(i'_q),i(j'_q)}$ and $\hat{P}_{i(\alpha'_{2q}),i(\alpha'_{2q}),i(i'_{q+1}),i(j'_q)}$, respectively, 
{\colorlg where $\hat{P}_{\alpha,\beta,i,j}=\prod_{k_1=\alpha}^{i-1}p_{k_1}(c_1)\prod_{k_2=\beta}^{j-1}p_{k_2}(c'_2)$.}
Therefore, we obtain 
\begin{equation*}
\sup_{\sigma,m}\bigg|b_n^{-1/2}{\rm tr}((\dot{D}_{1,m}^{-1}\tilde{G}\dot{D}_{2,m}^{-1}\tilde{G}^{\top})^p)-b_n^{-1/2-2p}(\hat{a}_m^2)^p(\hat{a}_m^1)^{-3p}{\rm tr}((\dot{D}_{1,m}^{-1}\ddot{D}_{2,m}^{-1})^p)\bigg|=o_p(\ell_n^{-1}),
\end{equation*}
by a similar argument to (\ref{Lambda1-equ}). 
\begin{discuss}
{\colorr $\Lambda_1$型の評価も$b_n^{-2p}$があるからできる.}
\end{discuss}
Since {\colorlg $\partial_{\sigma}^l\dot{D}_{j,m}=\partial_{\sigma}^lc_jv_{j,\ast}\mathcal{E}$} for $1\leq l\leq 3$, we similarly obtain
\begin{equation}\label{key-eq}
{\colorlg \sup_{\sigma,m}\bigg|b_n^{-1/2}\partial_{\sigma}^k{\rm tr}((\dot{D}_{1,m}^{-1}\tilde{G}\dot{D}_{2,m}^{-1}\tilde{G}^{\top})^p)
-b_n^{-1/2-2p}(\hat{a}_m^2)^p(\hat{a}_m^1)^{-3p}\partial_{\sigma}^k{\rm tr}((\dot{D}_{1,m}^{-1}\ddot{D}_{2,m}^{-1})^p)\bigg|=o_p(\ell_n^{-1})}.
\end{equation}
{\colorlg Then (\ref{tr-lim-est2}), (\ref{Hn-lim-lemma2-eq0}) and (\ref{key-eq}) yield (\ref{tr-eq3}).}

\begin{discuss}
{\colorr ($j'$の和を$i'$に直すときに$a^2/a^1$, $\mathcal{E}_{k^2_m}$を$\mathcal{E}_{k^1_m}$に直すときに$a^2/a^1$.)

$\dot{D}_{2,m}^{-1}$を微分すると, $(\hat{a}^2_m)^2/(\hat{a}^1_m)^2$が余計に出てくるがそれが$\partial_{\sigma}c'_2$に吸収される.
同様にして,
\begin{equation*}
b_n^{-1/2}{\rm tr}(\partial_{\sigma}^k((\dot{D}_{1,m}^{-1}\tilde{G}\dot{D}_{2,m}^{-1}\tilde{G}^{\top})^p\dot{D}_{1,m}^{-1}D'_{1,m}))
=b_n^{-1/2-2p-1}(\hat{a}_m^2)^p(\hat{a}_m^1)^{-3p-1}{\rm tr}(\partial_{\sigma}^k((\dot{D}_{1,m}^{-1}\ddot{D}_{2,m}^{-1})^p\dot{D}_{1,m}^{-1})) + o_p(\ell_n^{-1})
\end{equation*}
and 
\begin{equation*}
b_n^{-1/2}{\rm tr}(\partial_{\sigma}^k((\dot{D}_{2,m}^{-1}\tilde{G}^{\top}\dot{D}_{1,m}^{-1}\tilde{G})^p\dot{D}_{2,m}^{-1}D'_{2,m}))
=b_n^{-1/2-2p-1}(\hat{a}_m^2)^{p+1}(\hat{a}_m^1)^{-3p-2}{\rm tr}(\partial_{\sigma}^k((\dot{D}_{1,m}^{-1}\ddot{D}_{2,m}^{-1})^p\ddot{D}_{2,m}^{-1})) + o_p(\ell_n^{-1})
\end{equation*}
($\tilde{G}$を$\mathcal{E}$に直すときに$b_n^{-1}(\hat{a}^1_m)^{-1}$が出てくる. $\dot{D}_{2,m}^{-1}$を$\ddot{D}_{2,m}^{-1}$に直すときに$\hat{a}^2_m(\hat{a}^1_m)^{-1}$がでてくる)
and $o_p(\ell_n^{-1})$ in the right-hand side of (\ref{})-(\ref{}) can be replaced by $\underbar{R}_n(\ell_n^{-1})$ if $[B2]$ is satisfied.
}
\end{discuss}

\begin{discuss}
{\colorr

(\ref{tr-lim-equ})の最初の等式は, ${\rm tr}((\dot{D}_1^{-1}\ddot{D}_2^{-1})^p)=\cdots $が恒等式だから両辺を微分できてOK.
二つ目の等式は, $\Sigma$を積分に直しただけで, 被積分関数の微分が$x$に関して連続で同じオーダーの上からの評価が成り立つからOK.
三つ目の等式は, $I_p$の等式の両辺を$\sigma$で微分した時, $a$の微分が$\sqrt{a}^{-1}$に一つでもかからなかったらオーダーが下がり, $\sigma$の微分も$\sqrt{4+a}^{-1}$にかかると$\partial_{\sigma}a=O_p(b_n^{-1})$がでて
オーダーがさがるので, $b_n^{p-1/2}\partial_{\sigma}I_p(a)=(2p-3)!!\pi2^{-p}/(p-1)!\partial_{\sigma}^k(b_na)^{-p+1/2}$よりOK.
四つ目の等式は, 恒等式だから微分できる.
}
\end{discuss}

We also have (\ref{tr-eq1}) and (\ref{tr-eq2}) by a similar argument. 

{\colord Similar arguments enable us to replace $o_p(b_n^{1/2}\ell_n^{-1})$ by $\underbar{R}_n(b_n^{1/2}\ell_n^{-1})$ in (\ref{tr-eq1})-(\ref{tr-eq3}) if $[B2]$ is satisfied.}

\end{proof}

\noindent
{\bf Proof of Proposition~\ref{Hn-lim}.}

We first prove the results under the additional condition $[A1']$.
\begin{discuss}
{\colorr $[B1]$で$\mu_t$に仮定したHolder連続性の条件などは$[A1]$を局所化しただけではでないから, $\mu_t\equiv 0$とするため$[A1']$を使う.
($\mu_t\equiv 0$とするためには係数は多項式増大だけでなく, boundedでないといけない)
}
\end{discuss}

{\colorlg
Since $\lVert \tilde{D}_{1,m}^{-1/2}\tilde{G}\tilde{D}_{2,m}^{-1}\tilde{G}^{\top}\tilde{D}_{1,m}^{-1/2}\rVert\leq |\tilde{b}^1_m|^{-2}|\tilde{b}^2_m|^{-2}$ by Lemma 3 in \cite{ogi-yos14},
for any $\epsilon,\delta>0$, there exists $P_1\in\mathbb{N}$ such that
\begin{equation*}
\sup_nP\bigg[\sup_{\sigma}b_n^{-\frac{1}{2}}\sum_m\sum_{p=P+1}^{\infty}\big|\partial_{\sigma}^k\big((\tilde{b}^1_m\cdot \tilde{b}^2_m)^{2p-1}{\rm tr}((\tilde{D}_{1,m}^{-1}\tilde{G}\tilde{D}_{2,m}^{-1}\tilde{G}^{\top})^p)\big)\big|
\geq \delta\bigg]<\epsilon,
\end{equation*}
\begin{equation*}
\sup_nP\bigg[\sup_{\sigma}b_n^{-\frac{1}{2}}\sum_m\sum_{p=P+1}^{\infty}
\bigg|\partial_{\sigma}^k\bigg(\frac{(\tilde{b}^1_m\cdot\tilde{b}^2_m)^{2p-1}\hat{a}_m^1k_n(\hat{a}_m^2)^pI_{p,p}(c_1,c'_2)}{(b_n^2(\hat{a}_m^1)^3v_{1,\ast}v_{2,\ast})^p}\bigg)\bigg|\geq \delta\bigg]<\epsilon
\end{equation*}
for $P\geq P_1$.
\begin{discuss}
{\colorr $b_n^{-1/2}{\rm tr}(\tilde{D}_m^{-1}(\tilde{S}_{m,\ast}-\tilde{S}_m))=b_n^{-1/2}{\rm tr}(\dot{D}_m^{-1}(\tilde{S}_{m,\ast}-\tilde{S}_m))+o_p(\ell_n^{-1})=b_n^{-3/2}\sum_jc_j{\rm tr}(\dot{D}_{j,m}^{-1})+o_p(\ell_n^{-1})$と
$\tilde{D}_{1,m}^{-1/2}\tilde{L}\tilde{D}_{2,m}^{-1/2}$の$\sigma$微分がおとなしいことも使う(??).}
\end{discuss}
Together with Lemma \ref{Hn-lim-lemma2}, we obtain
\begin{equation*}
\sup_{\sigma}\bigg|b_n^{-1/2}\partial_{\sigma}^k\sum_m
\sum_{p=1}^{\infty}(\tilde{b}^1_m\cdot \tilde{b}^2_m)^{2p-1}{\rm tr}((\tilde{D}_{1,m}^{-1}\tilde{G}\tilde{D}_{2,m}^{-1}\tilde{G}^{\top})^p)
- \frac{T\hat{a}_m^1k_n}{\pi b_n^{1/2}}\partial_{\sigma}^k\sum_m\sum_{p=1}^{\infty}\frac{(\tilde{b}_m^1\cdot \tilde{b}_m^2)^{2p-1}(\hat{a}_m^2)^pI_{p,p}(c_1,c'_2)}{(b_n^2(\hat{a}_m^1)^3v_{1,\ast}v_{2,\ast})^p}
\bigg|\to^p0.
\end{equation*}

Let $\dot{a}_m^j=\tilde{a}_{s_{m-1}}^j$, $\mathfrak{C}_m=|\tilde{b}^1_m|^2|\tilde{b}^2_m|^2-(\tilde{b}^1_m\cdot \tilde{b}^2_m)^2$,
$\mathfrak{A}_t=\varphi(\tilde{a}^1_t|b^1_t|^2+\tilde{a}^2_t|b^2_t|^2,\tilde{a}^1_t\tilde{a}^2_t\det(b_tb_t^{\top}))$,
and 
\begin{equation*}
P_n=\varphi(c_1+c'_2,b_n^{-2}\hat{a}_m^2(\hat{a}_m^1)^{-3}v_{1,\ast}^{-1}v_{2,\ast}^{-1}\mathfrak{C}_m)
=b_n^{-1/2}(\hat{a}_m^1)^{-1}\varphi(\dot{a}^1_m|\tilde{b}^1_m|^2+\dot{a}^2_m|\tilde{b}^2_m|^2,\dot{a}^1_m\dot{a}^2_m\mathfrak{C}_m).
\end{equation*}
Then Lemma \ref{sum-integral-lemma} yields
\begin{eqnarray}
&&\frac{T\hat{a}_m^1k_n}{\pi b_n^{1/2}}\partial_{\sigma}^k\sum_m\sum_{p=1}^{\infty}\frac{(\tilde{b}^1_m\cdot\tilde{b}^2_m)^{2p-1}(\hat{a}_m^2)^pI_{p,p}(c_1,c'_2)}{(b_n^2(\hat{a}_m^1)^3v_{1,\ast}v_{2,\ast})^p} \nonumber \\
&=&\partial_{\sigma}^k\sum_m\frac{Tb_n^{1/2}\hat{a}_m^1\ell_n^{-1}\frac{\hat{a}_m^2\tilde{b}_m^1\cdot \tilde{b}_m^2}{(\hat{a}_m^1)^3b_n^2v_{1,\ast}v_{2,\ast}}}
{b_n^{-1/2}(\hat{a}_m^1)^{-1}\sqrt{2}\varphi(\dot{a}^1_m|\tilde{b}_m^1|^2+\dot{a}^2_m|\tilde{b}_m^2|^2,\dot{a}^1_m\dot{a}^2_m\mathfrak{C}_m)
\sqrt{\dot{a}^1_m\dot{a}^2_m}\sqrt{\mathfrak{C}_m}b_n^{-1}(\hat{a}_m^1)^{-2}}+o_p(1) \nonumber \\
&=&\partial_{\sigma}^k\sum_m\frac{T\ell_n^{-1}\sqrt{\dot{a}^1_m\dot{a}^2_m}\tilde{b}_m^1\cdot \tilde{b}_m^2}{\sqrt{2\mathfrak{C}_m}
\varphi(\dot{a}^1_m|\tilde{b}_m^1|^2+\dot{a}^2_m|\tilde{b}_m^2|^2,\dot{a}^1_m\dot{a}^2_m\mathfrak{C}_m)}+o_p(1) 
= \partial_{\sigma}^k \int^T_0\frac{\sqrt{\tilde{a}^1_t\tilde{a}^2_t}b_t^1\cdot b_t^2}{\sqrt{2\det (b_tb_t^{\top})}\mathfrak{A}_t}dt +o_p(1). \nonumber
\end{eqnarray}

Therefore, we have
\begin{equation}\label{Hn-lim-prop-eq1}
\sup_{\sigma}\bigg|b_n^{-1/2}\partial_{\sigma}^k\sum_m
\sum_{p=1}^{\infty}(\tilde{b}^1_m\cdot \tilde{b}^2_m)^{2p-1}{\rm tr}((\tilde{D}_{1,m}^{-1}\tilde{G}\tilde{D}_{2,m}^{-1}\tilde{G}^{\top})^p)
- \partial_{\sigma}^k \int^T_0\frac{\sqrt{\tilde{a}^1_t\tilde{a}^2_t}b_t^1\cdot b_t^2}{\sqrt{2\det (b_tb_t^{\top})}\mathfrak{A}_t}dt
\bigg|\to^p0.
\end{equation}

Similarly, we obtain
\begin{equation}\label{Hn-lim-prop-eq15}
\sup_{\sigma}\bigg|b_n^{-1/2}\partial_{\sigma}^k\sum_m
\sum_{p=1}^{\infty}{\rm tr}((\tilde{D}_{1,m}^{-1}\tilde{L}\tilde{D}_{2,m}^{-1}\tilde{L}^{\top})^p\tilde{D}_{1,m}^{-1}D'_{1,m})
- \partial_{\sigma}^k \int^T_0\frac{|b^2_t|^2\sqrt{\tilde{a}^1_t\tilde{a}^2_t}+\tilde{a}^1_t\sqrt{\det(b_tb_t^{\top})}}{\sqrt{2\det (b_tb_t^{\top})}\mathfrak{A}_t}dt\bigg|\to^p0, 
\end{equation}
\begin{equation}\label{Hn-lim-prop-eq18}
\sup_{\sigma}\bigg|b_n^{-1/2}\partial_{\sigma}^k\sum_m
\sum_{p=1}^{\infty}{\rm tr}((\tilde{D}_{2,m}^{-1}\tilde{L}^{\top}\tilde{D}_{1,m}^{-1}\tilde{L})^p\tilde{D}_{2,m}^{-1}D'_{2,m})
- \partial_{\sigma}^k \int^T_0\frac{|b^1_t|^2\sqrt{\tilde{a}^1_t\tilde{a}^2_t}+\tilde{a}^2_t\sqrt{\det(b_tb_t^{\top})}}{\sqrt{2\det (b_tb_t^{\top})}\mathfrak{A}_t}dt\bigg|\to^p0.
\end{equation}

\begin{discuss}
{\colorr 
\begin{eqnarray}
b_n^{-1/2}{\rm tr}(\tilde{D}_{j,m}^{-1}D'_{j,m})&=&b_n^{-3/2}(\hat{a}_m^j)^{-1}{\rm tr}(\dot{D}_{j,m}^{-1})+o_p(\ell_n^{-1})=\frac{b_n^{-3/2}k^j_m}{2c_{j,\ast}^{1/2}\hat{a}_m^jv_{j,\ast}}+o_p(\ell_n^{-1})
=\frac{T\ell_n^{-1}(\hat{a}_m^j)^{1/2}}{2|\tilde{b}^j_m|v_{j,\ast}^{1/2}}+o_p(\ell_n^{-1}). \nonumber
\end{eqnarray}
和を積分に直すところは全て$a^j_t$と$b^j_t$の連続性と, 被和項のtightnessから言える. また, $a^j_t$, $b^j_t$のHolder連続性より, $[B2]$の時に$b_n^{\delta}$の隙間をとれることもわかる.
}
\end{discuss}

Furthermore, Lemma~\ref{log-det-exp} and a similar argument yield
\begin{eqnarray}
\partial_{\sigma}^k\log \det (\tilde{S}_m\tilde{D}_m^{-1})
&=&\partial_{\sigma}^k\log\det \bigg(\mathcal{E}+\left(
\begin{array}{ll}
0 & \tilde{D}_{1,m}^{-1/2}\tilde{L}\tilde{D}_{2,m}^{-1/2} \\
\tilde{D}_{2,m}^{-1/2}\tilde{L}^{\top}\tilde{D}_{1,m}^{-1/2} & 0
\end{array}
\right)\bigg) \nonumber \\
&=&-\sum_{p=1}^{\infty}\frac{1}{p}\partial_{\sigma}^k{\rm tr}((\tilde{D}_{1,m}^{-1}\tilde{L}\tilde{D}_{2,m}^{-1}\tilde{L}^{\top})^p) \nonumber \\
&=&-\frac{T\hat{a}_m^1k_n}{\pi}\sum_{p=1}^{\infty}\frac{(\hat{a}_m^2)^p(\tilde{b}_m^1\cdot \tilde{b}_m^2)^{2p}I_{p,p}(c_1,c'_2)}{p(\hat{a}_m^1)^{3p}b_n^{2p}v_{1,\ast}^pv_{2,\ast}^p}
+ o_p(b_n^{1/2}\ell_n^{-1}). \nonumber 
\end{eqnarray}
Then Lemma~\ref{sum-integral-lemma} yields
\begin{eqnarray}\label{Hn-lim-prop-eq2}
&&\partial_{\sigma}^k\log \det (\tilde{S}_m\tilde{D}_m^{-1}) \nonumber \\
&=&-T\hat{a}_m^1k_n\bigg(\sqrt{\frac{|\tilde{b}_m^1|^2}{b_n\hat{a}_m^1v_{1,\ast}}}+\sqrt{\frac{\hat{a}_m^2|\tilde{b}_m^2|^2}{b_n(\hat{a}_m^1)^2v_{2,\ast}}}\bigg)
+\frac{T\hat{a}_m^1k_nb_n^{-1/2}}{\sqrt{2}\hat{a}_m^1}
\varphi(\dot{a}^1_m|\tilde{b}_m^1|^2+\dot{a}^2_m|\tilde{b}_m^2|^2,\dot{a}^1_m\dot{a}^2_m\det(\tilde{b}_m\tilde{b}_m^{\top})) \nonumber \\
&&+o_p(b_n^{1/2}\ell_n^{-1}) \nonumber \\
&=&Tb_n^{1/2}\ell_n^{-1}\bigg(\frac{1}{\sqrt{2}}\varphi(\dot{a}^1_m|\tilde{b}_m^1|^2+\dot{a}^2_m|\tilde{b}_m^2|^2,\dot{a}^1_m\dot{a}^2_m\det(\tilde{b}_m\tilde{b}_m^{\top}))
-\sqrt{\dot{a}^1_m|\tilde{b}_m^1|^2}-\sqrt{\dot{a}^2_m|\tilde{b}_m^2|^2}\bigg)+o_p(b_n^{1/2}\ell_n^{-1}). \nonumber \\
\end{eqnarray}
}{\colord
Moreover, Lemmas~\ref{log-det-exp} and~\ref{D1-change} yield
\begin{eqnarray}\label{Hn-lim-prop-eq3}
&&\partial_{\sigma}^k\log\det (\tilde{D}_{j,m}\tilde{D}_{j,m,\ast}^{-1}) \nonumber \\
&=&\partial_{\sigma}^k\log\det (\mathcal{E}+\tilde{D}_{j,m,\ast}^{-1/2}(\tilde{D}_{j,m}-\tilde{D}_{j,m,\ast})\tilde{D}_{j,m,\ast}^{-1/2})
=\sum_{p=1}^{\infty}\frac{(-1)^{p-1}}{p}\partial_{\sigma}^k{\rm tr}((\tilde{D}_{j,m,\ast}^{-1}(\tilde{D}_{j,m}-\tilde{D}_{j,m,\ast}))^p) \nonumber \\
&=&\sum_{p=1}^{\infty}\frac{(-1)^{p-1}}{p}\partial_{\sigma}^k{\rm tr}((\dot{D}_{j,m,\ast}^{-1}(\dot{D}_{j,m}-\dot{D}_{j,m,\ast}))^p)+o_p(b_n^{1/2}\ell_n^{-1})=\partial_{\sigma}^k\log\det (\dot{D}_{j,m} \dot{D}_{j,m,\ast}^{-1})+o_p(b_n^{1/2}\ell_n^{-1}) \nonumber \\
\end{eqnarray}
when $|\tilde{b}^j_{m,\ast}|\geq |\tilde{b}^j_m|$, where $\tilde{D}_{j,m,\ast}$ and $\dot{D}_{j,m,\ast}$ are obtained by substituting $\sigma=\sigma_{\ast}$ in $\tilde{D}_{j,m}$ and $\dot{D}_{j,m}$, respectively. 
Similarly, we have $\partial_{\sigma}^k\log \det (\tilde{D}_{j,m}\tilde{D}_{j,m,\ast}^{-1})=\partial_{\sigma}^k\log \det (\dot{D}_{j,m}\dot{D}_{j,m,\ast}^{-1})+o_p(b_n^{1/2}\ell_n^{-1})$ when $|\tilde{b}^j_{m,\ast}|< |\tilde{b}^j_m|$.
}

On the other hand, {\colord results in Section~\ref{noise-cov-property-subsection} yield}
\begin{eqnarray}\label{Hn-lim-prop-eq4}
\partial_{\sigma}^k\log\frac{\det \dot{D}_{j,m}}{\det \dot{D}_{j,m,\ast}}&=&\frac{k^j_m}{\pi}\partial_{\sigma}^k\int^{\pi}_0\log \frac{c_j+2(1-\cos x)}{c_{j,\ast}+2(1-\cos x)}dx+o_p(b_n^{1/2}\ell_n^{-1}) \nonumber \\
&=&2k^j_m\partial_{\sigma}^k\log\frac{\sqrt{c_j}+\sqrt{4+c_j}}{\sqrt{c_{j,\ast}}+\sqrt{4+c_{j,\ast}}} +o_p(b_n^{1/2}\ell_n^{-1}) \nonumber \\
&=& k^j_m\partial_{\sigma}^k(\sqrt{c_j}-\sqrt{c_{j,\ast}})+o_p(b_n^{1/2}\ell_n^{-1})=Tb_n^{1/2}\ell_n^{-1}\sqrt{\dot{a}_m^j}\partial_{\sigma}^k(|b^j_m|-|b^j_{m,\ast}|)+o_p(b_n^{1/2}\ell_n^{-1}). \nonumber \\
\end{eqnarray}
\begin{discuss}
{\colorr 積分と和の差はlogを差に分解して評価すればよい.}
\end{discuss}
{\colord The residuals are bounded uniformly with respect to $\sigma$ and $m$.}
Then we obtain $\sup_{\sigma}|b_n^{-1/2}\partial_{\sigma}^k(H_n(\sigma,\hat{v}_n)-H_n(\sigma_{\ast},\hat{v}_n))- \partial_{\sigma}^k\mathcal{Y}_1(\sigma)|\to^p0$ as $n\to\infty$ for any $\sigma\in\Lambda$ and $0\leq k\leq 3$
by (\ref{Hn-diff-eq1}) and (\ref{Hn-lim-prop-eq1})--(\ref{Hn-lim-prop-eq4}).
\begin{discuss}
{\colorr 微分は$\log $分子の状態で計算して評価.}
\end{discuss}

Finally, we obtain the results without $[A1']$ by using the arguments in Proposition 3.1 of Gloter and Jacod~\cite{glo-jac01b}.

\qed

\begin{discuss}
{\colorr
証明:
条件$[A1'']$を条件$[A1]$に加え, $\sup_{t,x,\sigma}\lVert (bb^{\top})^{-1}\rVert(t,x,\sigma) <\infty$かつ
$0\leq 2i+j\leq 3$と$0\leq k\leq 4$に対し, $Y_0$と$\partial_t^i\partial_x^j\partial_{\sigma}^kb$と$\mu_t$が有界とすると,
通常の局所化の議論から$[A1]$を$[A1'']$に置き換えて証明してよいことがわかる.

$[A1'']$の下, $\int^T_0|b^{\top}(bb^{\top})^{-1}\mu_s|^2ds=\int^T_0\mu_s^{\top}(bb^{\top})^{-1}\mu_sds$は有界より,
Girsanovの定理から$W'_t=W_t+\int^t_0b^{\top}_s(b_sb^{\top}_s)^{-1}\mu_sds$は$dQ=\exp(\int^T_0\mu_t^{\top}(b_tb^{\top}_t)^{-1}b_tdW_t-\int^T_0\mu_t^{\top}(b_tb_t^{\top})^{-1}\mu_tdt/2)dP$
の下で標準ブラウン運動になる.
$Q$の空間と$P$の空間の$L^2$収束は保たれているので, $Q$の空間において,$Y_t$は
\begin{eqnarray}
\int^t_0\mu_sds+\sum_kb_{t_{k-1}}(W_{t\wedge t_k}-W_{t\wedge t_{k-1}})&=&\int^t_0\mu_sds+\sum_kb_{t_{k-1}}(W'_{t\wedge t_k}-W'_{t\wedge t_{k-1}}+\int^{t\wedge t_k}_{t\wedge t_{k-1}}b_s^{\top}(b_sb_s^{\top})^{-1}\mu_sds) \nonumber 
\end{eqnarray}
の$L^2$収束極限となるので, $Y_t=\int^t_0b_sdW'_s$となる.
$Q$の空間でノイズやサンプリングは独立だから引き続き条件を満たす.
絶対連続測度間でtightnessが保たれることもわかり, $[V]$が成立する. $E[|dQ/dP|^q]<\infty$より $[A1]$ 5.も引き続き成立する.

よって$\mathcal{G}_t=\mathcal{F}_t\vee \sigma(\{\epsilon^{n,k}_i\}_{n,k,i})\vee \sigma(\{\Pi_n\}_n)$とおくと,
$Z$: $\mathcal{F}$-measurable bounded random variableと$f$:有界連続関数に対して,
\begin{eqnarray}
E_P[Zf(b_n^{1/2}(\hat{\sigma}_n-\sigma_{\ast}))]&=&E_P[E_P[Z|\mathcal{G}_T]f(b_n^{1/2}(\hat{\sigma}_n-\sigma_{\ast}))] \nonumber \\
&=&E_Q[E_P[Z|\mathcal{G}_T]f(b_n^{1/2}(\hat{\sigma}_n-\sigma_{\ast}))(dP/dQ)] \nonumber \\
&\to & E_Q[E_P[Z|\mathcal{G}_T]f(\Gamma_1 \mathcal{N})(dP/dQ)] \nonumber \\
&=& E_P[E_P[Z|\mathcal{G}_T]f(\Gamma_1 \mathcal{N})]=E_P[E_P[E_P[Z|\mathcal{G}_T]f(\Gamma_1 z)]|_{z=\mathcal{N}}] \nonumber \\
&=&E_P[E_P[Zf(\Gamma_1 z)]|_{z=\mathcal{N}}]=E_P[Zf(\Gamma_1 \mathcal{N})]. \nonumber 
\end{eqnarray}
安定収束の際は有界性が言えないがカットオフした時の残差が$n$によらずいくらでも小さくできるので良い.
limitの空間でもmeasure changeができるのは, $(\Omega,\mathcal{F},P)$に独立正規分布をのっけた直積空間と思ってよいから, $\tilde{P}(A\cap B)=\tilde{P}(A)\tilde{P}(B)=E_{\tilde{Q}}[1_A(dP/dQ)]\tilde{P}[B]$
で定めれば元の$P$の空間に直積空間をつけたものと同値になる.

$b$:constの時は$bb^{\top}$の上から$\sigma_{\ast}$-一様評価があるので基本的に$\bar{R}_n()$の評価はすぐに$\sigma_{\ast}$-一様評価に変えられる.
期待値の外に$\sup_{\sigma_{\ast}}$をとるので問題は起きない. $Y_t^{\sigma}$の評価がでても問題は起きないだろうが
constantだとない.
}
\end{discuss}

\section{Identifiability of the model}\label{identifiability-section}

In this section, we check the identifiability condition, $\inf_{\sigma\neq \sigma_{\ast}}((-\mathcal{Y}_1(\sigma))/|\sigma-\sigma_{\ast}|^2)>0$ almost surely.
This condition is necessary to deduce consistency of the maximum-likelihood-type estimator, as seen in Proposition~\ref{consistency}.
In general, it is not easy to check this condition directly because $\mathcal{Y}_1(\sigma)$ is a complicated function of $b_t$ and $a^j_t$.
On the other hand, Ogihara and Yoshida~\cite{ogi-yos14} proved that the identifiability condition $[A3]$ of a model for equidistant observations without noise
is sufficient for the identifiability of a model for nonsynchronous observations.
This is also the case for our model. 

\begin{proposition}\label{separability-prop}
Assume $[A1]$, $[A2]$, and $[V]$. Then there exists a positive constant $c$ such that
\begin{equation}\label{separability-ineq}
-\mathcal{Y}_1(\sigma)\geq \chi\int^T_0\left\{(|b^1_t|^2-|b^1_{t,\ast}|^2)^2+(|b^2_t|^2-|b^2_{t,\ast}|^2)^2+(b^1_t\cdot b^2_t-b^1_{t,\ast}\cdot b^2_{t,\ast})^2\right\}dt
\end{equation}
for any $\sigma$, where 
\begin{equation*}
\chi=c(1-\bar{\rho}^2)\frac{(v_{1,\ast}\wedge v_{2,\ast})^{1/2}}{v_{1,\ast}\vee v_{2,\ast}}	\left(\sup_{j,t}a^j_t\right)^{-1/2}\left(\sup_{j,t,\sigma}(|b^j(t,X_t,\sigma)|\vee |b^j(t,X_t,\sigma)|^{-1})\right)^{-5}.
\end{equation*}
In particular, $\inf_{\sigma\neq \sigma_{\ast}}((-\mathcal{Y}_1(\sigma))/|\sigma-\sigma_{\ast}|^2)>0$ almost surely under $[A1]$--$[A3]$ and $[V]$.
\end{proposition}

\begin{proof}
{\colord It is sufficient to show the results under the additional condition $[A1']$ 
by localization techniques similar to the proof of Proposition~\ref{Hn-lim}.}

Let $\hat{D}_m=\tilde{D}_m-M_{m,\ast}$ and ${\bf B}=\sup_{j,t,\sigma}(|b^j(t,X_t,\sigma)|\vee |b^j(t,X_t,\sigma)|^{-1})$, then since
\begin{eqnarray}
&&u^{\top}\hat{D}_m^{-1/2}\tilde{S}_m\hat{D}_m^{-1/2}u \nonumber \\
&\geq& u^{\top}\left(
\begin{array}{ll}
|\tilde{b}^1_m|^2\mathcal{E} & \tilde{b}^1_m\cdot \tilde{b}^2_m\{|I^1_{i,m}\cap I^2_{j,m}||I^1_{i,m}|^{-1/2}|I^2_{j,m}|^{-1/2}\}_{i,j} \\
\tilde{b}^1_m\cdot \tilde{b}^2_m\{|I^1_{i,m}\cap I^2_{j,m}||I^1_{i,m}|^{-1/2}|I^2_{j,m}|^{-1/2}\}_{j,i} & |\tilde{b}^2_m|^2\mathcal{E} 
\end{array}
\right)u \nonumber
\end{eqnarray}
for any $u \in \mathbb{R}^{{\bf J}_{1,n}+{\bf J}_{2,n}}$, 
we have $\lVert (\hat{D}_m^{1/2}\tilde{S}_m^{-1}\hat{D}_m^{1/2})^{1/2}\rVert \leq C{\bf B}(1-\bar{\rho}^2)^{-1/2}$ by Lemma~\ref{inverse-norm-est},
and hence we obtain
\begin{eqnarray}
&&\lVert(\hat{D}_m^{1/2}\tilde{S}_m^{-1}\hat{D}_m^{1/2})^{1/2}(\hat{D}_m^{-1/2}\tilde{S}_{m,\ast}\hat{D}_m^{-1/2})(\hat{D}_m^{1/2}\tilde{S}_m^{-1}\hat{D}_m^{1/2})^{1/2}\rVert \nonumber \\
&=&\lVert \mathcal{E}+(\hat{D}_m^{1/2}\tilde{S}_m^{-1}\hat{D}_m^{1/2})^{1/2}(\hat{D}_m^{-1/2}(\tilde{S}_{m,\ast}-\tilde{S}_m)\hat{D}_m^{-1/2})(\hat{D}_m^{1/2}\tilde{S}_m^{-1}\hat{D}_m^{1/2})^{1/2}\rVert
\leq 1+C{\bf B}^2(1-\bar{\rho}^2)^{-1}. \nonumber
\end{eqnarray}
Then Lemma~\ref{log-det-est} yields
\begin{eqnarray}\label{separability-prop-eq1}
{\rm tr}(\tilde{S}^{-1}_m\tilde{S}_{m,\ast}-\mathcal{E})+\log\det \tilde{S}_m-\log\det \tilde{S}_{m,\ast}
\geq C{\bf B}^{-2}(1-\bar{\rho}^2){\rm tr}(\tilde{S}^{-1}_m(\tilde{S}_{m,\ast}-\tilde{S}_m)\tilde{S}^{-1}_m(\tilde{S}_{m,\ast}-\tilde{S}_m)). 
\end{eqnarray}
\begin{discuss}
{\colorr 非負定値対称行列$A$に対し, $(A^{1/2})^{-1}=(A^{-1})^{1/2}$となることは, 固有値と直交行列を用いた両辺の表現を見れば明らか.}
\end{discuss}

Therefore, we have 
\begin{eqnarray}
&&{\rm tr}(\tilde{S}^{-1}_m\tilde{S}_{m,\ast}-\mathcal{E})+\log\det \tilde{S}_m-\log\det \tilde{S}_{m,\ast} \nonumber \\
&\geq & C{\bf B}^{-2}(1-\bar{\rho}^2){\rm tr}(\tilde{D}_m^{-1}(\tilde{S}_{m,\ast}-\tilde{S}_m)\tilde{S}_m^{-1}(\tilde{S}_{m,\ast}-\tilde{S}_m)) 
\geq C{\bf B}^{-2}(1-\bar{\rho}^2){\rm tr}(\tilde{D}_m^{-1}(\tilde{S}_{m,\ast}-\tilde{S}_m)\tilde{D}_m^{-1}(\tilde{S}_{m,\ast}-\tilde{S}_m)) \nonumber \\
&=&C{\bf B}^{-2}(1-\bar{\rho}^2)\bigg\{\sum_{j=1}^2(|\tilde{b}^j_{m,\ast}|^2-|\tilde{b}^j_m|^2)^2{\rm tr}(\tilde{D}_{j,m}^{-1}D'_{j,m}\tilde{D}_{j,m}^{-1}D'_{j,m})
+2(\tilde{b}^1_{m,\ast}\cdot \tilde{b}^2_{m,\ast}-\tilde{b}^1_m\cdot \tilde{b}^2_m)^2{\rm tr}(\tilde{D}_{1,m}^{-1}\tilde{G}\tilde{D}_{2,m}^{-1}\tilde{G}^{\top})\bigg\}. \nonumber
\end{eqnarray}

Hence it is sufficient to show that $\limsup$ of three quantities ${\rm tr}(\tilde{D}_{j,m}^{-1}D'_{j,m}\tilde{D}_{j,m}^{-1}D'_{j,m})$ for $j=1,2$ and 
${\rm tr}(\tilde{D}_{1,m}^{-1}\tilde{G}\tilde{D}_{2,m}^{-1}\tilde{G}^{\top})$
are estimated from below by positive random variables.

By Lemma~\ref{D1-change} and (\ref{key-eq}) with a sampling scheme $S^{n,1}\equiv S^{n,2}$, we obtain
\begin{eqnarray}
b_n^{-1/2}{\rm tr}(\tilde{D}_{j,m}^{-1}D'_{j,m}\tilde{D}_{j,m}^{-1}D'_{j,m})
&=& b_n^{-1/2}{\rm tr}(\dot{D}_{j,m}^{-1}D'_{j,m}\dot{D}_{j,m}^{-1}D'_{j,m})+\bar{R}_n(\ell_n^{-1}) 
=b_n^{-5/2}(\hat{a}_m^j)^{-2}{\rm tr}(\dot{D}_{j,m}^{-2})+\bar{R}_n(\ell_n^{-1}) \nonumber \\
&=&\frac{b_n^{-5/2}}{(\hat{a}_m^j)^2v_{j,\ast}^2}I_2\bigg(\frac{b_n^{-1}|\tilde{b}^j_m|^2}{\hat{a}_m^jv_{j,\ast}}\bigg)+\bar{R}_n(\ell_n^{-1}) 
=\frac{\pi \ell_n^{-1}}{4(\hat{a}_m^j)^{1/2}v_{j,\ast}^{1/2}|\tilde{b}^j_m|^3}+\bar{R}_n(\ell_n^{-1}). \nonumber
\end{eqnarray}
\begin{discuss}
{\colorr 途中式:
\begin{equation*}
\frac{\ell_n^{-1}}{\hat{a}_m^jv_{j,\ast}^2}\frac{\pi}{4}\frac{v_{j,\ast}^{3/2}(\hat{a}_m^j)^{3/2}}{|\tilde{b}^j_m|^3}+o_p(\ell_n^{-1}) 
\end{equation*}
}
\end{discuss}
Moreover, Lemma~\ref{D1-change} and (\ref{key-eq}) yield
\begin{eqnarray}
b_n^{-1/2}{\rm tr}(\tilde{D}_{1,m}^{-1}\tilde{G}\tilde{D}_{2,m}^{-1}\tilde{G}^{\top})&=&b_n^{-5/2}\frac{\hat{a}_m^2}{(\hat{a}_m^1)^3}{\rm tr}(\dot{D}_{1,m}^{-1}\ddot{D}_{2,m}^{-1})+\bar{R}_n(\ell_n^{-1}) \nonumber \\
&\geq &\frac{b_n^{-5/2}\hat{a}_m^2}{(\hat{a}_m^1)^3v_{1,\ast}v_{2,\ast}}{\rm tr}\bigg(\bigg(\bigg(\bigg(\frac{|\tilde{b}^2_m|^2b_n^{-1}\hat{a}_m^2}{v_{2,\ast}(\hat{a}_m^1)^2}\bigg)\vee \frac{|\tilde{b}^1_m|^2b_n^{-1}}{\hat{a}_m^1v_{1,\ast}}\bigg)\mathcal{E}+M_{1,m}\bigg)^{-2}\bigg)+\bar{R}_n(\ell_n^{-1}) \nonumber \\
&=&\frac{\ell_n^{-1}\hat{a}_m^2}{v_{1,\ast}v_{2,\ast}(\hat{a}_m^1)^3}\frac{\pi}{4}\bigg(\bigg(\frac{|\tilde{b}^2_m|^2\hat{a}_m^2}{v_{2,\ast}(\hat{a}_m^1)^2}\bigg)\vee \frac{|\tilde{b}^1_m|^2}{\hat{a}_m^1v_{1,\ast}}\bigg)^{-3/2}+\bar{R}_n(\ell_n^{-1}). \nonumber
\end{eqnarray}
Similarly, we obtain
\begin{eqnarray}
b_n^{-1/2}{\rm tr}(\tilde{D}_{1,m}^{-1}\tilde{G}\tilde{D}_{2,m}^{-1}\tilde{G}^{\top})
&\geq &\frac{\ell_n^{-1}\hat{a}_m^1}{v_{1,\ast}v_{2,\ast}(\hat{a}_m^2)^3}\frac{\pi}{4}\bigg(\bigg(\frac{|\tilde{b}^1_m|^2\hat{a}_m^1}{v_{1,\ast}(\hat{a}_m^2)^2}\bigg)\vee \frac{|\tilde{b}^2_m|^2}{\hat{a}_m^2v_{2,\ast}}\bigg)^{-3/2}+\bar{R}_n(\ell_n^{-1}).
\end{eqnarray}
Therefore, we obtain (\ref{separability-ineq}).

In particular, by Lemma 6 and Remark 4 in~\cite{ogi-yos14}, there exists a positive-valued random variable $\mathcal{R}$ such that 
\begin{discuss}
{\colorr 
$[A1]$はOgiYosのLemma6の条件より弱いから$\mathcal{R}$のモーメントはわからないが正値でとれることはわかる。
} 
\end{discuss}
\begin{equation*}
-\mathcal{Y}_1(\sigma)\geq \chi \mathcal{R}(-\mathcal{Y}_0(\sigma))
\end{equation*}
for any $\sigma$. Therefore we have $\inf_{\sigma\neq \sigma_{\ast}}((-\mathcal{Y}_1(\sigma))/|\sigma-\sigma_{\ast}|^2)>0$ almost surely under $[A1]$--$[A3]$ and $[V]$.

\end{proof}

\begin{discuss}
{\colorr $b$:constのとき$\sigma_{\ast}$に対する一様な$[A3]$があれば, $\mathcal{R}$や$\chi$の一様な評価から一様な分離性が言える.
}
\end{discuss}

\section{Asymptotic mixed normality of the estimator}\label{mixed-normality-section}

In this section we prove the consistency and asymptotic mixed normality of $\hat{\sigma}_n$.
To obtain asymptotic mixed normality, we prove stable convergence of the score function $b_n^{-1/4}\partial_{\sigma}H_n(\sigma_{\ast},v_{\ast})$
by means of the martingale limit theorem for a mixed normal limit in Jacod~\cite{jac97}.
We also use the idea by Jacod et al.~\cite{jac-etal09} to adapt the limit theorem to models containing observation noise.

Consistency is an immediate consequence of Proposition~\ref{Hn-lim} and the identifiability condition.

\begin{proposition}\label{consistency}
Assume $[A1]$--$[A3]$ and $[V]$. Then $\hat{\sigma}_n\to^p \sigma_{\ast}$ as $n\to \infty$.
\end{proposition}

\begin{proof}
Let $\epsilon,\delta$ be arbitrary positive constants.
By Proposition~\ref{Hn-lim}, we have $\sup_{\sigma}|H_n(\sigma,\hat{v}_n)-H_n(\sigma_{\ast},\hat{v}_n)-\mathcal{Y}_1(\sigma)|\to^p0$ as $n\to\infty$.
Moreover, Proposition~\ref{separability-prop} ensures that there exists $\eta>0$ such that $P[\inf_{\sigma\neq \sigma_{\ast}}((-\mathcal{Y}_1(\sigma))/|\sigma-\sigma_{\ast}|^2)\leq \eta]<\epsilon$.
Since $H_n(\hat{\sigma}_n,\hat{v}_n)-H_n(\sigma_{\ast},\hat{v}_n)\geq 0$ {\colorg by the definition of $\hat{\sigma}_n$}, we obtain
\begin{eqnarray}
P[|\hat{\sigma}_n-\sigma_{\ast}|\geq \delta] &<& P[\mathcal{Y}_1(\hat{\sigma}_n)\leq -\eta\delta^2]+\epsilon
\leq P[\sup_{\sigma}|H_n(\sigma,\hat{v}_n)-H_n(\sigma_{\ast},\hat{v}_n)-\mathcal{Y}_1(\sigma)|\geq \eta \delta^2]+\epsilon < 2\epsilon \nonumber
\end{eqnarray}
for sufficiently large $n$.

\end{proof}

\begin{discuss}
{\colorr JacodのTheorem 3.2は, Gloter and JacodでのStable convergenceの証明を見る限り, $n\to \ell_n$としても適用できそう. 
JacodのThm3.2.の証明から問題なさそう. [4]の(1.15)と(2.12)を使うところとoptional Lenglartのくだりはしっかりみていないが
$n$と$l_n$の違いが問題という風には見えない.}
\end{discuss}

\begin{proposition}\label{st-conv}
Assume $[A1]$, $[A2]$, and $[V]$. Then $b_n^{-1/4}\partial_{\sigma}H_n(\sigma_{\ast},\hat{v}_n)\to^{s\mathchar`-\mathcal{L}}\Gamma_1^{1/2}\mathcal{N}$ as $n\to\infty$.
\end{proposition}

\begin{proof}
It is sufficient to prove the results assuming the additional condition $[A1']$.

Since $b_n^{-1/4}\partial_{\sigma}\tilde{H}_n(\sigma_{\ast},v_{\ast})=-2^{-1}b_n^{-1/4}\sum_m\bar{E}_m[\tilde{Z}_m^{\top}\partial_{\sigma}\tilde{S}_m^{-1}\tilde{Z}_m] + o_p(1)$,
we only need to check assumptions of Theorem 3.2 in Jacod~\cite{jac97} for
$\mathcal{X}^n_m=-2^{-1}b_n^{-1/4}\bar{E}_m[\tilde{Z}_m^{\top}\partial_{\sigma}\tilde{S}_{m,\ast}^{-1}\tilde{Z}_m]$.
For any $\epsilon>0$, Lemma~\ref{ZSZ-est} yields
\begin{eqnarray}
\sum_{m=1}^{[\ell_nt]}E_m[|\mathcal{X}^n_m|^21_{\{|\mathcal{X}^n_m|>\epsilon\} }] 
&\leq &\frac{Cb_n^{-1}}{\epsilon^2}\sum_{m=1}^{[\ell_nt]}E_m[(\tilde{Z}_m^{\top}\partial_{\sigma}\tilde{S}_{m,\ast}^{-1}\tilde{Z}_m)^4]\to^p 0. \nonumber 
\end{eqnarray}
Moreover, it is easy to see that
$\sum_{m=1}^{[\ell_nt]}E_m[\mathcal{X}^n_m(W_{s_m}-W_{s_{m-1}})]\to^p 0$.
\begin{discuss}
{\colorr
\begin{equation*}
(\tilde{Z}_m)_i(\tilde{Z}_m)_j-E_m[(\tilde{Z}_m)_i(\tilde{Z}_m)_j]
\sim \Delta W \Delta W+ (\Delta W^2-t) + (\epsilon_i-\epsilon_{i-1})\Delta W + (\epsilon_i-\epsilon_{i-1})(\epsilon_j-\epsilon_{j-1})-M_{ij}
\end{equation*}

}
\end{discuss}

Let $N$ be a bounded martingale orthogonal to $W_t$. We will show $\sum_{m=1}^{[\ell_nt]}E_m[\mathcal{X}^n_m(N_{s_m}-N_{s_{m-1}})]\to^p 0$.
\begin{discuss}
{\colorr $S^{n,k}_i$はすべて$\mathcal{G}_0$に入っているから, これとorthogonalになるかどうかを考えなくてよい.}
\end{discuss}
Let ${\bf N}$ be the set of finite sums of random variables $f({\bf X}_T)\prod_{j=1}^lg_j(\epsilon^{n_j,k_j}_{i_j})$ 
where $f$ and $g_j$ are bounded Borel functions, ${\bf X}_T$ is an $\mathcal{F}^{(0)}_T$-measurable random variable, 
$n_1,\cdots, n_l\in\mathbb{N}$, $1\leq k_1,\cdots, k_l\leq 2$, and $i_1,\cdots, i_l\in\mathbb{Z}_+$.
Since ${\bf N}$ is dense in $L^1(\Omega, \mathcal{F}_T,P)$, Jacod~\cite{jac79} (4.15) ensures that the set ${\bf N}'$ of 
linear combinations of martingales $\{E[N|\mathcal{F}_t]\}_{0\leq t\leq T}$ with $N\in {\bf N}$
are dense in all bounded martingales orthogonal to $W$.
\begin{discuss}
{\colorr 任意のマルチンゲール空間$\{\sum_k\int^t_0a^k_tdM^k_t\}$はZornの補題から基底を持つことが分かる. よって
閉包が全体でないならあるマルチンゲール$N$があって, $N\perp M$ for any $M\in\mathcal{M}$.
よって$E[f({\bf X}_T)g(\epsilon_1,\cdots, \epsilon_k)|\mathcal{F}_t]$は$\mathcal{M}(W^{\perp})$でdense. 
}
\end{discuss}
Therefore, it is sufficient to show that
\begin{equation}\label{N'-conv}
\sum_{m=1}^{[\ell_nt]}E_m[\mathcal{X}^n_m(N'_{s_m}-N'_{s_{m-1}})]\to^p0
\end{equation} 
for $N'\in {\bf N}'$.

Martingales in the form $N'_t=\int^t_0 a^0_tdW_t+\sum_k\int^t_0a^k_tdM^k_t$ with a bounded step function $a^0_t$, bounded progressively measurable functions $\{a^k_t\}_t$
and bounded $\mathcal{F}^{(0)}_t$-martingales $\{M^k_t\}_t$ orthogonal to $W$ obviously satisfy (\ref{N'-conv}) and are dense in the set of all bounded $\mathcal{F}^{(0)}_t$-martingales. 
Therefore, (\ref{N'-conv}) holds for any bounded $\mathcal{F}^{(0)}$-martingale $N'$.

Moreover, let $N\in {\bf N}$, $N'_t=E[N|\mathcal{F}_t]$, and ${\bf T}=\{\alpha;\tilde{I}_{\alpha,m}\cap \{S^{n_{l'},k_{l'}}_{i_{l'}}\}_{l'=1}^l\neq\emptyset\}$, then we have
\begin{equation*}
N'_t=E[f({\bf X}_T)E[\prod_{j=1}^lg_j(\epsilon^{n_j,k_j}_{i_j})|\mathcal{F}^{(0)}_T\otimes \mathcal{F}^{(1)}_t]|\mathcal{F}_t]
=E[\prod_{j=1}^lg_j(\epsilon^{n_j,k_j}_{i_j})|\mathcal{F}^{(1)}_{\inf_{\tilde{I}_{\alpha,m}}}]E[f({\bf X}_T)|\mathcal{F}^{(0)}_t]
\end{equation*}
for $\alpha\not\in {\bf T}$ and $t\not\in \cup_{\alpha'}\tilde{I}_{\alpha',m}$. 
Therefore, we obtain
\begin{eqnarray}
&&|E_m[\mathcal{X}^n_m(N'_{s_m}-N'_{s_{m-1}})]| \nonumber \\
&=&\frac{1}{2}|E_m[\sum_{\alpha,\beta}({\bf A}_m^{\top}\partial_{\sigma}\tilde{S}_{m,\ast}{\bf A}_m)^{-1}_{\alpha,\beta}\bar{E}_m[(\tilde{\epsilon}_{\alpha,m}-\dot{\epsilon}_{\alpha,m})(\tilde{\epsilon}_{\beta,m}-\dot{\epsilon}_{\beta,m})](N'_{s_m}-N'_{s_{m-1}})]| \nonumber \\
&\leq &\frac{1}{2}|E_m[\sum_{\alpha,\beta\in {\bf T}}({\bf A}_m^{\top}\partial_{\sigma}\tilde{S}_{m,\ast}{\bf A}_m)^{-1}_{\alpha,\beta}\bar{E}_m[\tilde{\epsilon}_{\alpha,m}\tilde{\epsilon}_{\beta,m}-\dot{\epsilon}_{\alpha,m}\tilde{\epsilon}_{\beta,m}-\dot{\epsilon}_{\beta,m}\tilde{\epsilon}_{\alpha,m}](N'_{s_m}-N'_{s_{m-1}})]| \nonumber \\
&&+\frac{1}{2}|E_m[\sum_{\alpha,\beta\in {\bf T}}({\bf A}_m^{\top}\partial_{\sigma}\tilde{S}_{m,\ast}{\bf A}_m)^{-1}_{\alpha,\beta}\bar{E}_m[\dot{\epsilon}_{\alpha,m}\dot{\epsilon}_{\beta,m}](N'_{s_m}-N'_{s_{m-1}})]| \to^p 0. \nonumber 
\end{eqnarray}

Lemma~\ref{ZSZ-est} yields
\begin{equation*}
E_m[(\mathcal{X}^n_m)^2]=\frac{b_n^{-1/2}}{4}\{E_m[(\tilde{Z}_m\partial_{\sigma}\tilde{S}_{m,\ast}^{-1}\tilde{Z}_m)^2]-E_m[\tilde{Z}_m\partial_{\sigma}\tilde{S}_{m,\ast}^{-1}\tilde{Z}_m]^2\}=\frac{b_n^{-1/2}}{2}{\rm tr}(\tilde{S}_{m,\ast}\partial_{\sigma}\tilde{S}_{m,\ast}^{-1}\tilde{S}_{m,\ast}\partial_{\sigma}\tilde{S}_{m,\ast}^{-1})+\bar{R}_n(b_n^{-1/2}).
\end{equation*}
On the other hand, since $\partial_{\sigma}\log \det \tilde{S}_m(x,\sigma)=-{\rm tr}(\partial_{\sigma}\tilde{S}_m \tilde{S}_m^{-1})$, we have
\begin{equation*}
E_m[\tilde{Z}_m^{\top}\partial_{\sigma}^2\tilde{S}_m^{-1}\tilde{Z}_m+\partial_{\sigma}^2\log\det \tilde{S}_m]|_{\sigma=\sigma_{\ast}}
={\rm tr}(\partial_{\sigma}^2\tilde{S}_{m,\ast}^{-1}\tilde{S}_{m,\ast})
-{\rm tr}(\partial_{\sigma}^2\tilde{S}_{m,\ast}^{-1}\tilde{S}_{m,\ast})+{\rm tr}(\tilde{S}_{m,\ast}^{-1}\partial_{\sigma}\tilde{S}_{m,\ast}\tilde{S}_{m,\ast}^{-1}\partial_{\sigma}\tilde{S}_{m,\ast}).
\end{equation*}
Therefore we have
\begin{equation*}
\sum_{m=1}^{[\ell_nt]}E_m[(\mathcal{X}^n_m)^2]=-b_n^{-1/2}\sum_{m=1}^{[\ell_nt]}E_m[\tilde{Z}_m^{\top}\partial_{\sigma}^2\tilde{S}_m^{-1}\tilde{Z}_m+\partial_{\sigma}^2\log\det \tilde{S}_m)]\bigg|_{\sigma=\sigma_{\ast}} \to^p -\partial_{\sigma}^2\mathcal{Y}_1(\sigma_{\ast},t), 
\end{equation*}
where {\colorlg
\begin{eqnarray}
\mathcal{Y}_1(\sigma,t)&=&\int^t_0\bigg\{
\frac{\sum_{j=1}^2(|b^j_s|^2-|b^j_{s,\ast}|^2)(|b^{3-j}_s|^2\sqrt{\tilde{a}^1_s\tilde{a}^2_s}+\tilde{a}^j_s\sqrt{\det(b_sb_s^{\top})})-2(b^1_s\cdot b^2_s-b^1_{s,\ast}\cdot b^2_{s,\ast})b^1_s\cdot b^2_s\sqrt{\tilde{a}_s^1\tilde{a}_s^2}}
{2\sqrt{2}\sqrt{\det(b_sb_s^{\top})}\varphi(\tilde{a}^1_s|b^1_s|^2+\tilde{a}^2_s|b^2_s|^2,\tilde{a}^1_s\tilde{a}^2_s\det(b_sb_s^{\top}))} \nonumber \\
&&\quad \quad -\frac{\varphi(\tilde{a}^1_s|b^1_s|^2+\tilde{a}^2_s|b^2_s|^2,\tilde{a}^1_s\tilde{a}^2_s\det(b_sb_s^{\top}))-\varphi(\tilde{a}^1_s|b^1_{s,\ast}|^2+\tilde{a}^2_s|b^2_{s,\ast}|^2,\tilde{a}^1_s\tilde{a}^2_s\det(b_{s,\ast}b_{s,\ast}^{\top}))}{2\sqrt{2}}\bigg\}ds. \nonumber
\end{eqnarray}
}

Then Theorem 2.1 in Jacod~\cite{jac97} yields $b_n^{-1/4}\partial_{\sigma}H_n(\sigma_{\ast},\hat{v}_n)\to^{s\mathchar`- \mathcal{L}} \Gamma_1^{1/2}\mathcal{N}$.
\end{proof}

\noindent
{\bf Proof of Theorem~\ref{main}.}
Since the parameter space $\Lambda$ is open, there exists $\epsilon>0$ such that $O(\epsilon,\sigma_{\ast})=\{\sigma;|\sigma-\sigma_{\ast}|<\epsilon\}\subset \Lambda$.
Then we have
\begin{equation*}
-\partial_{\sigma}H_n(\sigma_{\ast},\hat{v}_n)=\int^1_0\partial_{\sigma}^2H_n(\sigma_{\ast},\hat{v}_n)(\sigma_{\ast}+t(\hat{\sigma}_n-\sigma_{\ast}))(\hat{\sigma}_n-\sigma_{\ast})dt
\end{equation*}
for $\hat{\sigma}_n\in \Lambda$, by $\partial_{\sigma}H_n(\hat{\sigma}_n,\hat{v}_n)=0$.

Hence we obtain $b_n^{1/4}(\hat{\sigma}_n-\sigma_{\ast})={\colorlg \tilde{\Gamma}_{1,n}^{-1}}b_n^{-1/4}\partial_{\sigma}H_n(\sigma_{\ast},\hat{v}_n)$ on $\{\det {\colorlg \tilde{\Gamma}_{1,n}}\neq 0$ and $\hat{\sigma}_n\in O(\epsilon,\sigma_{\ast})\}$,
where ${\colorlg \tilde{\Gamma}_{1,n}}=-b_n^{-1/2}\int^1_0\partial_{\sigma}^2H_n(\sigma_{\ast}+t(\hat{\sigma}_n-\sigma_{\ast}))dt$.
Then since Propositions~\ref{Hn-lim} and~\ref{consistency} yield $P[\det {\colorlg \tilde{\Gamma}_{1,n}}=0]\to 0$, $P[\hat{\sigma}_n\in O(\epsilon,\sigma_{\ast})^c]\to 0$
and ${\colorlg \tilde{\Gamma}_{1,n}^{-1}}1_{\{\det {\colorlg \tilde{\Gamma}_{1,n}}\neq 0\}}\to^p {\colorlg \Gamma_1^{-1}}$, 
we have $b_n^{1/4}(\hat{\sigma}_n-\sigma_{\ast})\to^{s\mathchar`-\mathcal{L}} {\colorlg \Gamma_1^{-1/2}}\mathcal{N}$ as $n\to\infty$ {\colorg by Proposition \ref{st-conv}}. 

{\colorlg 
Moreover, Proposition \ref{Hn-lim} and Theorem \ref{consistency} ensure that $\hat{\Gamma}_{1,n}\to^p\Gamma_1$, which completes the proof.
}
\qed 

\section{Proof of the LAN property}\label{LAN-section}

{\colord To obtain the LAN property of our model, the arguments in the proof of Theorem \ref{main} are essential.
Indeed, by using Propositions \ref{Hn-lim} and \ref{st-conv}, we obtain a LAMN-type property of the quasi-log-likelihood function $H_n$ with respect to $\sigma$:
$H_n(\sigma_{\ast}+b_n^{-1/4}u_1,v_{\ast})-H_n(\sigma_{\ast},v_{\ast})-u_1\cdot b_n^{-1/4}\partial_{\sigma}H_n(\sigma_{\ast},v_{\ast})-u_1^{\top}b_n^{-1/2}\partial_{\sigma}^2H_n(\sigma_{\ast},v_{\ast})u_1/2\to^p0$ as $n\to\infty$ for any $u_1\in\mathbb{R}^d$,
and $(b_n^{-1/4}\partial_{\sigma}H_n(\sigma_{\ast},v_{\ast}),-b_n^{-1/2}\partial_{\sigma}^2H_n(\sigma_{\ast},v_{\ast}))\to^{s\mathchar`-\mathcal{L}}(\Gamma_1^{1/2}\mathcal{N},\Gamma_1)$,
where $\mathcal{N}$ is a $d$-dimensional standard normal random variable independent of $\mathcal{F}$.
On the other hand, under the assumptions of Theorem \ref{LAN-theorem},}
the {\it true} log-likelihood ratio $\log(dP_{\sigma_{\ast}+b_n^{-1/4}u_1,v_{\ast}+b_n^{-1/2}u_2,n}/dP_{\sigma_{\ast},v_{\ast},n})$ for $u_1\in\mathbb{R}^d$
and $u_2\in\mathbb{R}^2$ is obtained as $-(Z_1^{\top}S_1^{-1}Z_1+\log\det S_1)/2$ if we set $k_n=b_n$.
We cannot apply the argument of Section~\ref{Hn-limit-section} to this quantity because the estimate $\ell_n\to\infty$ is essential there.
Therefore, we follow the approaches by Gloter and Jacod~\cite{glo-jac01a} to show the LAN property.
We set a `subexperiment' and a `superexperiment', which are obtained by respectively removing and adding observations from the original experiment.
{\colord The likelihood functions of these experiments have similar properties to $H_n$, and therefore we can prove the LAN properties for these experiments with the same limit distribution.}
{\colord We can prove that these results lead us to} the LAN property of the original one.

Let $\mathcal{Z}=(\mathbb{R}^8)^{\mathbb{N}}$, {\colord $\pi_i(z)=(x^{k,j}_i,t^k_i,e^k_i)_{j,k=1,2}$} for $i\in \mathbb{Z}_+$ and $z=(x^{k,j}_{i'},t^k_{i'},e^k_{i'})_{i'\in\mathbb{Z}_+,j,k=1,2}\in\mathcal{Z}$.
Let $\mathcal{H}=\mathfrak{B}(\{\pi_i^{-1}(A);i\in\mathbb{Z}_+, A\in \mathcal{B}(\mathbb{R}^8)\})$, $P'_{\sigma'_{\ast},v'_{\ast}}$ be the induced probability measure on $(\mathcal{Z},\mathcal{H})$
by \\
{\colord $((Y^j_{S^{n,k}_i}1_{\{i\leq {\bf J}_{k,n}\}},S^{n,k}_i1_{\{i\leq {\bf J}_{k,n}\}},\epsilon^{n,k}_i1_{\{i\leq {\bf J}_{k,n}\}})_{i\in\mathbb{Z}_+,j,k=1,2})$}
with a true value $(\sigma'_{\ast},v'_{\ast})$.
\begin{discuss}
{\colorr 強い解があるから任意の$(\sigma'_{\ast},v'_{\ast})$に対し, processがdefされる.}
\end{discuss}
{\colord We can ignore the event $\min_{j,m}k^j_m\leq 0$.} 

Let $\mathcal{H}'=\mathfrak{B}(t^k_i;i\in \mathbb{Z}_+,k=1,2)$, $j^k_0=-1$, $j^k_m=\max\{i;t^k_i<s_m\}\vee 0 \ (1\leq m\leq \ell_n)$, 
{\colord ${\bf l}(0)=1$, ${\bf l}(m)=\min\{k;t^k_i=\max_{i',k'}\{t^{k'}_{i'}<s_m\} \ {\rm for} \ {\rm some}\ i\}$ for $1\leq m\leq \ell_n$},
\begin{eqnarray}
\mathcal{H}^{n,0}&=&\mathfrak{B}(({\colord x^{k,k}_{i+1}+e^k_{i+1}-x^{k,k}_i-e^k_i})1_{\{i\not{\in}\{j^k_m\}_m\} };i\in\mathbb{Z}_+,k=1,2) \bigvee \mathcal{H}', \nonumber \\
\mathcal{H}^{n,1}&=&\mathfrak{B}({\colord x^{k,k}_i}+e^k_i;i\in\mathbb{Z}_+,k=1,2) \bigvee \mathcal{H}', \nonumber \\
\mathcal{H}^{n,2}&=&\mathcal{H}^{n,1}\bigvee {\colord \mathfrak{B}(x^{{\bf l}(m),j}_{j^{{\bf l}(m)}_m};1\leq m\leq \ell_n,j=1,2)}. \nonumber 
\end{eqnarray}
Then we can see $\mathcal{H}^{n,0}\subset\mathcal{H}^{n,1}\subset\mathcal{H}^{n,2}$ and
\begin{equation}\label{log-likelihood-eq2}
\log(dP_{\sigma_u,v_u}/dP_{\sigma_{\ast},v_{\ast}})=\log (dP'_{\sigma_u,v_u}/dP'_{\sigma_{\ast},v_{\ast}})|_{\mathcal{H}^{n,1}}.
\end{equation}
Moreover, we obtain
\begin{equation}\label{log-likelihood-eq}
\log \frac{dP'_{\sigma_u,v_u}}{dP'_{\sigma_{\ast},v_{\ast}}}\bigg|_{\mathcal{H}^{n,l}}({\colord (Y^j_{S^{n,k}_i}1_{\{i\leq {\bf J}_{k,n}\}},S^{n,k}_i1_{\{i\leq {\bf J}_{k,n}\}},\epsilon^{n,k}_i1_{\{i\leq {\bf J}_{k,n}\}})_{i\in\mathbb{Z}_+,j,k=1,2}})
=H^{(l)}_n(\sigma_u,v_u)-H^{(l)}_n(\sigma_{\ast},v_{\ast})
\end{equation}
for $l=0$, where 
{\colord $Z^{(0)}_m=Z_m$ and $S^{(0)}_m=S_m$ for $2\leq m\leq \ell_n$, $Z^{(0)}_1$ and $S^{(0)}_1$ are defined similarly,
$H_n^{(0)}(\sigma,v)=-\sum_{m=1}^{\ell_n}\{(Z^{(0)}_m)^{\top}(S^{(0)}_m)^{-1}(\sigma,v)Z^{(0)}_m+\log\det S^{(0)}_m(\sigma,v)\}/2$}, $\sigma_u=\sigma_{\ast}+b_n^{-1/4}u_1$ and $v_u=v_{\ast}+b_n^{-1/2}u_2$ for $u=(u_1,u_2)\in \mathbb{R}^d\times \mathbb{R}^2$.
{\colord $\mathcal{H}^{n,0}$ and $\mathcal{H}^{n,2}$ are $\sigma$-fields for `subexperiment' and `superexperiment', respectively, while $\mathcal{H}^{n,1}$ is the one for the original one.
Therefore, (\ref{log-likelihood-eq}) means that our quasi-likelihood function $H_n$ is equal to the log-likelihood function of `subexperiment' except the term for $m=1$.}

{\colord To obtain similar formula to (\ref{log-likelihood-eq}) for $l=2$, let ${\bf R}_m=S^{n,1}_{K^1_m}\vee S^{n,2}_{K^2_m}$,
$\tilde{{\bf Y}}_{m,-}^k=\tilde{Y}^k_{K^k_{m-1}+1}-Y^k_{{\bf R}_{m-1}}$,
\begin{equation*}
\tilde{{\bf Y}}^k_{m,+}=\left\{
\begin{array}{ll}
Y^k_{S^{n,k}_{K^k_m}}-\tilde{Y}^k_{K^k_m-1} & {\rm if} \ S^{n,k}_{K^k_m}={\bf R}_m \\
(\tilde{Y}^k(I^k_{k^k_m,m}), Y^k_{{\bf R}_m}-\tilde{Y}^k_{S^{n,k}_{K^k_m}})^{\top} & {\rm if} \ S^{n,k}_{K^k_m}<{\bf R}_m
\end{array}
\right.
\end{equation*}
${\bf Y}_{m,0}=\epsilon^{n,k}_{K^k_m}$ if $S^{n,3-k}_{K^{3-k}_m}<{\bf R}_m$, ${\bf Y}_{m,0}=(\epsilon^{n,1}_{K^1_m},\epsilon^{n,2}_{K^2_m})^{\top}$ if $S^{n,1}_{K^1_m}=S^{n,2}_{K^2_m}$, and
}
\begin{eqnarray}
Z^{(2)}_m&=&{\colord (((\tilde{{\bf Y}}_{m,-}^k)^{\top}, (\tilde{Y}^k(I^k_{i,m}))^{\top}_{1\leq i<k^k_m}, (\tilde{{\bf Y}}^k_{m,+})^{\top})_{k=1}^2,{\bf Y}^{\top}_{m,0})^{\top}} \nonumber 
\end{eqnarray}
{\colord for $2\leq m\leq \ell_n$. Then Observations $((\tilde{Y}^k_i)_{k,i},(S^{n,k}_i)_{k,i},(Y^j_{{\bf R}_m})_{j,m})$ are equivalent to $Z^{(2)}_m$, 
and hence (\ref{log-likelihood-eq}) holds for $l=2$, where
${\bf E}(v)=v_{3-k}$, $k^{(2),k}_m=k^k_m+1$, $k^{(2),3-k}_m=k^{3-k}_m$, and $I^k_{k^{(2),k}_m,m}=[S^{n,k}_{K^k_m},{\bf R}_m)$ if $S^{n,k}_{K^k_m}<{\bf R}_m$, 
${\bf E}(v)={\rm diag}(v_1,v_2)$ and $(k^{(2),1}_m,k^{(2),2}_m)=(k^1_m,k^2_m)$ if $S^{n,1}_{K^1_m}=S^{n,2}_{K^2_m}$,

}
\begin{eqnarray}
(M^{(2)}_{j,m})_{ii'}&=&2\delta_{ii'}-\delta_{|i-i'|=1}-\delta_{(i,i')=(1,1)}-\delta_{i=i'=k^{(2),j}_m}, \nonumber
\end{eqnarray}
\begin{equation*}
S^{(2)}_m(\sigma,v)=\left(
\begin{array}{lll}
{\rm diag}((|b^1|^2|I^1_{i,m}|)_{1\leq i\leq k^{(2),1}_m})+v_1M_{1,m}^{(2)} & \{b^1\cdot b^2|I^1_{i,m}\cap I^2_{j,m}|\}_{1\leq i\leq k^{(2),1}_m,1\leq j\leq k^{(2),2}_m} & \\
\{b^1\cdot b^2|I^1_{i,m}\cap I^2_{j,m}|\}_{1\leq j\leq k^{(2),2}_m,1\leq i\leq k^{(2),1}_m} & {\rm diag}((|b^2|^2|I^2_{j,m}|)_{1\leq j\leq k^{(2),2}_m})+v_2M_{2,m}^{(2)} & \\
 & & {\bf E}(v)
\end{array}
\right)
\end{equation*}
for $2\leq m \leq \ell_n$, $Z^{(2)}_1,M^{(2)}_1$, and $S^{(2)}_1$ are similarly defined, 
\begin{discuss}
{\colorr 先頭に$e^k_0$がつくなど}
\end{discuss}
and $H^{(2)}_n(\sigma,v)=-\sum_{m=1}^{\ell_n}\{(Z^{(2)}_m)^{\top}S^{(2)}_m(\sigma,v)^{-1}Z^{(2)}_m+\log\det S^{(2)}_m\}/2$.
\begin{discuss}
{\colorr 初期値は$\sigma_{\ast},v_{\ast}$に依存しないので$P_{Y_0}$は消える.
$\min_{j,m} k^j_m=0$の時も$\mathcal{H}^{n,0}$, $\mathcal{H}^{n,2}$はdefできて, $\log (dP/dP)|\mathcal{H}$はそのmeasureで無視できるからOK}
\end{discuss}

{\colord
The log-likelihood functions $H^{(0)}_n$ and $H^{(2)}_n$ of `subexperiment' and `superexperiment', respectively, have similar forms to that of $H_n$,
and hence we can prove convergence of likelihood ratios. 
Gloter and Jacod~\cite{glo-jac01a} showed that convergence of likelihood ratios of `subexperiment' and `superexperiment' imply convergence of that of the original experiment.
Here, we use a slight extension of their result. The proof is straightforward.
Let ${\bf U}^{n,l}_{\sigma,v}=dP'_{\sigma,v}/dP'_{\sigma_{\ast},v_{\ast}}|_{\mathcal{H}^{n,l}}$, $K\in \mathbb{N}$,
and $\{\sigma^k_n\}_{n\in\mathbb{N},1\leq k\leq K}\subset \Lambda$ and $\{v^l_n\}_{n\in\mathbb{N},1\leq k\leq K}\subset (0,\infty)\times (0,\infty)$ be arbitrary sequences.

\begin{theorem}\label{glo-jac-thm}
Suppose that $({\bf U}^{n,l}_{\sigma_n^1,v_n^1},\cdots, {\bf U}^{n,l}_{\sigma_n^K,v_n^K})$ converges in law under $P'^n_{\sigma_{\ast},v_{\ast}}$ to a limit $Y=(Y^1,\cdots, Y^K)$ with $0<Y^k<\infty$ a.s.
and $E[Y^k]=1$ for $l=0,2$ and $1\leq k\leq K$. Then the same convergence holds for $l=1$.
\end{theorem}
}

\begin{discuss}
{\colorr
\begin{proof}
任意の部分列に対してさらなる部分列を取れば
\begin{equation*}
({\bf U}^{n,0}_{\sigma_n^1,v_n^1},{\bf U}^{n,1}_{\sigma_n^1,v_n^1},{\bf U}^{n,2}_{\sigma_n^1,v_n^1},\cdots, {\bf U}^{n,0}_{\sigma_n^K,v_n^K}, {\bf U}^{n,1}_{\sigma_n^K,v_n^K}, {\bf U}^{n,2}_{\sigma_n^K,v_n^K})
\to^d \exists  ({\bf U}^0_1,{\bf U}^1_1,{\bf U}^2_1,\cdots, {\bf U}^0_K, {\bf U}^1_K, {\bf U}^2_K)
\end{equation*}
かつ$\mathcal{L}({\bf U}^0_1,\cdots, {\bf U}^0_K)=\mathcal{L}({\bf U}^2_1,\cdots, {\bf U}^2_K)=\mathcal{L}(Y)$.
${\bf U}^{n,l}=({\bf U}^{n,l}_{\sigma_n^1,v_n^1},\cdots, {\bf U}^{n,l}_{\sigma_n^K,v_n^K})$, 
${\bf U}^l=({\bf U}^l_1,\cdots, {\bf U}^l_K)$, $\psi_p$:conti on $\mathbb{R}^K$, $\phi$ : bdd conti on $\mathbb{R}^{2K}$.
\begin{equation*}
|E^n_{\ast}[(\phi({\bf U}^{n,0},{\bf U}^{n,1})(1-\psi_p({\bf U}^{n,l})){\bf U}^{n,l}]|\leq E^n_{\ast}[(1-\psi_p({\bf U}^{n,l})){\bf U}^{n,l}]\to 0 \quad (n,p\to\infty)
\end{equation*}
よって, $E^n_{\ast}[\phi({\bf U}^{n,0},{\bf U}^{n,1}){\bf U}^{n,l}]\to E[\phi({\bf U}^0,{\bf U}^1){\bf U}^l]$.
有限の$n$の時の関係式から${\bf U}^1=E[{\bf U}^2|\mathcal{F}_1]$, ${\bf U}^0=E[{\bf U}^1|\mathcal{F}_0]$.
\begin{equation*}
(E[\sqrt{{\bf U}^{j,k}}|\mathcal{F}_{j-1}])^2\leq E[{\bf U}^{j,k}|\mathcal{F}_{j-1}]={\bf U}^{j-1,k}.
\end{equation*}
ゆえに
\begin{equation*}
\sqrt{{\bf U}^{0,k}}\geq E[\sqrt{{\bf U}^{1,k}}|\mathcal{F}_0]\geq E[\sqrt{{\bf U}^{2,k}}|\mathcal{F}_0].
\end{equation*}
$E[\sqrt{{\bf U}^{0,k}}]=E[\sqrt{{\bf U}^{2,k}}]$より, $\sqrt{{\bf U}^{0,k}}=E[\sqrt{{\bf U}^{2,k}}|\mathcal{F}_0]$.
\begin{equation*}
E[(\sqrt{{\bf U}^{2,k}}-E[\sqrt{{\bf U}^{2,k}}|\mathcal{F}_0])^2]=E[{\bf U}^{2,k}-{\bf U}^{0,k}]=0.
\end{equation*}
よって$\sqrt{{\bf U}^{2,k}}=E[\sqrt{{\bf U}^{2,k}}|\mathcal{F}_0]=\sqrt{{\bf U}^{0,k}}$.
${\bf U}^{1,k}=E[{\bf U}^{2,k}|\mathcal{F}_1]={\bf U}^{0,k}$. ゆえに$\mathcal{L}({\bf U}^1)=\mathcal{L}({\bf U}^0)$.
\end{proof}

}
\end{discuss}

{\colord
We first prove the LAN properties of `subexperiment' and `superexperiment'. Then Theorem~\ref{glo-jac-thm} leads to the LAN property of the original one.
}
Taylor's formula yields
\begin{eqnarray}
&&H_n^{(l)}(\sigma_u,v_u)-H_n^{(l)}(\sigma_{\ast},v_{\ast}) \nonumber \\
&=&b_n^{-1/4}\partial_{\sigma}H_n^{(l)}(\sigma_{\ast},v_{\ast})\cdot u_1+2^{-1}b_n^{-1/2}u_1^{\top}\partial_{\sigma}^2H_n^{(l)}(\sigma_{\ast},v_{\ast})u_1+b_n^{-1/2}\partial_vH_n^{(l)}(\sigma_{\ast},v_{\ast})\cdot u_2 \nonumber \\
&&+2^{-1}b_n^{-1}u_2^{\top}\partial_v^2H_n^{(l)}(\sigma_{\ast},v_{\ast})u_2 
+ \int^1_0\int^1_0\sum_{i,j}\partial_{v_i}\partial_{\sigma_j}H_n^{(l)}(\sigma_{tu},v_{su})b_n^{-3/4}u_{2,i}u_{1,j}dsdt \nonumber \\
&&+\int^1_0\frac{(1-t)^3}{2}\bigg(\sum_{i,j,k}\partial_{\sigma_i}\partial_{\sigma_j}\partial_{\sigma_k}H_n^{(l)}(\sigma_{tu},v_{\ast})u_{1,i}u_{1,j}u_{1,k}b_n^{-3/4}
+\sum_{i,j,k}\partial_{v_i}\partial_{v_j}\partial_{v_k}H_n^{(l)}(\sigma_{\ast},v_{tu})u_{2,i}u_{2,j}u_{2,k}b_n^{-3/2}\bigg)dt. \nonumber
\end{eqnarray}

We examine the limit of each term on the right-hand side.
\begin{lemma}\label{LAN-lemma}
Assume $[A1'']$, $[A2]$, and $[V]$. Then 
\begin{enumerate}
\item $\sup_{\sigma}|b_n^{-1/2}\partial_{\sigma}^k(H_n^{(l)}(\sigma,v_{\ast})-H_n^{(l)}(\sigma_{\ast},v_{\ast}))-\partial_{\sigma}^k\mathcal{Y}_1(\sigma)|\to^p 0$,
\item $\sup_v|b_n^{-1}\partial_v^k(H_n^{(l)}(\sigma_{\ast},v)-H_n^{(l)}(\sigma_{\ast},v_{\ast}))-\partial_v^k\mathcal{Y}_2(v)|\to^p 0$,
\item $\sup_{\sigma,v}|b_n^{-3/4}\partial_{\sigma}\partial_vH_n^{(l)}(\sigma,v)|\to^p 0$
\end{enumerate}
as $n\to \infty$ for $0\leq k\leq 3$ and $l=0,2$.
\end{lemma}

\begin{proof}
{\it 1}. We obtain the results by a similar argument to the proof of Proposition~\ref{Hn-lim} together with Lemma~\ref{ddotD-properties},
the results in Section 8 of~\cite{glo-jac01a}, and similar estimates to Lemmas~\ref{D1-change} and~\ref{Sm-properties}.
{\colord For any $\epsilon>0$, $(\epsilon\mathcal{E}+M_{j,m}^{(2)})^{-1}$ has a similar decomposition to (\ref{dotD-eq}) by replacing $p_{i-1},\cdots,p_j$ by $p_{i-1}',\cdots, p_j'$.}
Therefore, estimate for the quantity corresponding to $\Lambda_1$ is obtained since
\begin{equation*}
{\colord ((\epsilon\mathcal{E}+M^{(2)}_{j,m})^{-1})_{11}=\frac{\prod_{l=1}^{k^j_m-1}p'_l(\epsilon)}{(p'_{k^j_m}(\epsilon)-1)\prod_{l=1}^{k^j_m-1}p'_l(\epsilon)}=O(b_n^{1/2}).}
\end{equation*}
\begin{discuss}
{\colorr and hence $b_n^{-1/2}{\rm tr}((\dot{D}^{(-1)}_{j,m})^{-1}G')\vee (b_n^{-1/2}{\rm tr}((\dot{D}^{(-1)}_{j,m})^{-1}\mathcal{E}'))=O_p(\ell_n^{-1})$.
Lemma~\ref{Hn-lim-lemma2}に対応する評価は$(\epsilon\mathcal{E}+M)^{-1}\leq (\epsilon\mathcal{E}+M-(E_{11}-\mathcal{E}))^{-1}\leq (\epsilon\mathcal{E}+M-(E_{11}-\mathcal{E})-(E_{kk}-\mathcal{E}))^{-1}$と挟み込めばよい.
traceの評価も\cite{glo-jac01a}の8章で与えられる.
}
\end{discuss}

\noindent
{\it 2}. We first obtain
\begin{eqnarray}
&&b_n^{-1}\partial_v^lH_n^{(l)}(\sigma,v) \nonumber \\
&=&-\frac{1}{2}b_n^{-1}\sum_m\left\{E_m[(Z_m^{(l)})^{\top}\partial_v^l(S_m^{(l)})^{-1}Z^{(l)}_m]+\partial_v^l\log\det S^{(l)}_m\right\} 
-\frac{1}{2}b_n^{-1}\sum_m\bar{E}_m[(Z_m^{(l)})^{\top}\partial_v^l(S_m^{(l)})^{-1}Z^{(l)}_m] \nonumber \\
&=&-\frac{1}{2}b_n^{-1}\sum_m\left\{{\rm tr}(\partial_v^l(S_m^{(l)})^{-1}S_{m,\ast}^{(l)})+\partial_v^l\log\det S_m^{(l)}\right\} 
+O_p\bigg(\bigg(b_n^{-2}\sum_m{\rm tr}(\partial_v^l(S_m^{(l)})^{-1}S_{m,\ast}^{(l)}\partial_v^l(S_m^{(l)})^{-1}S_{m,\ast}^{(l)})\bigg)^{1/2}\bigg) \nonumber \\
&=&-\frac{1}{2}b_n^{-1}\sum_m\left\{{\rm tr}(\partial_v^l(S_m^{(l)})^{-1}S_{m,\ast}^{(l)})+\partial_v^l\log\det S_m^{(l)}\right\} +o_p(1). \nonumber
\end{eqnarray}
{\colord Let $\tilde{D}^{(l)}_m=(\tilde{D}^{(l)}_{1,m},\tilde{D}^{(l)}_{2,m})$ for $l=0,2$, $k^{(0),j}_m=k^j_m$ for $j=1,2$, 
$\tilde{D}^{(l)}_{1,m}=((S^{(l)}_m)_{i,i'})_{1\leq i,i'\leq k^{(l),1}_m}$,
\begin{equation*}
\tilde{D}^{(0)}_{2,m}=((S^{(0)}_m)_{j,j'})_{k^{(0),1}_m< j,j'\leq k^{(0),1}_m+k^{(0),2}_m}, 
\quad \tilde{D}^{(2)}_{2,m}={\rm diag}(((S^{(2)}_m)_{j,j'})_{k^{(2),1}_m< j,j'\leq k^{(2),1}_m+k^{(2),2}_m},{\bf E}(v)),
\end{equation*}
\begin{equation*}
\hat{G}^{(l)}=(\tilde{D}^{(l)}_{1,m})^{-1/2}\{|I^1_{i,m}\cap I^2_{j,m}|1_{\{j\leq k^{(l),2}_m\}}\}_{1\leq i\leq k^{(l),1}_m,1\leq j\leq \tilde{k}^{(l),2}_m}(\tilde{D}^{(l)}_{2,m})^{-1/2}, 
\end{equation*}
where $\tilde{k}^{(0),2}_m=k^{(0),2}_m$ and $\tilde{k}^{(2),2}_m$ is the size of $\tilde{D}_{2,m}^{(2)}$. Then} we obtain
\begin{eqnarray}
{\rm tr}((S_m^{(l)})^{-1}S_{m,\ast}^{(l)})&=&{\rm tr}\bigg((\tilde{D}^{(l)}_m)^{-1/2}\left(
\begin{array}{ll}
\mathcal{E} & \hat{G}^{(l)} \\
(\hat{G}^{(l)})^{\top} & \mathcal{E}
\end{array}
\right)^{-1}(\tilde{D}^{(l)}_m)^{-1/2}(\tilde{D}^{(l)}_{m,\ast})^{1/2}\left(
\begin{array}{ll}
\mathcal{E} & \hat{G}^{(l)}_{\ast} \\
(\hat{G}^{(l)}_{\ast})^{\top} & \mathcal{E}
\end{array}
\right)(\tilde{D}^{(l)}_{m,\ast})^{1/2}\bigg) \nonumber \\
&=&\sum_{p=0}^{\infty}\big\{{\rm tr}((\tilde{D}^{(l)}_{1,m})^{-1/2}(\hat{G}^{(l)}(\hat{G}^{(l)})^{\top})^p(\tilde{D}^{(l)}_{1,m})^{-1/2}\tilde{D}^{(l)}_{1,m,\ast} \nonumber \\
&&-(\tilde{D}^{(l)}_{1,m})^{-1/2}(\hat{G}^{(l)}(\hat{G}^{(l)})^{\top})^p\hat{G}^{(l)}(\tilde{D}^{(l)}_{2,m})^{-1/2}(\tilde{D}^{(l)}_{2,m,\ast})^{1/2}(\hat{G}^{(l)}_{\ast})^{\top}(\tilde{D}^{(l)}_{1,m,\ast})^{1/2}) \nonumber \\
&&+{\rm tr}((\tilde{D}^{(l)}_{2,m})^{-1/2}((\hat{G}^{(l)})^{\top}\hat{G}^{(l)})^p(\tilde{D}^{(l)}_{2,m})^{-1/2}\tilde{D}^{(l)}_{2,m,\ast} \nonumber \\
&&-(\tilde{D}^{(l)}_{2,m})^{-1/2}(\hat{G}^{(l)})^{\top}(\hat{G}^{(l)}(\hat{G}^{(l)})^{\top})^p(\tilde{D}^{(l)}_{1,m})^{-1/2}((\tilde{D}^{(l)}_{1,m,\ast}))^{1/2}\hat{G}_{\ast}^{(l)}(\tilde{D}^{(l)}_{2,m,\ast})^{1/2})\big\}. \nonumber
\end{eqnarray}

Since $\lVert (\tilde{D}^{(l)}_{j,m})^{-1}\tilde{D}^{(l)}_{j,m,\ast}\rVert=O_p(1)$, terms involving $\hat{G}$ are $O_p(b_n^{1/2}\ell_n^{-1})$. Therefore we have
\begin{eqnarray}
&&b_n^{-1}{\rm tr}((S^{(l)}_m)^{-1}S_{m,\ast}^{(l)}) \nonumber \\
&=&b_n^{-1}\sum_{j=1}^2{\rm tr}((\tilde{D}^{(l)}_{j,m})^{-1}\tilde{D}^{(l)}_{j,m,\ast})+o_p(\ell_n^{-1}) \nonumber \\
&=&b_n^{-1}\sum_{j=1}^2\frac{v_{j,\ast}}{v_j}{\rm tr}(\mathcal{E}_{k^j_m}-(\tilde{D}^{(l)}_{j,m})^{-1}(\tilde{D}^{(l)}_{j,m}-v_jv_{j,\ast}^{-1}\tilde{D}^{(l)}_{j,m,\ast}))+o_p(\ell_n^{-1})
=\ell_n^{-1}\sum_{j=1}^2\hat{a}_m^j\frac{v_{j,\ast}}{v_j}+o_p(\ell_n^{-1}). \nonumber 
\end{eqnarray}
Similarly we have
$b_n^{-1}{\rm tr}(\partial_v^k(S^{(l)}_m)^{-1}S^{(l)}_{m,\ast})=\ell_n^{-1}\sum_{j=1}^2\hat{a}_m^j\partial_v^k\frac{v_{j,\ast}}{v_j}+o_p(\ell_n^{-1})$.
Moreover, we obtain
\begin{eqnarray}
&&b_n^{-1}\partial_v^k\log \frac{\det S_m^{(l)}}{\det S_{m,\ast}^{(l)}} \nonumber \\
&=&\sum_{j=1}^2b_n^{-1}\partial_v^k\log \det ((\tilde{D}^{(l)}_{j,m,\ast})^{-1}\tilde{D}^{(l)}_{j,m})
+b_n^{-1}\partial_v^k\log\det(\mathcal{E}-\hat{G}^{(l)}(\hat{G}^{(l)})^{\top})-b_n^{-1}\partial_v^k\log\det(\mathcal{E}-\hat{G}_{\ast}^{(l)}(\hat{G}^{(l)}_{\ast})^{\top}) \nonumber \\
&=&b_n^{-1}\sum_{j=1}^2\partial_v^k\log\det(v_jv_{j,\ast}^{-1}\mathcal{E}_{k^j_m}+(\tilde{D}^{(l)}_{j,m,\ast})^{-1}(\tilde{D}^{(l)}_{j,m}-v_jv_{j,\ast}^{-1}\tilde{D}^{(l)}_{j,m,\ast}))+o_p(\ell_n^{-1}) \nonumber \\
&=&\ell_n^{-1}\sum_{j=1}^2\hat{a}_m^j\partial_v^k\log (v_jv_{j,\ast}^{-1})+o_p(\ell_n^{-1}). \nonumber
\end{eqnarray}
\begin{discuss}
{\colorr $b_n^{-1}\log\det \tilde{S}_m$は収束しなそう.}
\end{discuss}

\noindent
{\it 3}. Since $\partial_v\log \det S_m^{(i)}=-{\rm tr}(\partial_vS^{(i)}_m(S^{(i)}_m)^{-1})$ and 
\begin{equation*}
{\colord b_n^{-3/4}}\partial_{\sigma}\partial_vH^{(i)}_n(\sigma,v)=-\frac{1}{2}{\colord b_n^{-3/4}}\sum_m\{{\rm tr}(\partial_{\sigma}\partial_v(S^{(i)}_m)^{-1}S^{(i)}_{m,\ast})+\partial_{\sigma}\partial_v\log \det S^{(i)}_m\} + o_p(1),
\end{equation*}
we have ${\colord b_n^{-3/4}}\partial_{\sigma}\partial_vH^{(i)}_n(\sigma,v)=o_p(1)$.
Sobolev's inequality and similar estimates for $\partial_{\sigma}\partial_v^2$ and $\partial_{\sigma}^2\partial_v$ yield the results.
\begin{discuss}
{\colorr 似たような式は繰り返し書くのではなく, このように書く.}
\end{discuss}
\end{proof}

The following lemma completes the proof of {\colord the LAN properties of `subexperiment' and `superexperiment'}.
\begin{lemma}\label{LAN-lemma2}
Assume $[A1'']$, $[A2]$, and $[V]$. Then $(b_n^{-1/4}\partial_{\sigma}\hat{H}^{(l)}_n(\sigma_{\ast},v_{\ast}),b_n^{-1/2}\partial_v\hat{H}^{(l)}_n(\sigma_{\ast},v_{\ast}))\to^{s\mathchar`-\mathcal{L}} {\rm diag}(\Gamma_1^{1/2},\Gamma_2^{1/2})\tilde{\mathcal{N}}$
for $l=0,2$, {\colord where $\tilde{\mathcal{N}}$ is a $(d+2)$-dimensional normal random variable independent of $\mathcal{F}$}.
\end{lemma}

\begin{proof}
Let
$\tilde{\mathcal{X}}^n_m=-b_n^{-1/4}\bar{E}_m[(Z^{(l)}_m)^{\top}\partial_{\sigma}(S^{(l)}_{m,\ast})^{-1}Z^{(l)}_m]/2 
-b_n^{-1/2}\bar{E}_m[(Z^{(l)}_m)^{\top}\partial_v(S^{(l)}_{m,\ast})^{-1}Z^{(l)}_m]/2$,
then we have
\begin{eqnarray}
\sum_{m=1}^{[\ell_nt]}E_m[|\tilde{\mathcal{X}}^n_m|^21_{\{|\mathcal{X}^n_m|>\epsilon\} }] 
&\leq &\frac{C}{\epsilon^2}b_n^{-1}\sum_mE_m[((Z^{(l)}_m)^{\top}\partial_{\sigma}(S^{(l)}_{m,\ast})^{-1}Z^{(l)}_m)^4] 
+\frac{C}{\epsilon^2}b_n^{-2}\sum_mE_m[((Z^{(l)}_m)^{\top}\partial_v(S^{(l)}_{m,\ast})^{-1}Z^{(l)}_m)^4] \nonumber \\
&\leq &\frac{Cb_n^{-2}}{\epsilon^2}\sum_m\sum_{j=1}^4{\rm tr}((\partial_v(S^{(l)}_m)^{-1}S^{(l)}_{m,\ast})^j) + o_p(1)\to^p 0. \nonumber
\end{eqnarray}
\begin{discuss}
{\colorr
$U\tilde{Z}_m$が八個出てくるところはちゃんと上から押さえられるかチェックする.
}
\end{discuss}

Moreover, similarly to the proof of Proposition~\ref{st-conv}, we have
\begin{equation*}
\sum_{m=1}^{[\ell_nt]}E_m[\tilde{\mathcal{X}}^n_m(N_{s_m}-N_{s_{m-1}})]=\sum_{m=1}^{[\ell_nt]}E_m[\tilde{\mathcal{X}}^n_m(W_{s_m}-W_{s_{m-1}},W'_{s_m}-W'_{s_{m-1}})]=0
\end{equation*}
for any bounded martingale $N$ orthogonal to $(W_t, W'_t)_t$.

Therefore, by Theorem 3.2 in Jacod~\cite{jac97}, it is sufficient to show that
\begin{equation*}
\sum_{m=1}^{[\ell_nt]}E_m[(\tilde{\mathcal{X}}^n_m)^2]\to^p {\rm diag}(-\partial_{\sigma}^2\mathcal{Y}_1(\sigma_{\ast},t),-\partial_v^2\mathcal{Y}_2(v_{\ast},t)),
\end{equation*}
where $\mathcal{Y}_2(v,t)=-\int^t_0\sum_{j=1}^2a^j_s\{(v_{j,\ast}/v_j)-1+\log(v_j/v_{j,\ast})\}ds/2$.

Then we obtain the desired results by
\begin{eqnarray}
\sum_mE_m[(\tilde{\mathcal{X}}^n_m)^2]
&=&\frac{1}{4}b_n^{-1/2}\sum_mE_m[\bar{E}_m[(Z_m^{(l)})^{\top}(\partial_{\sigma}(S^{(l)}_{m,\ast})^{-1}+b_n^{-1/4}\partial_v(S^{(l)}_{m,\ast})^{-1})Z_m^{(l)}]^2] \nonumber \\
&=&\frac{b_n^{-1/2}}{2}\sum_m{\rm tr}(S^{(l)}_{m,\ast}(\partial_{\sigma}(S^{(l)}_{m,\ast})^{-1}+b_n^{-1/4}\partial_v(S^{(l)}_{m,\ast})^{-1})
S^{(l)}_{m,\ast}(\partial_{\sigma}(S^{(l)}_{m,\ast})^{-1}+b_n^{-1/4}\partial_v(S^{(l)}_{m,\ast})^{-1})) \nonumber \\
&=&-b_n^{-1/2}\partial_{\sigma}^2H^{(l)}_n(\sigma_{\ast},v_{\ast})-b_n^{-1}\partial_v^2H^{(l)}_n(\sigma_{\ast},v_{\ast})+\frac{b_n^{-3/4}}{2}\sum_m{\rm tr}(\partial_{\sigma}S^{(l)}_{m,\ast}(S^{(l)}_{m,\ast})^{-1}\partial_vS^{(l)}_{m,\ast}(S^{(l)}_{m,\ast})^{-1}) \nonumber \\
&\to^p& {\rm diag}(-\partial_{\sigma}^2\mathcal{Y}_1(\sigma_{\ast},t),-\partial_v^2\mathcal{Y}_2(v_{\ast},t)). \nonumber
\end{eqnarray}

\end{proof}

{\colord
\noindent
{\it Proof of Theorem~\ref{LAN-theorem}}.
Let ${\bf U}(u)=\exp(u^{\top}\Gamma^{1/2}\tilde{\mathcal{N}}-u^{\top}\Gamma u/2)$ for $u\in\mathbb{R}^{d+2}$. 
Let $Z^{(1)}=(\epsilon^{n,k}_0,(\tilde{Y}^k_i-\tilde{Y}^k_{i-1})_{i=1}^{{\bf J}_{k,n}})_{k=1,2}$,
$S^{(1)}(\sigma,v)$ be a symmetric matrix of size ${\bf J}_{1,n}+{\bf J}_{2,n}+2$ defined by
$(S^{(1)}(\sigma,v))_{11}=v_1$, $(S^{(1)}(\sigma,v))_{{\bf J}_{1,n}+2,{\bf J}_{1,n}+2}=v_2$,
\begin{equation*}
\begin{array}{ll}
(S^{(1)}(\sigma,v))_{ij}={\rm diag}(v_1M({\bf J}_{1,n}+1),v_2M({\bf J}_{2,n}+1))_{ij} & {\rm if} \ i\neq j \ {\rm and} \ \{i,j\}\cap \{1,{\bf J}_{1,n}+2\}\neq \emptyset, \\
(S^{(1)}(\sigma,v))_{ij}=|b^1(\sigma)|^2(S^{n,1}_{i-1}-S^{n,1}_{i-2})\delta_{ij}+v_1M({\bf J}_{1,n}+1)_{ij} & {\rm if} \ 2\leq i,j\leq {\bf J}_{1,n}+1, \\
(S^{(1)}(\sigma,v))_{ij}=|b^2(\sigma)|^2(S^{n,2}_{i'-1}-S^{n,2}_{i'-2})\delta_{ij}+v_2M({\bf J}_{2,n}+1)_{i'j'} & {\rm if} \ 2\leq i',j'\leq {\bf J}_{2,n}+1, \\
(S^{(1)}(\sigma,v))_{ij}=b^1\cdot b^2(\sigma)(S^{n,1}_{i-1}\wedge S^{n,2}_{j'-1}-S^{n,1}_{i-2}\vee S^{n,2}_{j'-2})_+ & {\rm if} \ 2\leq i\leq {\bf J}_{1,n}+1 \ {\rm and} \ 2\leq j'\leq {\bf J}_{2,n}+1, 
\end{array}
\end{equation*}
where $i'=i-{\bf J}_{1,n}-1$ and $j'=j-{\bf J}_{1,n}-1$.
Then we have (\ref{log-likelihood-eq}) for $l=1$ with $H^{(1)}_n(\sigma,v)=-((Z^{(1)})^{\top}(S^{(1)}(\sigma,v))^{-1}Z^{(1)}+\log\det S^{(1)}(\sigma,v))/2$.
Moreover, Theorem~\ref{glo-jac-thm} and Lemmas~\ref{LAN-lemma} and~\ref{LAN-lemma2} yield
\begin{equation}\label{LAN-proof-eq1}
(H_n^{(1)}(\sigma_{u^{(1)}},v_{u^{(1)}})-H_n^{(1)}(\sigma_{\ast},v_{\ast}),\cdots, H_n^{(1)}(\sigma_{u^{(k)}},v_{u^{(k)}})-H_n^{(1)}(\sigma_{\ast},v_{\ast}))\to^d (\log{\bf U}(u^{(1)}),\cdots, \log{\bf U}(u^{(k)}))
\end{equation}
as $n\to\infty$ for $u^{(1)},\cdots, u^{(k)}\in \mathbb{R}^{d+2}$.

Furthermore, similar estimates to the proof of Lemma \ref{LAN-lemma} yield $\sup_{\sigma,v}|b_n^{-3/4}\partial_{\sigma}\partial_vH_n^{(1)}(\sigma,v)|\to^p 0$,\\ 
$\sup_{\sigma}|b_n^{-3/4}\partial_{\sigma}^3H_n^{(1)}(\sigma,v_{\ast})|\to^p0$, and $\sup_v|b_n^{-3/2}\partial_v^3H_n^{(1)}(\sigma_{\ast},v)|\to^p0$.
Therefore we obtain 
\begin{equation}\label{LAN-proof-eq2}
H_n^{(1)}(\sigma_u,v_u)-H_n^{(1)}(\sigma_{\ast},v_{\ast})-(u\cdot {\bf V}_{1,n}-u^{\top}{\bf V}_{2,n}u/2)\to^p 0
\end{equation}
as $n\to \infty$ for any $u\in\mathbb{R}^{d+2}$, where
${\bf V}_{1,n}=(b_n^{-1/4}\partial_{\sigma}H_n^{(1)}(\sigma_{\ast},v_{\ast}),b_n^{-1/2}\partial_vH_n^{(1)}(\sigma_{\ast},v_{\ast}))$
and \\
${\bf V}_{2,n}=-{\rm diag}(b_n^{-1/2}\partial_{\sigma}^2H_n^{(1)}(\sigma_{\ast},v_{\ast}),b_n^{-1}\partial_v^2H_n^{(1)}(\sigma_{\ast},v_{\ast}))$.
(\ref{LAN-proof-eq1}) and (\ref{LAN-proof-eq2}) yield ${\bf V}_{1,n}\to^d \Gamma^{1/2}\tilde{\mathcal{N}}$ and ${\bf V}_{2,n}\to^p \Gamma$, and therefore 
we obtain the LAN property of the original experiment with $\Gamma_n={\bf V}_{2,n}$ and $\mathcal{N}_n=\Gamma^{-1/2}{\bf V}_{1,n}$ by (\ref{log-likelihood-eq2}).
\qed
}

\begin{discuss}
{\colorr $u=(1,0,0,0,0)$と$u=(2,0,0,0)$を組み合わせて第1成分の収束が言えるなどして${\bf V}_{j,n}$の収束が言える.
${\bf V}_{2,n}$の非対角成分は$u=(1,1,0,...)$と$u=(1,0,0,...)$と$u=(0,1,0,...)$を組み合わせればよい.
Jeganathan (1983)の論文のLAMNのdefを見れば$\mathcal{N}_n$と$\Gamma_n$が$\sigma_{\ast}$に依存してよいことが分かる.
}
\end{discuss}

\section{Proof of the results in Section~\ref{PLD-section}}\label{PLD-proof-section}

In this final section, we complete the proof of remaining results in Section~\ref{results-section}.
Proposition~\ref{PLD-prop} is proven by the scheme of Yoshida~\cite{yos06, yos11}.
Proposition~\ref{separability-prop} and moment estimates in Lemmas~\ref{Hn-diff-lemma} and~\ref{Hn-lim-lemma2} enable us to check the assumptions of Theorem 2 in \cite{yos11}.
Then the results on convergence of moments and the Bayes-type estimator are obtained by Proposition~\ref{PLD-prop}.

\noindent
{\it Outline of the proof of Proposition~\ref{PLD-prop}}. We apply Theorem 2 in Yoshida~\cite{yos11}.
It is sufficient to prove the following five conditions for any $L>0$ with some positive constant $\delta_1$ and $\delta_2$:
\begin{enumerate}
\item There exists $C_L>0$ such that $P[\inf_{\sigma\neq \sigma_{\ast}}(-\mathcal{Y}_1(\sigma)/|\sigma-\sigma_{\ast}|^2)\leq r^{-1}]\leq C_L/r^L$ 
and \\
$P[\{r^{-1}|u|^2\leq u^{\top}\Gamma_1u/4 {\rm \ for \ any \ } u\in\mathbb{R}^d\}^c]\leq C_L/r^L$ for any $r>0$.
\item $\sup_nE[(b_n^{-1/4}|\partial_{\sigma}H_n(\sigma_{\ast},\hat{v}_n)|)^L]<\infty$.
\item $\sup_nE[(b_n^{\delta_1}\sup_{\sigma}|b_n^{-1/2}(H_n(\sigma,\hat{v}_n)-H_n(\sigma_{\ast},\hat{v}_n))-\mathcal{Y}_1(\sigma)|)^L]<\infty$.
\item $\sup_nE[(b_n^{-1/2}\sup_{\sigma}|\partial_{\sigma}^3H_n(\sigma,\hat{v}_n)|)^L]<\infty$.
\item $\sup_nE[(b_n^{\delta_2}|b_n^{-1/2}\partial_{\sigma}^2H_n(\sigma_{\ast},\hat{v}_n)+\Gamma_1|)^L]<\infty$.
\end{enumerate}
By Taylor's formula for $\mathcal{Y}_1(\sigma)$ and relations $\mathcal{Y}_1(\sigma_{\ast})=\partial_{\sigma}\mathcal{Y}_1(\sigma_{\ast})=0$,
we obtain $\inf_{\sigma\neq \sigma_{\ast}}(-\mathcal{Y}_1(\sigma)/|\sigma-\sigma_{\ast}|^2) \leq \inf_{u\in \mathbb{R}^d\setminus\{0\}}u^{\top}\Gamma_1u/(2|u^2|)$. 
Then Proposition~\ref{separability-prop} and $[B3]$ yield point 1.
By Lemmas~\ref{Hn-diff-lemma} and~\ref{Hn-lim-lemma2} and a similar argument to the proof of Proposition~\ref{Hn-lim}, we obtain 3--5 and
$\sup_nE[(b_n^{-1/4}|\partial_{\sigma}H_n(\sigma_{\ast},\hat{v}_n)-\partial_{\sigma}\tilde{H}_n(\sigma_{\ast},v_{\ast})|)^L]<\infty$.
\begin{discuss}
{\colorr Lemma~\ref{Hn-diff-lemma}の式に$b_n^{L/4}$をかけて得られる.}
\end{discuss}
Moreover, by the Burkholder--Davis--Gundy inequality, we obtain
\begin{eqnarray}
&&E[|b_n^{-1/4}\partial_{\sigma}\tilde{H}_n(\sigma_{\ast},v_{\ast})|^L] \nonumber \\
&=&E\bigg[\bigg|\frac{b_n^{-1/4}}{2}\sum_m\bar{E}_m[\tilde{Z}_m^{\top}\partial_{\sigma}\tilde{S}_{m,\ast}^{-1}\tilde{Z}_m]\bigg|^L\bigg] 
\leq CE\bigg[\bigg(b_n^{-\frac{1}{2}}\sum_m\bar{E}_m[\tilde{Z}_m^{\top}\partial_{\sigma}\tilde{S}_{m,\ast}^{-1}\tilde{Z}_m]^2\bigg)^{L/2}\bigg] \nonumber \\
&\leq &CE\bigg[\bigg(b_n^{-\frac{1}{2}}\sum_mE_m[\bar{E}_m[\tilde{Z}_m^{\top}\partial_{\sigma}\tilde{S}_{m,\ast}^{-1}\tilde{Z}_m]^2]\bigg)^{L/2}\bigg]
+ CE\bigg[\bigg(b_n^{-1}\sum_m\bar{E}_m[\bar{E}_m[\tilde{Z}_m^{\top}\partial_{\sigma}\tilde{S}_{m,\ast}^{-1}\tilde{Z}_m]^2]^2\bigg)^{L/4}\bigg] \nonumber \\
&\leq &CE\bigg[\bigg(b_n^{-\frac{1}{2}}\sum_m{\rm tr}((\partial_{\sigma}\tilde{S}_{m,\ast}^{-1}\tilde{S}_{m,\ast})^2)\bigg)^{L/2}\bigg]
+CE\bigg[\bigg(b_n^{-1}\sum_mE_m[(\tilde{Z}_m^{\top}\partial_{\sigma}\tilde{S}_{m,\ast}^{-1}\tilde{Z}_m)^L]^{4/L}\bigg)^{L/4}\bigg] \nonumber \\
&=&O((b_n^{-1/2}\ell_nb_n^{1/2}\ell_n^{-1})^{L/2})+E[\bar{R}_n((b_n^{-5}k_n^8\ell_n)^{L/4})]=O(1), \nonumber
\end{eqnarray}
which implies point 2.

\qed

\noindent
{\it Proof of Theorem~\ref{moment-conv-thm}}.

We extend ${\bf Z}_n(u)$ to a continuous function on $\mathbb{R}^d$ satisfying $\lim_{|u|\to\infty}{\bf Z}(u)=0$ 
with the supremum norm of the extended function the same as for the original one.
Then by Theorem 5 and Remark 5 in Yoshida~\cite{yos11}, 
it is sufficient to show $\limsup_{n\to\infty}E[|b_n^{1/4}(\hat{\sigma}_n-\sigma_{\ast})|^p]<\infty$ for any $p>0$
and ${\bf Z}_n\to^{s\mathchar`- \mathcal{L}}{\bf Z}$ in $C(B(R))$ as $n\to\infty$ for any $R>0$,
where ${\bf Z}(u)=\exp(\mathcal{N}\cdot u-u^{\top}\Gamma_1u/2)$ and $B(R)=\{u;|u|\leq R\}$.

By Lemma~\ref{Hn-diff-lemma} and a similar argument to the proof of Proposition~\ref{Hn-lim}, we have 
\begin{equation*}
\sup_nE\bigg[\sup_{u\in C(B(R))}|\partial_u\log{\bf Z}_n(u)|\bigg]<\infty.
\end{equation*}
Then Propositions~\ref{Hn-lim} and~\ref{st-conv} and tightness criterion in $C$ space in Billingsley~\cite{bil99} yield
$\log {\bf Z}_n\to^{s\mathchar`- \mathcal{L}}\log{\bf Z}$ in $C(B(R))$. {\colord Then (\ref{PLD-ineq}) completes the proof.}

\qed\\

\noindent
{\it Proof of Theorem~\ref{bayes-thm}}.

By Theorem 10 in Yoshida~\cite{yos11}, it is sufficient to show
\begin{equation}\label{bayes-thm-eq1}
\sup_nE_n\bigg[\bigg(\int_{U_n}{\bf Z}_n(u)\pi(\sigma_{\ast}+b_n{-1/4}u)du\bigg)^{-1}\bigg]<\infty.
\end{equation}
By Proposition~\ref{Hn-lim}, we obtain $sup_nE[|H_n(\sigma_{\ast}+b_n^{-1/4}u)-H_n(\sigma_{\ast})|^p]\leq C_p|u|^p$
for any $U(\delta)$, where $U(\delta)=\{u\in\mathbb{R}^d;|u_i|\leq \delta (i=1,\cdots, d)\}$.
Then we have (\ref{bayes-thm-eq1}) by Lemma 2 in~\cite{yos11}.
\qed

\appendix
\section{Appendix}

\subsection{Results from linear algebra}
\begin{lemma}\label{tr-est}
Let $A$ and $B$ be matrices, with $A$ nonnegative definite and symmetric. Then
\begin{equation*}
|{\rm tr}(AB)|\leq {\rm tr}(A)\lVert B\rVert. 
\end{equation*}
\end{lemma}
\begin{discuss}
{\colorr
$UAU^{\top}=\Lambda$として,
\begin{eqnarray}
|{\rm tr}(AB)|=|{\rm tr}(\Lambda UBU^{\top})|\leq \sum_j\lambda_j|(UBU^{\top})_{jj}|\leq \sum_j\lambda_j\lVert UBU^{\top}\rVert\leq {\rm tr}(A)\lVert B\rVert. \nonumber 
\end{eqnarray}
}
\end{discuss}

{\colord

\begin{lemma}\label{tr-est-mod}
Let $l\in\mathbb{N}$, $A^j$ and $B^j$ be real-valued matrices and $\{\lambda^j_k\}_k$ be eigenvalues of $A^j$ for $1\leq j\leq l$. 
Assume that $A^j$ is symmetric and all the elements of $B^j$ are nonnegative for $1\leq j\leq l$. Then
\begin{equation*}
\sum_{i_1,\cdots, i_{2l}}\prod_{j=1}^l\left(|A^j_{i_{2j-1},i_{2j}}|B^j_{i_{2j},i_{2j+1}}\right)\leq \prod_{j=1}^l\bigg(\lVert B^j \rVert \sum_k|\lambda^j_k|\bigg),
\end{equation*}
where $i_{2l+1}=i_1$.
\end{lemma}

\begin{proof}
Let $U^j$ be an orthogonal matrix such that $A^j=(U^j)^{\top}{\rm diag}((\lambda^j_k)_k)U^j$. Then
\begin{eqnarray}
\sum_{i_1,\cdots, i_{2l}}\prod_{j=1}^l\left(|A^j_{i_{2j-1},i_{2j}}|B^j_{i_{2j},i_{2j+1}}\right)
&\leq &\sum_{k_1,\cdots,k_l}\sum_{i_1,\cdots, i_{2l}}\prod_{j=1}^l\left(|\lambda^j_{k_j}||U^j_{k_j,i_{2j-1}}||U^j_{k_j,i_{2j}}|B^j_{i_{2j},i_{2j+1}}\right) \nonumber \\
&\leq &\sum_{k_1,\cdots,k_l}\prod_{j=1}^l\bigg\{|\lambda^j_{k_j}|\lVert B^j\rVert \sum_i(U^j_{k_j,i})^2\bigg\}
=\prod_{j=1}^l\bigg(\lVert B^j \rVert \sum_k|\lambda^j_k|\bigg). \nonumber
\end{eqnarray}
\end{proof}

\begin{lemma}\label{log-det-exp}
Let $A$ be a symmetric matrix with $\lVert A\rVert<1$. Then $\log\det (\mathcal{E}+A)=\sum_{p=1}^{\infty}(-1)^{p-1}p^{-1}{\rm tr}(A^p)$.
\end{lemma}
\begin{proof}
Let $\{\lambda_j\}_{j=1}^k$ be eigenvalues of $A$. Then $\sup_j|\lambda_j|=\lVert A\rVert<1$, and hence
\begin{equation*}
\log\det(\mathcal{E}+A)=\sum_j\log(1+\lambda_j)=\sum_j\sum_{p=1}^{\infty}(-1)^{p-1}p^{-1}\lambda_j^p=\sum_{p=1}^{\infty}(-1)^{p-1}p^{-1}{\rm tr}(A^p).
\end{equation*}
\end{proof}
}

\begin{lemma}\label{inverse-norm-est}
Let $A$ and $B$ be symmetric, positive definite matrices. Assume that $v^{\top}Av\geq v^{\top}Bv$ for any vector $v$.
Then $\lVert A^{-1}\rVert\leq \lVert B^{-1}\rVert$ and $\lVert A^{-1/2}\rVert\leq \lVert B^{-1/2}\rVert$.
\end{lemma}
{\colord
\begin{proof}
Let $(\lambda^A_j)_j$ and $(\lambda^B_j)_j$ be eigenvalues of $A$ and $B$, respectively. Then for any unit vector $v$, 
there exists an orthogonal matrix $U$ such that
\begin{equation*}
\sum_j\lambda^A_jv_j^2\geq \sum_j\lambda^B_j(Uv)_j^2\geq \inf_j\lambda^B_j.
\end{equation*}
Therefore we obtain $\lVert A^{-1}\rVert^{-1}=\inf_j\lambda^A_j\geq \inf_j\lambda^B_j=\lVert B^{-1}\rVert^{-1}$
and $\lVert A^{-1/2}\rVert^{-1}=\inf_j(\lambda^A_j)^{1/2}\geq \inf_j(\lambda^B_j)^{1/2}=\lVert B^{-1/2}\rVert^{-1}$.
\end{proof}
}


\begin{lemma}\label{tr-lower-est}
Let $B$ be a symmetric, positive definite matrix and $A$ be a symmetric, nonnegative definite matrix. Then
${\rm tr}(AB)\geq {\rm tr}(A)\lVert B^{-1}\rVert^{-1}$.
\end{lemma}
\begin{proof}
Let $\{\lambda^A_j\}_j$ and $\{\lambda^B_j\}_j$ be eigenvalues of $A$ and $B$, respectively, and $U$ be an orthogonal matrix satisfying $UAU^{\top}={\rm diag}((\lambda^A_j)_j)$. Then since
$(UBU^{\top})_{jj}\geq \inf_j\lambda^B_j=\lVert B^{-1}\rVert^{-1}$, we obtain
\begin{eqnarray}
{\rm tr}(AB)=\sum_j\lambda^A_j(UBU^{\top})_{jj}\geq \sum_j\lambda^A_j\lVert B^{-1}\rVert^{-1}= {\rm tr}(A)\lVert B^{-1}\rVert^{-1}. \nonumber
\end{eqnarray}
\end{proof}

\begin{lemma}\label{log-det-est}
Let $\eta>0$ and $A$ be a symmetric matrix. Assume that $\mathcal{E}+A$ is positive definite and $\lVert \mathcal{E}+A\rVert\leq \eta$.
Then ${\rm tr}(A)-\log\det(\mathcal{E}+A)\geq {\rm tr}(A^2)/(4\eta+4)$.
\end{lemma}
\begin{discuss}
{\colorr さらに以下の命題も成り立つ.	
Assume that $\mathcal{E}+A$ is positive definite and $\lVert (\mathcal{E}+A)^{-1}\rVert\leq \eta<\infty$.
Then there exists a constant $c_{\eta}^+>0$ which depends only on $\eta$ such that 
$0\leq {\rm tr}(A)-\log\det(\mathcal{E}+A)\leq c_{\eta}^+{\rm tr}(A^2)$.
}
\end{discuss}
\begin{proof}
We easily obtain the results by using the fact that
$\log\det(\mathcal{E}+A)=\sum_k\log(1+\lambda_k)$ and that $x-x^2/(4\eta+4)\geq \log(1+x)$ for $-1<x\leq \eta +1$,
{\colord where $(\lambda_j)_j$ are eigenvalues of $A$.}
\begin{discuss}
{\colorr $f(x)=x-x^2/(4\eta+4)-\log(1+x)$とおくと, $f'(x)=1-x/(2\eta+2)-1/(1+x)=((2\eta+1)x-x^2)/(1+x)/(2\eta+2)$より, 
$-1<x\leq \eta+1$での$f(x)$の最小値は$f(0)=0$.}
\end{discuss}
\end{proof}

\begin{lemma}\label{ZSZ-est-lemma}
Let $A$ be a symmetric matrix, $B$ a matrix of suitable size and $(\lambda_j)_j$ eigenvalues of $B^{\top}AB$. Then\\
{\colord
1. $|(B^{\top}AB)_{ii}|\leq \lVert A\rVert (B^{\top}B)_{ii}$ for any $i$.\\
2. $\sum_j|\lambda_j|\leq \lVert A\rVert {\rm tr}(B^{\top}B)$.
}
\end{lemma}
\begin{proof}
{\it 1.} Let $U$ be an orthogonal matrix and let $\{\lambda_j\}_j$ be eigenvalues of $A$ such that $U^{\top}AU={\rm diag}((\lambda_j)_j)$. Then we obtain
\begin{equation*}
|(B^{\top}AB)_{ii}|=|\sum_j\lambda_j((U^{\top}B)_{ji})^2|\leq \lVert A\rVert \sum_j((U^{\top}B)_{ji})^2 = \lVert A\rVert (B^{\top}B)_{ii}.
\end{equation*}
{\colord
{\it 2.} There exists an orthogonal matrix $V$ such that $\lambda_j=(V^{\top}B^{\top}ABV)_{jj}$ for any $j$. Then
\begin{equation*}
\sum_j|\lambda_j|=\sum_j|(V^{\top}B^{\top}ABV)_{jj}|\leq \lVert A\rVert \sum_j(V^{\top}B^{\top}BV)_{jj}=\lVert A\rVert {\rm tr}(B^{\top}B)
\end{equation*}
by {\it 1}.
}
\end{proof}

\subsection{Proof of Lemma~\ref{ZSZ-est}}
Let ${\bf A}_m$ be a $(k^1_m+k^2_m)\times (k^1_m+k^2_m)$ matrix with elements $({\bf A}_m)_{ij}=1_{i\geq j}1_{\{i\leq k^1_m \ {\rm or} \ j>k^1_m\} }$,
{\colord ${\bf 1}$ be a matrix with all elements equal to $1$}, $M_{m,\ast}=M_m(v_{\ast})$ and $\hat{{\bf S}}=({\bf A}_m^T)^{-1}M_{m,\ast}^{-2}{\bf A}_m^{-1}$.

\begin{lemma}\label{ZSZ-est-aux-lemma}
Let $1\leq m\leq \ell_n$, $q,q'\in\mathbb{N}$ such that $q'\geq 2q$, $A_m:\{1,\cdots ,k^1_m+k^2_m\}^{q'}\to\{0,1\}$ be a random map and $\iota:\{1,\cdots,q'\}\to \{1,\cdots, 2q\}$ be an injection. 
Assume that there exists a sequence $\{\mathcal{K}_n\}_n$ of positive numbers such that {\colord $\sum_{j_1,\cdots, j_{q'}=1}^{k^1_m+k^2_m}A_m(j_1,\cdots, j_{q'})=\bar{R}_n( \mathcal{K}_n)$}.
Then 
\begin{equation*}
\sum_{i_1,\cdots,i_{2q}}\prod_{j=1}^{q}\hat{{\bf S}}_{i_{2j-1},i_{2j}}A_m(i_{\iota(1)},\cdots, i_{\iota(q')})=\bar{R}_n(k_n^{q+q'-[q'/2]\cdot 2}\mathcal{K}_n).
\end{equation*}

\end{lemma}

\begin{discuss}
{\colorr
$b$: constの時は$A_m$の和が$\sigma_{\ast}$に関して一様に抑えられていれば$\sigma_{\ast}$に関する一様評価もでる.
}
\end{discuss}

\begin{proof}
Let ${\bf K}_m={\rm diag}(((k^1_m+1)v_{1,\ast})^{-1}\mathcal{E}, ((k^2_m+1)v_{2,\ast})^{-1}\mathcal{E})$.
Then since $({\bf A}_mM_{m,\ast}{\bf A}_m^{\top})^{-1}{\bf 1}={\bf K}_m{\bf 1}$ and \\
$(\sum_j|(({\bf A}_mM_{m,\ast}{\bf A}_m^{\top})^{-1})_{i,j}|)\vee (\sum_j|(({\bf A}_mM_{m,\ast}{\bf A}_m^{\top})^{-1})_{j,i}|)\leq 2$ for any $i$, 
we obtain 
\begin{equation*}
{\rm tr}(\hat{{\bf S}}{\bf 1})={\rm tr}(({\bf A}_mM_{m,\ast}{\bf A}_m^{\top})^{-1}{\bf A}_m{\bf A}_m^{\top}({\bf A}_mM_{m,\ast}{\bf A}_m^{\top})^{-1}{\bf 1})={\rm tr}({\bf A}_m{\bf A}_m^{\top}{\bf K}_m{\bf 1}{\bf K}_m)=\bar{R}_n(k_n),
\end{equation*}
and 
\begin{eqnarray}
|(\hat{{\bf S}}{\bf 1})^l\hat{{\bf S}})_{i,j}|
&\leq & \bigg|\sum_{j_1,\cdots, j_{2l+2}}(({\bf A}_mM_{m,\ast}{\bf A}_m^{\top})^{-1})_{i,j_1}\bigg(\prod_{1\leq k\leq l}({\bf A}_m{\bf A}_m^{\top})_{j_{2k-1},j_{2k}}({\bf K}_m{\bf 1}{\bf K}_m)_{j_{2k},j_{2k+1}}\bigg) \nonumber \\
&&\times ({\bf A}_m{\bf A}_m^{\top})_{j_{2l+1},j_{2l+2}}(({\bf A}_mM_{m,\ast}{\bf A}_m^{\top})^{-1})_{j_{2l+2},j}\bigg| \nonumber \\
&=&\bar{R}_n(\underbar{k}_n^{-2l}\bar{k}_n^{3l+1})=\bar{R}_n(k_n^{l+1}) \nonumber
\end{eqnarray}
for $l=0,1$.

If both $i_{2j-1}$ and $i_{2j}$ are outside the image of $\iota$, we have 
\begin{equation*}
\sum_{i_{2j-1},i_{2j}}\hat{{\bf S}}_{i_{2j-1},i_{2j}}A_m(i_{\iota (1)},\cdots, i_{\iota (q')})={\rm tr}(\hat{{\bf S}}{\bf 1})A_m(i_{\iota (1)},\cdots, i_{\iota (q')}).
\end{equation*}
Moreover, if both $i_{2j-1}$ and $i_{2k-1}$ are in the image of $\iota$ and neither $i_{2j}$ nor $i_{2k}$ is in it, then we have
\begin{equation*}
\sum_{i_{2j},i_{2k}}\hat{{\bf S}}_{i_{2j-1},i_{2j}}\hat{{\bf S}}_{i_{2k-1},i_{2k}}A_m(i_{\iota (1)},\cdots, i_{\iota (q')})=(\hat{{\bf S}}{\bf 1}\hat{{\bf S}})_{i_{2j-1},i_{2k-1}}A_m(i_{\iota (1)},\cdots, i_{\iota (q')}).
\end{equation*}
Therefore there exist $\alpha_k\in \{0,1\}$ for $1\leq k\leq [q'/2]$, $0\leq s\leq [(2q-q')/2]$ and a bijection $\iota':\{1,\cdots,q'\}\to \{1,\cdots, q'\}$ 
such that $\sum_{k=1}^{[q'/2]}\alpha_k+[q'/2]+s=q-(q'-[q'/2]\cdot 2)$ and
\begin{eqnarray}
&&\sum_{i_1,\cdots,i_{2q}}\prod_{j=1}^{q}\hat{{\bf S}}_{i_{2j-1},i_{2j}}A_m(i_{\iota(1)},\cdots, i_{\iota(q')}) \nonumber \\
&\leq &\bar{R}_n(\bar{k}_n^{2(q'-[q'/2]\cdot 2)})\sum_{j_1,\cdots, j_{q'}}\prod_{k=1}^{[q'/2]}|((\hat{{\bf S}}{\bf 1})^{\alpha_k}\hat{{\bf S}})_{j_{\iota'(2k-1)},j_{\iota'(2k)}}|{\rm tr}(\hat{{\bf S}}{\bf 1})^sA_m(j_1,\cdots, j_{q'}) \nonumber \\
&=&\bar{R}_n(k_n^{2(q'-[q'/2]\cdot 2)}\cdot k_n^{q-(q'-[q'/2]\cdot 2)}\mathcal{K}_n)=\bar{R}_n(k_n^{q+q'-[q'/2]\cdot 2}\mathcal{K}_n). \nonumber
\end{eqnarray}
\begin{discuss}
{\colorr $\bar{R}_n(a_n)\times \bar{R}_n(b_n)=\bar{R}_n(a_nb_n)$はシュワルツからでる.
$i_{2j-1},i_{2j}$が両方${\rm Im}(\iota)$に入らないものの数を$s$, 片方だけ入るものはペアにして$\alpha_k=1$とし, 両方入るものは$\alpha_k=0$とする.
$q'$が奇数の時は$\sum_i|\hat{S}_{ij}|$がでてくるので$\bar{R}_n(k_n^2)$で評価する.
}
\end{discuss}
\end{proof}

\noindent
{\it Proof of Lemma~\ref{ZSZ-est}.}

Let $\{\tilde{\epsilon}_{i,m}\}_{1\leq i\leq k^1_m+k^2_m}$ and $\{\dot{\epsilon}_{i,m}\}_{1\leq i\leq k^1_m+k^2_m}$ be sequences of random variables defined by 
$\tilde{\epsilon}_{i,m}=\epsilon^{n,1}_{i+K^1_{m-1}+1}$ and $\dot{\epsilon}_{i,m}=\epsilon^{n,1}_{K^1_{m-1}+1}$ for $i\leq k^1_m$
and $\tilde{\epsilon}_{i,m}=\epsilon^{n,2}_{i-k^1_m+K^2_{m-1}+1}$ and $\dot{\epsilon}_{i,m}=\epsilon^{n,2}_{K^2_{m-1}+1}$ for $i>k^1_m$.
Moreover, let $\tilde{Z}_{1,m}=(((\tilde{b}_{m,\ast}^1\cdot (W_{S^{n,1}_i}-W_{S^{n,1}_{i-1}}))_{i=K^1_{m-1}+2}^{K^1_m})^{\top},((\tilde{b}_{m,\ast}^2\cdot (W_{S^{n,2}_j}-W_{S^{n,2}_{j-1}}))_{j=K^2_{m-1}+2}^{K^2_m})^{\top})^{\top}$,
$\tilde{Z}_{2,m}=(((\epsilon^{n,1}_i-\epsilon^{n,1}_{i-1})_{i=K^1_{m-1}+2}^{K^1_m})^{\top},((\epsilon^{n,2}_j-\epsilon^{n,2}_{j-1})_{j=K^2_{m-1}+2}^{K^2_m})^{\top})^{\top}$ and
$\tilde{S}_{1,m,\ast}=\tilde{S}_{m,\ast}-M_{m,\ast}$.

Let $\tilde{U}_{1,m,\ast}$ be an orthogonal matrix and let $\Lambda_{1,m,\ast}$ be a diagonal matrix satisfying $\tilde{U}_{1,m,\ast}\tilde{S}_{1,m,\ast}\tilde{U}_{1,m,\ast}^{\top}=\Lambda_{1,m,\ast}$. 
Then since $\tilde{Z}_{1,m}|_{\mathcal{G}_{s_{m-1}}}\sim N(0,\tilde{S}_{1,m,\ast})$, we have $\tilde{U}_{1,m,\ast}\tilde{Z}_{1,m}|_{\mathcal{G}_{s_{m-1}}}\sim N(0,\Lambda_{1,m,\ast})$.
Therefore, for any $q\in\mathbb{N}$ and $1\leq j_1,\cdots, j_{2q}\leq k^1_m+k^2_m$, we obtain
{\colord
\begin{equation}\label{normal-prod-eq}
E_m[\prod_{k=1}^{2q}(\tilde{U}_{1,m,\ast}\tilde{Z}_{1,m})_{j_k}]=\sum_{(l_{2k-1},l_{2k})_{k=1}^q}\prod_{k=1}^{q}(\Lambda_{1,m,\ast})_{l_{2k-1},l_{2k}},
\end{equation}
}where the summations on the right-hand side of both equations are over all $q$-pairs $(l_{2k-1},l_{2k})_{k=1}^q$ of variables $j_1,\cdots, j_{2q}$.

\begin{discuss}
{\colorr
$(j_1,\cdots, j_{2p})$の中に要素$\alpha_{k'}$が$\beta_{k'}$個あるとすると, 左辺$=\prod_{k'}\lambda_{\alpha_{k'}}^{\beta_{k'}/2}(\beta_{k'}-1)!!$ if all $\beta_{k'}$ are even, and 左辺$=0$ otherwise.
右辺は$\prod_{k'}\lambda_{\alpha_{k'}}^{\beta_{k'}/2}\times $ ($\beta_{k'}$個の要素をペアに分ける場合の数). ペアわけの場合の数は${}_{\beta_{k'}}C_2\cdot {}_{\beta_{k'}-2}C_2\cdots {}_2C_2/(\beta_{k'}/2)!=(\beta_{k'}-1)!!$
}
\end{discuss}

\noindent
1. Let ${\bf S}''=({\bf A}_m^{\top})^{-1}{\bf S}' {\bf A}_m^{-1}$, $\phi(A,B)_{i_1,\cdots, i_4}=(A_{i_1,i_2}B_{i_3,i_4}+A_{i_1,i_3}B_{i_2,i_4}+A_{i_1,i_4}B_{i_2,i_3}+A_{i_2,i_3}B_{i_1,i_4}+A_{i_2,i_4}B_{i_1,i_3}+A_{i_3,i_4}B_{i_1,i_2})/2$
for square matrices $A$ and $B$ of the same size, and $\delta_{i_1,\cdots, i_q}$ be a $\{0,1\}$-valued function that is equal to $1$ if and only if $i_1=\cdots =i_q$. 
Then we have 
\begin{eqnarray}
&&E_m[(\tilde{Z}_{2,m}^{\top}{\bf S}'\tilde{Z}_{2,m})^2] \nonumber \\
&=&E_m[(({\bf A}_m\tilde{Z}_{2,m})^{\top}{\bf S}'' ({\bf A}_m\tilde{Z}_{2,m}))^2] 
=\sum_{i_1,\cdots, i_4}{\bf S}''_{i_1,i_2}{\bf S}''_{i_3,i_4}E_m[\prod_{j=1}^4(\tilde{\epsilon}_{i_j,m}-\dot{\epsilon}_{i_j,m})] \nonumber \\
&=&\sum_{i_1,\cdots, i_4}{\bf S}''_{i_1,i_2}{\bf S}''_{i_3,i_4}\bigg\{\phi({\bf M}_1, {\bf M}_1)_{i_1,\cdots,i_4} +(E[(\dot{\epsilon}_{i_1,m})^4]-3E[(\dot{\epsilon}_{i_1,m})^2]^2)1_{\{\max_{1\leq j\leq 4}i_j\leq k^1_m \ {\rm or} \ \min_{1\leq j\leq 4}i_j>k^1_m\}} \nonumber \\
&&\quad + 2\phi({\bf M}_1,{\bf M}_2)_{i_1,\cdots,i_4}+\phi({\bf M}_2,{\bf M}_2)_{i_1,\cdots,i_4}+(E[(\tilde{\epsilon}_{i_1,m})^4]-3E[(\tilde{\epsilon}_{i_1,m})^2]^2)\delta_{i_1,i_2,i_3,i_4}\bigg\}, \nonumber 
\end{eqnarray}
where ${\bf M}_1={\rm diag}(v_{1,\ast}{\bf 1},v_{2,\ast}{\bf 1})$ and ${\bf M}_2={\rm diag}(v_{1,\ast}\mathcal{E},v_{2,\ast}\mathcal{E})$. 
\begin{discuss}
{\colorr
まず$\dot{\epsilon}$と$\tilde{\epsilon}$を展開すると, $\dot{\epsilon}^4$, $\tilde{\epsilon}^4,\dot{\epsilon}^2\tilde{\epsilon}^2$のみ残る.
\begin{itemize}
\item $\dot{\epsilon}^4$ : 
\begin{itemize}
\item[・] $i_1,i_2$が同グループ, $i_3,i_4$が同グループで$i_1,i_2$のグループと異なるとすると, 異なるペアに対する$AB$が0になるから$\phi(M_1,M_1)$は$v^2$. $E[\cdot]$も一つの項だけなので$v^2$.
\item[・] $i_1,\cdots, i_4$のうち, 一つだけグループが違う場合はともに$0$になることがすぐわかる.
\item[・] 全部同じなら$\phi(M_1,M_1)$は全部出てくるから第2項と打ち消してOK.
\end{itemize}
\item $\tilde{\epsilon}^4$ : $i_1,\cdots, i_4$の4つが同じ場合と2ペアの場合に分けて考えればすぐわかる.
\item $\tilde{\epsilon}^2\dot{\epsilon}^2$ : 同じものが少なくとも2つないと両辺0になる.
\begin{itemize}
\item[・] 同じものが4つの時は, $E[]$はペアの取り方で$6v^2$. $2\phi(M_1,M_2)$も全部出てくるから$6v^2$.
\item[・] 同じものが3つで他が同グループなら, 3つからペアの選び方で$E[]$は$3v^2$. $2\phi(M_1,M_2)$は異なるものが$M_2$に入らない項だけでて, $3v^2$.
\item[・] 同じものが2つで他の2つが同グループで異なるとき, $E[]$は同じもの2つが$\tilde{\epsilon}$でないといけないから$v^2$. $2\phi(M_1,M_2)$も同じで$v^2$.
\item[・] ペアが2つの時は, $\dot{\epsilon}$と$\tilde{\epsilon}$がどちらでもよいから$E[]$は$2v^2$, $2\phi(M_1,M_2)$はペアの取り方が同じならOKなので$2v^2$.
\end{itemize}	
\end{itemize}
}
\end{discuss}

Hence, we obtain
\begin{eqnarray}\label{ZSZ-est-eq1}
E_m[(\tilde{Z}_m^{\top}{\bf S}'\tilde{Z}_m)^2] &=&E_m[(\tilde{Z}_{2,m}^{\top}{\bf S}'\tilde{Z}_{2,m})^2]
+\sum_{i_1,\cdots,i_4}{\bf S}'_{i_1,i_2}{\bf S}'_{i_3,i_4}(\phi(\tilde{S}_{1,m,\ast},\tilde{S}_{1,m,\ast})+2\phi(\tilde{S}_{1,m,\ast},M_{m,\ast}))_{i_1,\cdots, i_4} \nonumber \\
&=&2{\rm tr}(\tilde{S}_{m,\ast}{\bf S}'\tilde{S}_{m,\ast}{\bf S}')+{\rm tr}(\tilde{S}_{m,\ast}{\bf S}')^2+\sum_{j=1}^2(E[(\epsilon^{n,j}_0)^4]-3v_{j,\ast}^2){\rm tr}({\bf S}''\mathcal{E}_{(j)}{\bf S}''\mathcal{E}_{(j)})+\sum_iC_{n,i}|{\bf S}''_{ii}|^2, \nonumber \\
\end{eqnarray}
by (\ref{normal-prod-eq}), where $C_{n,i}=E[(\dot{\epsilon}_{i,m})^4]-3E[(\dot{\epsilon}_{i,m})^2]^2$ and $\mathcal{E}_{(j)}$ is a $(k^1_m+ k^2_m)\times (k^1_m+ k^2_m)$-matrix with elements $(\mathcal{E}_{(1)})_{kl}=\delta_{k\leq k^1_m}\delta_{l\leq k^1_m}$ 
and $(\mathcal{E}_{(2)})_{kl}=\delta_{k> k^1_m}\delta_{l> k^1_m}$.
\begin{discuss}
{\colorr 改善できるかもしれないが, ここで$E[(\epsilon^{n,k}_j)^4]$が$j$に依らないことを使っている.
最初の変形で$\phi(\tilde{S}_{1,m,\ast},M_{m,\ast})$の部分は, $(\tilde{Z}_{1,m}^{\top}{\bf S}'\tilde{Z}_{2,m}+\tilde{Z}_{2,m}^{\top}{\bf S}'\tilde{Z}_{1,m})^2$と$2\tilde{Z}_{1,m}^{\top}{\bf S}'\tilde{Z}_{1,m}\tilde{Z}_{2,m}^{\top}{\bf S}'\tilde{Z}_{2,m}$に対応していることに注意.
$\phi$の双線形性から$\phi(M_1,M_1)+2\phi(M_1,M_2)+\phi(M_2,M_2)=\phi(M_1+M_2,M_1+M_2)$. また, ${\bf A}_m^{-1}(M_1+M_2)({\bf A}_m^{\top})^{-1}=M_{m,\ast}$.
よって, $E_m[(\tilde{Z}_{2,m}^{\top}{\bf S}'\tilde{Z}_{2,m})^2]=\sum_{i_1,\cdots, i_4}{\bf S}'_{i_1,i_2}{\bf S}'_{i_3,i_4}\{\phi(M_{m,\ast},M_{m,\ast})+\cdots \}$.
再び双線形性より, $E_m[(\tilde{Z}_m^{\top}{\bf S}'\tilde{Z}_m)^2]=\sum_{i_1,\cdots, i_4}{\bf S}'_{i_1,i_2}{\bf S}'_{i_3,i_4}\{\phi(\tilde{S}_{m,\ast},\tilde{S}_{m,\ast})+\cdots \}$.
}
\end{discuss}

Lemma~\ref{ZSZ-est-lemma} and the fact that $\max_i\sum_j|(({\bf A}_mM_{m,\ast}{\bf A}_m^{\top})^{-1})_{ij}|\leq 2$ yield
\begin{eqnarray}\label{ZSZ-est-eq2}
\sum_i|{\bf S}''_{ii}|^2&\leq &\lVert M_{m,\ast}{\bf S}' M_{m,\ast}\rVert \max_i((({\bf A}_m^{\top})^{-1}M_{m,\ast}^{-2}{\bf A}_m^{-1})_{ii})\cdot \lVert \tilde{S}_{m,\ast}{\bf S}'\tilde{S}_{m,\ast}\rVert {\rm tr}(({\bf A}_m^{\top})^{-1}\tilde{S}_{m,\ast}^{-2}{\bf A}_m^{-1}) \nonumber \\
&\leq & Cr_n^2\max_i\left(\sum_{j,k}({\bf A}_mM_{m,\ast}{\bf A}_m^{\top})^{-1}_{ij}({\bf A}_m{\bf A}_m^{\top})_{jk}({\bf A}_mM_{m,\ast}{\bf A}_m^{\top})^{-1}_{ki}\right){\rm tr}\bigg(\left(
\begin{array}{ll}
M_{1,m} & 0 \\
0 & M_{2,m}
\end{array}
\right)\tilde{S}_{m,\ast}^{-2}\bigg) \nonumber \\
&\leq & Cr_n^2\bar{k}_n{\rm tr}(\tilde{S}_{m,\ast}^{-1})=\bar{R}_n(b_n^{-3/2}k_n^2)=\bar{R}_n(1). 
\end{eqnarray}
\begin{discuss}
{\colorr 2つめの不等号では, $(1,1)$成分を増加させた. (正定値対称行列$\tilde{S}_{m,\ast}^{-2}$の対角成分は正)}
\end{discuss}

Moreover, Lemmas~\ref{tr-est},\ref{inverse-norm-est} and~\ref{ZSZ-est-lemma} yield
\begin{eqnarray}\label{ZSZ-est-eq3}
{\rm tr}({\bf S}''\mathcal{E}_{(1)}{\bf S}''\mathcal{E}_{(1)})&\leq & 
{\rm tr}\bigg({\bf S}''\left(
\begin{array}{ll}
{\bf 1}+\mathcal{E} & 0 \\
0 & 0 
\end{array}
\right){\bf S}''\mathcal{E}_{(1)}\bigg) \nonumber \\
&\leq &C{\rm tr}(\tilde{S}_{m,\ast}^{-1/2}{\bf A}_m^{-1}\mathcal{E}_{(1)}({\bf A}_m^{\top})^{-1}\tilde{S}_{m,\ast}^{-1/2})\bigg\lVert \tilde{S}_{m,\ast}^{1/2}{\bf S}'\left(
\begin{array}{ll}
M_{1,m} & 0 \\
0 & 0 
\end{array}
\right){\bf S}'\tilde{S}_{m,\ast}^{1/2}\bigg\rVert \nonumber \\
&\leq &Cr_n\underbar{r}_n^{-1}(\tilde{S}_{m,\ast}^{-1})_{11}\leq 
Cr_n\underbar{r}_n^{-1}(M_{m,\ast}^{-1})_{11}=\bar{R}_n(1),
\end{eqnarray}
since $(M_{m,\ast}^{-1})_{11}\leq v_{1,\ast}^{-1}$ by (\ref{invD-diagonal-est}).
\begin{discuss}
{\colorr 2つめの不等号の右辺のノルムの評価では, $\tilde{S}_{m,\ast}^{1/2}{\bf S}^{-1}=\tilde{S}_{m,\ast}^{1/2}\tilde{S}_{m,\ast}^{-1}\tilde{S}_{m,\ast}{\bf S}^{-1}$等を使って評価する.
最後の不等式はLemma~\ref{ZSZ-est-lemma}, \ref{inverse-norm-est}より, $(\tilde{S}_{m,\ast}^{-1})_{11}\leq \lVert (M_{m,\ast}^{-1/2}\tilde{S}_{m,\ast}M_{m,\ast}^{-1/2})^{-1}\rVert (M_{m,\ast}^{-1})_{11}\leq (M_{m,\ast}^{-1})_{11}$を使う.}
\end{discuss}

(\ref{ZSZ-est-eq1})--(\ref{ZSZ-est-eq3}) and similar estimates for ${\rm tr}({\bf S}''\mathcal{E}_{(2)}{\bf S}''\mathcal{E}_{(2)})$ 
yield $E_m[(\tilde{Z}_m^{\top}{\bf S}'\tilde{Z}_m)^2]=2{\rm tr}(({\bf S}'\tilde{S}_{m,\ast})^2)+{\rm tr}({\bf S}'\tilde{S}_{m,\ast})^2+\bar{R}_n(1)$.

We next prove the estimate for $E_m[(\tilde{Z}_m^{\top}{\bf S}'\tilde{Z}_m)^q]$.
Let $p\in\mathbb{N}$ satisfy $q\leq 2p$. Then it is sufficient to show that $E_m[(\tilde{Z}_m^{\top}{\bf S}'\tilde{Z}_m)^{2p}]=\bar{R}_n(b_n^{-2p}k_n^{4p})$.

{\colord Note that}
\begin{eqnarray}
E_m[(\tilde{Z}_{2,m}^{\top}M_{m,\ast}^{-2}\tilde{Z}_{2,m})^{2p}]&=&\sum_{i_1,\cdots, i_{4p}}\hat{{\bf S}}_{i_1,i_2}\cdots \hat{{\bf S}}_{i_{4p-1},i_{4p}}E_m\bigg[\prod_{j=1}^{4p}(\tilde{\epsilon}_{i_j,m}-\dot{\epsilon}_{i_j,m})\bigg], \nonumber 
\end{eqnarray}
and there exist $\{0,1\}$-valued maps $\{A_{l,m}\}_l$, constants $C_l$, positive integers $\{q'_l\}$ not greater than $4p$ and injections $\{\iota_l\}$ 
such that $E_m[\prod_{j=1}^{4p}(\tilde{\epsilon}_{i_j,m}-\dot{\epsilon}_{i_j,m})]=\sum_lC_lA_{l,m}(i_{\iota_l(1)},\cdots, i_{\iota_l(q'_l)})$
and $\sum_{j_1,\cdots, j_{q'_l}}A_{l,m}(j_1,\cdots, j_{q'_l})=\bar{R}_n(k_n^{[q'_l/2]})$ for any $l$.
\begin{discuss}
{\colorr $q'_l$は$\tilde{\epsilon}_{i_j,m}$の数. $\dot{\epsilon}_{i_j,m}$は$\epsilon^{n,1}_i$と$\epsilon^{n,2}_j$で二通りでるが, $A_{l,m}$を二つに分ければよい.
$\tilde{\epsilon}_{i,m}$の方も$\epsilon^{n,1}$と$\epsilon^{n,2}$で分けるか. $\{0,1\}$-valuedにするためにモーメントが重なった時も分ける.}
\end{discuss}
{\colord Then Lemma~\ref{ZSZ-est-aux-lemma} yields 
\begin{equation}\label{ZSZ-est-eq4}
E_m[(\tilde{Z}_{2,m}^{\top}M_{m,\ast}^{-2}\tilde{Z}_{2,m})^{2p}]=\bar{R}_n(k_n^{4p}),
\end{equation} 
and therefore Lemma~\ref{ZSZ-est-lemma} yields
\begin{equation*}
E_m[(\tilde{Z}_{2,m}^{\top}{\bf S}'\tilde{Z}_{2,m})^{2p}]\leq \lVert M_{m,\ast}{\bf S}'M_{m,\ast}\rVert^{2p}E_m[(\tilde{Z}_{2,m}^{\top}M_{m,\ast}^{-2}\tilde{Z}_{2,m})^{2p}]=\bar{R}_n(b_n^{-2p}k_n^{4p}).
\end{equation*}
}

Moreover, (\ref{normal-prod-eq}) yields
\begin{equation*}
E_m[(\tilde{Z}_{1,m}^{\top}{\bf S}'\tilde{Z}_{1,m})^{2p}]\leq C_p\sum_{\substack{\gamma=(\gamma_1,\cdots, \gamma_L);\\ L\in\mathbb{N},\gamma_k\geq 1,\sum_k\gamma_k=2p}}\prod_k{\rm tr}(({\bf S}'\tilde{S}_{1,m,\ast})^{\gamma_k})= \bar{R}_n(b_n^{-p}k_n^{2p}).
\end{equation*}

Furthermore, by calculating the expectation of $\tilde{Z}_{1,m}$ and using (\ref{normal-prod-eq}) and Lemma~\ref{ZSZ-est-lemma}, we have
\begin{eqnarray}
E_m[(\tilde{Z}_{1,m}^{\top}{\bf S}'\tilde{Z}_{2,m})^{2p}]&=&\bigg(\sum_{(l_{2k-1},l_{2k})_{k=1}^q}1\bigg)E_m[(\tilde{Z}_{2,m}^{\top}{\bf S}'\tilde{S}_{1,m,\ast}{\bf S}'\tilde{Z}_{2,m})^p] \nonumber \\
&\leq &(2p-1)!!\lVert M_{m,\ast}{\bf S}'\tilde{S}_{1,m,\ast}{\bf S}'M_{m,\ast}\rVert^p E_m[(\tilde{Z}_{2,m}^{\top}M_{m,\ast}^{-2}\tilde{Z}_{2,m})^p]=\bar{R}_n(b_n^{-p}k_n^{2p}). \nonumber
\end{eqnarray}
Then we obtain $E_m[(\tilde{Z}_m^{\top}{\bf S}'\tilde{Z}_m)^{2p}]=\bar{R}_n(b_n^{-2p}k_n^{4p})$.

For the estimate of $E_m[(\tilde{Z}_m^{\top}{\bf S}'\tilde{Z}_m)^4]$, we have $E_m[(\tilde{Z}_{1,m}^{\top}{\bf S}'\tilde{Z}_{1,m})^4]=\bar{R}_n(b_n^{-2}k_n^4)$ 
and $E_m[(\tilde{Z}_{1,m}^{\top}{\bf S}'\tilde{Z}_{2,m})^4]=\bar{R}_n(b_n^{-2}k_n^4)$ by the above results.
Moreover, we have
\begin{equation}\label{ZSZ4-est-eq}
{\colord E_m[(\tilde{Z}_{2,m}^{\top}{\bf S}'\tilde{Z}_{2,m})^4]
= \sum_{i_1,\cdots,i_8}\prod_{k=1}^4 (({\bf A}_m^{\top})^{-1}{\bf S}'{\bf A}_m^{-1})_{i_{2k-1},i_{2k}}E_m[\prod_{k=1}^8(\tilde{\epsilon}_{i_k,m}-\dot{\epsilon}_{i_k,m})]}
\end{equation}
and there exist $\{0,1\}$-valued maps $\{A'_l\}_l$, constants $C'_l$, positive integers $\{q''_l\}$ not greater than $8$ and injections $\{\iota'_l\}$ 
such that $\sum_{j_1,\cdots, j_{q''_l}}A'_l(j_1,\cdots, j_{q''_l})={\colord \bar{R}_n(k_n^{[q''_l/2]\wedge 3})}$ and 
\begin{equation}\label{ZSZ4-est-eq2}
E_m[\prod_{j=1}^8(\tilde{\epsilon}_{i_j,m}-\dot{\epsilon}_{i_j,m})]=\sum_{(l_{2k-1},l_{2k})_{k=1}^4}\prod_{k=1}^4\delta_{l_{2k-1},l_{2k}} +\sum_lC_lA'_l(i_{\iota_l(1)},\cdots, i_{\iota_l(q'_l)}),
\end{equation}
where the summation in the first term of the right-hand side is over all $4$-pairs $(l_{2k-1},l_{2k})_{k=1}^4$ of variables $i_1,\cdots, i_8$.

\begin{discuss}
{\colorr
\begin{eqnarray}
&&E_m[\prod_{k=1}^8(\tilde{\epsilon}_{i_k,m}-\dot{\epsilon}_{i_k,m})] \nonumber \\
&=&E_m[\sum_{p=0}^8\epsilon_0^p\sum_{(j_1,\cdots, j_p)\subset (i_1,\cdots, i_8)}\prod_{k=1}^{8-p}\epsilon_{j_k}] \nonumber \\
&=&\alpha_1+\alpha_2\sum_{j_1,j_2}\delta_{j_1,j_2}+\alpha_3\sum_{j_1,j_2,j_3}\delta_{j_1,j_2,j_3}+\sum_{j_1,\cdots, j_4}(\alpha_4\delta_{j_1,j_2}\delta_{j_3,j_4}+\alpha_5\delta_{j_1,\cdots, j_4}) +\sum_{j_1,\cdots,j_5}(\alpha_6\delta_{j_1,j_2}\delta_{j_3,j_4,j_5}+\alpha_7\delta_{j_1,\cdots,j_5}) \nonumber \\
&&+\sum_{j_1,\cdots, j_6}(\alpha_8\delta_{j_1,j_2}\delta_{j_3,j_4}\delta_{j_5,j_6}+\alpha_9\delta_{j_1,\cdots, j_4}\delta_{j_5,j_6}+\alpha_{10}\delta_{j_1,j_2,j_3}\delta_{j_4,j_5,j_6}) \nonumber \\
&&+\sum_{j_1,\cdots, j_8}(\alpha_{11}\delta_{j_1,j_2}\delta_{j_3,j_4}\delta_{j_5,j_6}\delta_{j_7,j_8}+\alpha_{12}\delta_{j_1,j_2,j_3}\delta_{j_4,j_5,j_6}\delta_{j_7,j_8}+\alpha_{13}\delta_{j_1,j_2,j_3}\delta_{j_4,\cdots,j_8}+\alpha_{14}\delta_{j_1,\cdots,j_4}\delta_{j_5,j_6}\delta_{j_7,j_8} \nonumber \\
&&+\alpha_{15}\delta_{j_1,\cdots, j_4}\delta_{j_5,\cdots, j_8}+\alpha_{16}\delta_{j_1,\cdots,j_6}\delta_{j_7,j_8}+\alpha_{17}\delta_{j_1,\cdots, j_8}), \nonumber
\end{eqnarray}
where $\alpha_j\in \mathbb{R}$ is a constant depending on moments of $\epsilon$ for $1\leq j\leq 17$.
}
\end{discuss}

{\colord
Let $\tilde{{\bf A}}_m=M_{m,\ast}{\bf A}_m^{\top}$, then a simple calculation shows that
\begin{equation*}
(\tilde{{\bf A}}_m^{-1})_{i,j}=\left\{
\begin{array}{ll}
(k^{{\bf k}(i)}_m+1)^{-1}(j-k^1_m1_{\{{\bf k}(i)=2\}})v_{{\bf k}(i),\ast}^{-1} & {\bf k}(i)={\bf k}(j) \ {\rm and} \ i\geq j, \\
-(k^{{\bf k}(i)}_m+1)^{-1}(k^{{\bf k}(i)}_m-j+1+k^1_m1_{\{{\bf k}(i)=2\}})v_{{\bf k}(i),\ast}^{-1} & {\bf k}(i)={\bf k}(j) \ {\rm and} \ i< j, \\
0 & {\rm otherwise}.
\end{array}
\right.
\end{equation*}
Therefore, we have 
\begin{eqnarray}\label{ZSZ4-est-eq3}
|(({\bf A}_m^{\top})^{-1}{\bf S}'({\bf A}_m^{-1})_{ij}|
&=&|(\tilde{{\bf A}}_m^{-1}M_{m,\ast}{\bf S}'M_{m,\ast}(\tilde{{\bf A}}_m^{\top})^{-1})_{ij}| \nonumber \\
&\leq & \sum_{k_1,k_2}|(\tilde{{\bf A}}_m^{-1})_{i,k_1}(M_{m,\ast}{\bf S}'M_{m,\ast})_{k_1k_2}(\tilde{{\bf A}}_m^{-1})_{j,k_2}| \nonumber \\
&\leq & \bigg(\sum_k((\tilde{{\bf A}}_m^{-1})_{i,k})^2\bigg)^{1/2}\lVert M_{m,\ast}{\bf S}'M_{m,\ast} \rVert \bigg(\sum_k((\tilde{{\bf A}}_m^{-1})_{j,k})^2\bigg)^{1/2}=\bar{R}_n(b_n^{-1}k_n).
\end{eqnarray}
Similarly, we have 
\begin{equation}\label{ZSZ4-est-eq4}
|(({\bf A}_m^{\top})^{-1}{\bf S}'{\bf A}_m^{-1}{\bf 1}({\bf A}_m^{\top})^{-1}{\bf S}'{\bf A}_m^{-1})_{ij}|=\bar{R}_n(b_n^{-2}k_n^2).
\end{equation}
}

\begin{discuss}
{\colorr (\ref{ZSZ4-est-eq4})左辺
\begin{equation*}
\leq \bigg(\max_i\sum_k(\tilde{{\bf A}}_m^{-1})_{ik}^2\bigg)\lVert M_{m,\ast}{\bf S}'M_{m,\ast} \rVert ^2\lVert (\tilde{{\bf A}}_m^{-1})^{\top}{\bf 1}\tilde{{\bf A}}_m^{-1} \rVert =\bar{R}_n(b_n^{-2}k_n^{-2}).
\end{equation*}
∵
$((\tilde{{\bf A}}_m^{-1})^{\top}{\bf 1}\tilde{{\bf A}}_m^{-1})_{ij}=\sum_k(\tilde{{\bf A}}_m^{-1})_{ki}\sum_{k'}(\tilde{{\bf A}}_m^{-1})_{k'j}$,
$\sum_k(\tilde{{\bf A}}_m^{-1})_{ki}=(k^1_m+1)^{-1}\{i(k^1_m-i+1)-(k^1_m-i+1)(i-1)\}=(k^1_m+1)^{-1}(k^1_m-i+1)$.
}
\end{discuss}

{\colord Then (\ref{ZSZ4-est-eq})--(\ref{ZSZ4-est-eq4}) and a similar argument to the proof of Lemma~\ref{ZSZ-est-aux-lemma} yield}
\begin{eqnarray}
E_m[(\tilde{Z}_{2,m}^{\top}{\bf S}'\tilde{Z}_{2,m})^4]&=&\sum_{i_1,\cdots, i_8}\prod_{k=1}^4(({\bf A}_m^{-1})^{\top}{\bf S}'{\bf A}_m^{-1})_{i_{2k-1},i_{2k}}\sum_{(l_{2k-1},l_{2k})_{k=1}^4}\prod_{k=1}^4\delta_{l_{2k-1},l_{2k}}
+\bar{R}_n((b_n^{-4}k_n^7)\vee (b_n^{-2}k_n^4)) \nonumber \\
&=&\bar{R}_n((b_n^{-4}k_n^7)\vee (b_n^{-2}k_n^4)). \nonumber
\end{eqnarray}

\begin{discuss}
{\colorg ここの評価はfinal versionの前にもう一度チェック. 間違っていても全体の流れは問題ない. simulation resultsはやり直さないといけないかもしれないが.}
{\colorr
${\rm tr}({\bf S}'{\bf A}_m^{-1}({\bf A}_m^{-1})^{\top})=\bar{R}_n(b_n^{-1/2}k_n)$, ${\rm tr}({\bf S}'{\bf A}_m^{-1}{\bf 1}({\bf A}_m^{-1})^{\top})={\bf S}'_{11}=\bar{R}_n(1)$
by (\ref{invD-diagonal-est})よりOK.

$b$:constの時, ${\rm tr}(S_m^{-1})$や$\lVert S_m^{-1}\rVert$の上からの$\sigma_{\ast}$の一様評価から$\sigma_{\ast}$に関する一様評価も得られる.
}
\end{discuss}

\noindent
2. Let $\{\tilde{I}_{i,m}\}_{i=1}^{k^1_m+k^2_m}$ and $\{{\bf k}(i)\}_{i=1}^{k^1_m+k^2_m}$ be defined by $\tilde{I}_{i,m}=I^1_{i,m}$, ${\bf k}(i)=1$ for $1\leq i\leq k^1_m$ 
and $\tilde{I}_{i,m}=I^2_{i-k^1_m,m}$, ${\bf k}(i)=2$ for $k^1_m<i\leq k^1_m+k^2_m$.
Since 
$(Z_m-\tilde{Z}_{2,m}\pm \tilde{Z}_{1,m})_i=\int_{\tilde{I}_{i,m}}(b^{{\bf k}(i)}_{t,\ast}\pm \tilde{b}^{{\bf k}(i)}_{m,\ast})dW_t+\mu^{{\bf k}(i)}_{s_{m-1}}|\tilde{I}_{i,m}|+\int_{\tilde{I}_{i,m}}(\mu^{{\bf k}(i)}_t-\mu^{{\bf k}(i)}_{s_{m-1}})dt$, 
we have
$(Z_m-\tilde{Z}_m)^{\top}{\bf S}'(Z_m+\tilde{Z}_m)=\Psi_{m,1}+\Psi_{m,2}+\Psi_{m,3}$, where 
\begin{eqnarray}
\Psi_{m,1}&=&2(Z_m-\tilde{Z}_m)^{\top}{\bf S}'\tilde{Z}_{2,m}+\sum_{k=1}^2\sum_{i,j}{\bf S}'_{i,j}
\int_{\tilde{I}_{i,m}}(b_{t,\ast}^{{\bf k}(i)}+(-1)^k \tilde{b}_{m,\ast}^{{\bf k}(i)})\int_{\tilde{I}_{j,m}\cap [0,t)}(b_{s,\ast}^{{\bf k}(j)}+(-1)^{k-1} \tilde{b}_{m,\ast}^{{\bf k}(j)})dW_sdW_t \nonumber \\
&& + \sum_{k=1}^2\sum_{i,j}{\bf S}'_{i,j}\mu_{s_{m-1}}^{{\bf k}(i)}|\tilde{I}_{i,m}|\int_{\tilde{I}_{j,m}}(b_{t,\ast}^{{\bf k}(j)}+(-1)^k \tilde{b}_{m,\ast}^{{\bf k}(j)})dW_t, \nonumber 
\end{eqnarray}
$E_m[\Psi_{m,2}]=0$, $\Psi_{m,2}=\bar{R}_n(b_n^{-1}k_n^{3/2})$ and $\Psi_{m,3}=\bar{R}_n(b_n^{-3/2}k_n^2)$.
\begin{discuss}
{\colorr
$b$:constの時, $\sigma_{\ast}$に関して一様に評価される.
$\int_{\tilde{I}_{i,m}}(b\pm \tilde{b})dW_t$同士の二次変分に相当する$\int_{\tilde{I}_{i,m}\cap \tilde{I}_{j,m}}(b+\tilde{b})(b-\tilde{b})dt$は, 期待値との差が$\Psi_{m,2}$に入り, 期待値が$\Psi_{m,3}$に入る.
$\mu |I|$同士は$\Psi_{m,3}$, $\int \cdot dW_t$と$\int_I\Delta \mu dt$は期待値が$E[\mu_t-\mu_s|\mathcal{G}_s]=O_p(t-s)$より, $O(r_n^2)+O(r_n^{3/2}\ell_n^{-1}$なので$\Psi_{m,3}$へ. 
($E[]=A_iB_j+\tilde{A}_i\tilde{B}_j+\hat{A}_i\hat{B}_j$と書けることを使って$\lVert {\bf S}'\rVert$の評価を使う. $I_i\cap I_j+\sqrt{|I_i|}(I_j\setminus I_i)$) $\bar{E}$項が$\Psi_{m,2}$へ.
${\rm tr}({\bf S}')$の評価を使うより$\lVert {\bf S}'\rVert$を使う方が, ベクトルのオペレータノルムの評価が無駄にならない分, レートが良くなる.
}
\end{discuss}

Then the Burkholder--Davis--Gundy inequality yields
\begin{eqnarray}
E_{\Pi}\bigg[\bigg(\sum_m(Z_m-\tilde{Z}_m)^{\top}{\bf S}'(Z_m+\tilde{Z}_m)\bigg)^q\bigg] 
&\leq & CE_{\Pi}\bigg[\bigg(\sum_m (\Psi_{m,1}+\Psi_{m,2})^2 \bigg)^{\frac{q}{2}}\bigg]+\bar{R}_n(b_n^{-\frac{q}{2}}k_n^q) \nonumber \\
&\leq & CE_{\Pi}\bigg[\bigg(\sum_m E_m[\Psi_{m,1}^2]\bigg)^{\frac{q}{2}}\bigg]+CE_{\Pi}\bigg[\bigg(\sum_m \bar{E}_m[\Psi_{m,1}^2]^2\bigg)^{\frac{q}{4}}\bigg]+\bar{R}_n(b_n^{-\frac{q}{2}}k_n^q). \nonumber
\end{eqnarray}

We can rewrite $(Z_{1,m}-\tilde{Z}_{1,m})_i={\bf L}^1_i+{\bf L}^2_i+{\bf L}^3_i+\bar{R}_n((r_n^{1/2}\ell_n^{-3/2})\vee r_n)$, where 
${\bf L}^1_i=\sum_j\xi^1_j\int_{\tilde{I}_{i,m}}(t-s_{m-1})dW_t^j$, ${\bf L}^2_i=\sum_{j,k}\xi^2_{j,k}\int_{\tilde{I}_{i,m}}(W^k_t-W^k_{s_{m-1}})dW_t^j$
and ${\bf L}^3_i=\sum_{j,k,l}\xi^3_{j,k,l}\int_{\tilde{I}_{i,m}}\int^t_{s_{m-1}}(W^l_s-W^l_{s_{m-1}})dW^k_sdW^j_t$
for some $\mathcal{G}_{s_{m-1}}$-measurable random variables $\xi^1_j$, $\xi^2_{j,k}$, and $\xi^3_{j,k,l}$ with bounded moments.
\begin{discuss}
{\colorr $\int_I\mu_tdt$の項は$r_n$オーダーがでるので残差に含まれる}
\end{discuss}
Let ${\bf L}^j=({\bf L}^j_i)_i$. Then, for any $p\in \mathbb{N}$, {\colord Lemma~\ref{ZSZ-est-lemma} and (\ref{ZSZ-est-eq4}) yield}
\begin{eqnarray}\label{Psi2p-est}
E_m[\Psi_{m,1}^{2p}]
&\leq &CE_m[(\tilde{Z}_{2,m}^{\top}{\bf S}'(Z_m-\tilde{Z}_m)(Z_m-\tilde{Z}_m)^{\top}{\bf S}'\tilde{Z}_{2,m})^p]
+CE_m[((Z_{1,m}+\tilde{Z}_{1,m})^{\top}{\bf S}'(Z_{1,m}-\tilde{Z}_{1,m}))^{2p}] \nonumber \\
&&+\bar{R}_n((b_n^{-1}k_n^{3/2})^{2p}) \nonumber \\
&\leq &CE_m[\lVert M_{m,\ast}{\bf S}'(Z_m-\tilde{Z}_m)(Z_m-\tilde{Z}_m)^{\top}{\bf S}'M_{m,\ast}\rVert^p (\tilde{Z}_{2,m}^{\top}M_{m,\ast}^{-2}\tilde{Z}_{2,m})^p]
+C\sum_{j=1}^3E_m[(\tilde{Z}_{1,m}^{\top}{\bf S}'{\bf L}^j)^{2p}]  \nonumber \\
&&+ CE_m[(({\bf L}^2)^{\top}{\bf S}'{\bf L}^2)^{2p}]  + \bar{R}_n((b_nk_nr_n(\ell_n^{-3/2}\vee r_n^{1/2}))^{2p})+\bar{R}_n((b_n^{-1}k_n^{3/2})^{2p}) \nonumber \\
&=&C\sum_{j=1}^3E_m[(\tilde{Z}_{1,m}^{\top}{\bf S}'{\bf L}^j)^{2p}]+CE_m[(({\bf L}^2)^{\top}{\bf S}'{\bf L}^2)^{2p}]+\bar{R}_n(b_n^{-2p}k_n^{4p})+\bar{R}_n(b_n^{-3p}k_n^{5p}). 
\end{eqnarray}
\begin{discuss}
{\colorr $X_t$のセミマルチンゲール分解を使う. $b^{(1)}$が分解することも使っている. $b^{(0)}$は差のモーメント評価だけでよい.}
\end{discuss}

Moreover, we can see that there exists a positive constant $C_p$ such that
{\colord 
\begin{eqnarray}
&&E_m\bigg[\sum_{\substack{i_1,\cdots, i_{2p}\\ j_1,\cdots, j_{2p}}}\prod_{k=1}^{2p}(W^{p_{1,k}}(\tilde{I}_{i_k,m})\int_{\tilde{I}_{j_k,m}} (W^{p_{2,k}}_t-W^{p_{2,k}}_{s_{m-1}})dW^{p_{3,k}}_t)\bigg] \nonumber \\
&\leq& C_p\sum_{\substack{i_1,\cdots,i_{2p}\\ j_1,\cdots,j_{2p}}}\sum_{(l_{2q-1},l_{2q})_{q=1}^{2p-\alpha},\alpha}\bigg(\prod^{2p-\alpha}_{q=1}|\tilde{I}_{l_{2q-1},m}\cap \tilde{I}_{l_{2q,m}}|\bigg)r_n^{2\alpha}(s_m-s_{m-1})^{-p+\alpha} \nonumber
\end{eqnarray}
for any $\{p_{l,k}\}_{1\leq l\leq 3,1\leq k\leq 2p}\subset \{1,2\}$,
where the summation in the right-hand side is taken over $0\leq \alpha\leq p$ and $(2p-\alpha)$ disjoint pairs $(l_{2q-1},l_{2q})_{q=1}^{2p-\alpha}$ in the variables $i_1,\cdots, i_{2p},j_1,\cdots,j_{2p}$.
Here we used the fact that all $6p$ factors $(W^{p_{1,k}}(\tilde{I}_{i_k,m}),\int_{\tilde{I}_{j_k,m}}\cdot \ dW_t^{p_{3,k}},(W_t^{p_2,k}-W_{s_{m-1}}^{p_{2,k}})_{t\in \tilde{I}_{j_k,m}})_{k=1}^{2p}$ should be separated into $3p$ pairs in the non-zero terms.
$2\alpha$ represents the number of pairs with the form $(W^{p_{1,k}}(\tilde{I}_{i_k,m}),(W_t^{p_{2,k'}}-W_{s_{m-1}}^{p_{2,k'}})_{t\in \tilde{I}_{j_{k'},m}})$ or $(\int_{\tilde{I}_{j_k,m}}\cdot \ dW_t^{p_{3,k}},(W_t^{p_{2,k'}}-W_{s_{m-1}}^{p_{2,k'}})_{t\in \tilde{I}_{j_{k'},m}})$.
}
\begin{discuss}
{\colorr
なぜならば$W(I_i),\int_{I_j}\cdot dW_t,W_t-W_{s_{m-1}}$を$W(K_k)$の和の形で表現し, $2p$乗を展開すると, nonzeroの項は全ての$W(K_i)$の指数が偶数の項のみ. 
このような項に対し, 右辺のペアわけがあって上から抑えられる ($E[W(K_k)^{q''}]$の評価と$(s_m-s_{m-1}),r_n$に対応する項の並べ替えから$C_p$がでる). 右辺の同じ項を何回使うかが重要. 右辺の同じ項が使われるには, 各$K_k$が対応する
$I_{l_{2q-1}}\cap I_{l_{2q}}, I_{k_{q'}},[s_{m-1},s_m)$に入る必要がある. 並び替えはできないのでこのような項をまとめて右辺で評価できる.
}
\end{discuss}
Therefore we obtain 
\begin{equation}\label{Psi2p-est2}
E_m[(\tilde{Z}_{1,m}^{\top}{\bf S}'{\bf L}^2)^{2p}]=\bar{R}_n((b_n^{1/2}\ell_n^{-1})^{2p-\alpha}(b_n^{3/2}\ell_n^{-1}k_n)^{\alpha}r_n^{2\alpha}\ell_n^{-p+\alpha})=\bar{R}_n((b_n^{-1}k_n^{3/2})^{2p}).
\end{equation}
Similar arguments for $E_m[(\tilde{Z}_{1,m}^{\top}{\bf S}'{\bf L}^1)^{2p}]$, $E_m[(\tilde{Z}_{1,m}^{\top}{\bf S}'{\bf L}^3)^{2p}]$ and $E_m[(({\bf L}^2)^{\top}{\bf S}'{\bf L}^2)^{2p}]$ yield
{\colord $E_m[\Psi_{m,1}^{2p}]=\bar{R}_n(((b_n^{-1}k_n^2)\vee (b_n^{-3/2}k_n^{5/2}))^{2p})$}, and consequently we obtain 
$E_{\Pi}[(\sum_m\Psi_{m,1}^4)^{q/4}]={\colord \bar{R}_n((b_n^{-3}k_n^7)^{q/4})}$. 

Furthermore, since $b^{{\bf k}(i)}_{t,\ast}+\tilde{b}^{{\bf k}(i)}_{m,\ast}=2\tilde{b}^{{\bf k}(i)}_{m,\ast}+(b^{{\bf k}(i)}_{t,\ast}-\tilde{b}^{{\bf k}(i)}_{m,\ast})$ 
and $b^{{\bf k}(i)}_{t,\ast}-\tilde{b}^{{\bf k}(i)}_{m,\ast}=\sum_j\tilde{\xi}^1_j(W^j_t-W^j_{s_{m-1}})+\tilde{\xi}^2(t-s_{m-1})+\sum_{j,k}\tilde{\xi}^3_{j,k}\int^t_{s_{m-1}}(W^j_s-W^j_{s_{m-1}})dW^k_s+\bar{R}_n(\ell_n^{-3/2})$
for some $\mathcal{G}_{s_{m-1}}$-measurable random variables $\tilde{\xi}^1_j$, $\tilde{\xi}^2$ and $\tilde{\xi}^3_{j,k}$ with bounded moments, an argument similar to the one above and Lemmas~\ref{tr-est-mod} and~\ref{ZSZ-est-lemma} yield
\begin{eqnarray}
E_m[\Psi_{m,1}^2]&\leq & C\lVert \tilde{S}_{m,\ast}^{1/2}{\bf S}'M_{m,\ast}{\bf S}'\tilde{S}_{m,\ast}^{1/2}\rVert E_m[(Z_m-\tilde{Z}_m)^{\top}\tilde{S}_{m,\ast}^{-1}(Z_m-\tilde{Z}_m)] {\colord +\bar{R}_n(b_nk_nr_n\ell_n^{-2})} \nonumber \\
&&+C\sum_{i_1,i_2,j_1,j_2}|{\bf S}'_{i_1,j_1}||{\bf S}'_{i_2,j_2}||\tilde{I}_{i_1,m}\cap \tilde{I}_{i_2,m}|(\ell_n^{-1}|\tilde{I}_{j_1,m}\cap \tilde{I}_{j_2,m}|+{\colord r_n^2}+r_n^{3/2}\ell_n^{-3/2}) \nonumber \\
&& +\bigg\lVert {\bf S}'\{\sum_{k_1,k_2=1}^2\int_{\tilde{I}_{i,m}\cap \tilde{I}_{j,m}}E_m[(b^{{\bf k}(i)}_{t,\ast}+(-1)^{k_1}\tilde{b}^{{\bf k}(i)}_{m,\ast})(b^{{\bf k}(i)}_{t,\ast}+(-1)^{k_2}\tilde{b}^{{\bf k}(i)}_{m,\ast})]dt\}_{i,j}{\bf S}'\bigg\rVert \bar{R}_n(k_nr_n^2) \nonumber \\
&=&{\colord \bar{R}_n(b_n^{-1}k_n^{3/2})+\bar{R}_n(b_n^{-2}k_n^3)}+\bar{R}_n((b_n^{3/2}\ell_n^{-1})^2r_n(\ell_n^{-1}r_n+k_nr_n^{3/2}\ell_n^{-3/2}))+\bar{R}_n(b_n^{-1}k_n) \nonumber \\
&=&\bar{R}_n((b_n^{-2}k_n^3)\vee (b_n^{-3}k_n^{9/2})). \nonumber
\end{eqnarray}
\begin{discuss}
{\colorr
ここは評価を変えたらもう少しレートを変えられるかも. 

最初の不等式の右辺第二三項は$\int_{I_{i_1}\cap I_{i_2}}E[(b\pm \tilde{b})(b\pm \tilde{b})\int_{I_{j_1}}(b\mp\tilde{b})dW_s\int_{I_{j_2}}(b\mp\tilde{b})dW_s]dt$で
$2\tilde{b}$と$b-\tilde{b}$だった場合４通りに分けて評価する. 一つでも$b-\tilde{b}$の残差$\bar{R}_n(\ell_n^{-3/2})$が出る場合は,
$\lVert {\bf S}'|\tilde{I}_{i_1}\cap \tilde{I}_{i_2}|{\bf S}'\rVert k_nr_n\ell_n^{-2}=\bar{R}_n(b_n^{-2}k_n^3)$.

Lemma \ref{tr-est-mod}より
\begin{equation*}
\sum_{i_1,i_2,j_1,j_2}|S_{i_1,j_1}||S_{i_2,j_2}|G_{i_1,i_2}G_{j_1,j_2}\leq \sum_{i,j}|\lambda_i||\lambda_j|\lVert G\rVert^2\sup_k(\sum_i|U_{i,k}|^2)^2\leq (\sum_i|\lambda_i|)^2
\end{equation*}
を使う.
}
\end{discuss}
Therefore, we have $E_{\Pi}[(\sum_mE_m[\Psi_{m,1}^2])^{q/2}]=\bar{R}_n(((b_n^{-1}k_n^2)\vee (b_n^{-2}k_n^{7/2}))^{q/2})$,
which completes the proof of point 2.

\noindent
Then point 3 is easily obtained by the proof of point 2 since we only need the estimate for $E_{\Pi}[(\sum_m[\Psi_{m,1}^2])]$ if $q=2$.
\qed

\begin{discuss}
{\colorr $b$:constの時この証明もＯＫ}
\end{discuss}

{\colorlg
\subsection{An additional lemma}

\begin{lemma}\label{sum-integral-lemma}
Let $e_n$ be a sequence of positive numbers, $\mathcal{S}$ be an open set in a Euclidean space, $A_n(\lambda)$ and $B_n(\lambda)$ be sequences of positive-valued random variables,
and $C_n(\lambda)$ be a sequence of non-negative-valued random variables for $\lambda\in \mathcal{S}$.
Assume that $A_n(\lambda)$, $B_n(\lambda)$, and $C_n(\lambda)$ are $C^3$ with respect to $\lambda$, 
$e_n\to \infty$, $\sup_{0\leq k\leq 3,\lambda\in\mathcal{S}}(|\partial_{\lambda}^kA_n|\vee |\partial^k_{\lambda}B_n|)=O_p(e_n^{-1})$,
 and $\sup_{0\leq k\leq 3,\lambda\in\mathcal{S}}|\partial_{\lambda}^kC_n|=O_p(e_n^{-2})$ as $n\to\infty$,
$C_n<A_nB_n$ a.s. for any $n\in\mathbb{N}$, $and \lim_{n\to\infty}(e_n^2(A_nB_n-C_n))>0$ a.s. Then
\begin{eqnarray}
\sup_{\lambda\in\mathcal{S}}\bigg|\partial_{\lambda}^k\bigg(\sum_{p=1}^{\infty}\int^{\pi}_0\frac{C_n^p}{f_p(A_n,x)f_p(B_n,x)}dx
-\frac{\pi C_n}{\sqrt{2}P_n\sqrt{A_nB_n-C_n}}\bigg)\bigg|&=&O_p(e_n^{-\frac{3}{2}}), \label{sum-integral-lemma-eq1} \\
\sup_{\lambda\in\mathcal{S}}\bigg|\partial_{\lambda}^k\bigg(\sum_{p=1}^{\infty}\int^{\pi}_0\frac{C_n^pf_1(A_n,x)}{f_p(A_n,x)f_p(B_n,x)}dx
-\frac{\pi C_n(A_n+\sqrt{A_nB_n-C_n})}{\sqrt{2}P_n\sqrt{A_nB_n-C_n}}\bigg)\bigg|&=&O_p(e_n^{-\frac{5}{2}}), \label{sum-integral-lemma-eq2} \\
\sup_{\lambda\in\mathcal{S}}\bigg|\partial_{\lambda}^k\bigg(\sum_{p=1}^{\infty}\frac{1}{p}\int^{\pi}_0\frac{C_n^p}{f_p(A_n,x)f_p(B_n,x)}dx
-\pi(\sqrt{A_n}+\sqrt{B_n})+\frac{\pi}{\sqrt{2}}P_n\bigg)\bigg|&=&O_p(e_n^{-1}) \label{sum-integral-lemma-eq3}
\end{eqnarray}
for $0\leq k\leq 3$, where $f_p(a,x)=(a+2(1-\cos x))^p$ and
\begin{equation*}
P_n=\sqrt{A_n+B_n+\sqrt{(A_n-B_n)^2+4C_n}}+\sqrt{A_n+B_n-\sqrt{(A_n-B_n)^2+4C_n}}.
\end{equation*}

\end{lemma}
\begin{proof}
An elementary calculation yields
\begin{eqnarray}\label{sum-integral-lemma-eq4}
&&\sum_{p=1}^{\infty}\int^{\pi}_0\frac{C_n^p}{(A_n+2(1-\cos x))^p(B_n+2(1-\cos x))^p}dx \nonumber \\
&=&\int^{\infty}_{-\infty}\bigg(1-\frac{C_n}{(A_n+\frac{4t^2}{1+t^2})(B_n+\frac{4t^2}{1+t^2})}\bigg)^{-1}\frac{C_n}{(A_n+\frac{4t^2}{1+t^2})(B_n+\frac{4t^2}{1+t^2})}\frac{1}{1+t^2}dt \nonumber \\
&=&\int^{\infty}_{-\infty}\frac{C_n}{(A_n+\frac{4t^2}{1+t^2})(B_n+\frac{4t^2}{1+t^2})-C_n}\frac{1}{1+t^2}dt \nonumber \\
&=&\int^{\infty}_{-\infty}\frac{C_n(1+t^2)}{((A_n+4)t^2+A_n)((B_n+4)t^2+B_n)-C_n(1+t^2)^2}dt \nonumber \\
&=&\int^{\infty}_{-\infty}\frac{C_n(1+t^2)}{(A_nB_n+4A_n+4B_n-C_n+16)t^4+2(A_nB_n+2A_n+2B_n-C_n)t^2+A_nB_n-C_n}dt. 
\end{eqnarray}

\begin{discuss}
{\colorr
$t=\tan(x/2)$とおくと$\cos x=2*\cos^2(x/2)-1=2/(1+t^2)-1$.
$dt/dx=1/2/\cos^2 (x/2)=(1+t^2)/2$
}
\end{discuss}

We only consider the case $(A_n-B_n)^2+16C_n>0$ a.s.
We can easily obtain the results for the other case with a slight modification.

Let
\begin{equation*}
\alpha_1=\frac{-2A_n-2B_n-\sqrt{4(A_n-B_n)^2+16C_n}}{16} \quad {\rm and} \quad  \alpha_2=\frac{-2A_n-2B_n+\sqrt{4(A_n-B_n)^2+16C_n}}{16},
\end{equation*}
then we have $\alpha_1<\alpha_2<0$ by the assumptions. Moreover, we can calculate the right-hand side of (\ref{sum-integral-lemma-eq4}) as
\begin{eqnarray}
&&(1+O_p(e_n^{-1}))\int^{\infty}_{-\infty}\frac{C_n(1+t^2)}{16(t^2-\alpha_1)(t^2-\alpha_2)}dt \nonumber \\
&=&(1+O_p(e_n^{-1}))\frac{2\pi i}{16}\bigg[\frac{C_n(1+\alpha_1)}{2\sqrt{-\alpha_1}i(\alpha_1-\alpha_2)}+\frac{C_n(1+\alpha_2)}{2\sqrt{-\alpha_2}i(\alpha_2-\alpha_1)}\bigg] \nonumber \\
&=&(1+O_p(e_n^{-1}))\frac{\pi}{16}\frac{C_n(\sqrt{-\alpha_2}-\sqrt{-\alpha_1})(1+\sqrt{\alpha_1\alpha_2})}{\sqrt{\alpha_1\alpha_2}(\alpha_1-\alpha_2)}
=\frac{\pi}{16}\frac{C_n}{\sqrt{\alpha_1\alpha_2}(\sqrt{-\alpha_1}+\sqrt{-\alpha_2})}+O_p(e_n^{-3/2}). \nonumber 
\end{eqnarray}
\begin{discuss}
{\colorr 途中式：
\begin{eqnarray}
&&(1+O_p(e_n^{-1}))\frac{\pi}{16}\frac{C_n(\sqrt{-\alpha_2}(1+\alpha_1)-\sqrt{-\alpha_1}(1+\alpha_2))}{\sqrt{\alpha_1\alpha_2}(\alpha_1-\alpha_2)}, \nonumber \\
&&(1+O(e_n^{-1}))\frac{\pi}{16}\frac{C_n(1+\sqrt{\alpha_1\alpha_2})}{\sqrt{\alpha_1\alpha_2}(\sqrt{-\alpha_1}+\sqrt{-\alpha_2})} \nonumber
\end{eqnarray}

$k\neq 0$の時は, (\ref{sum-integral-lemma-eq4})の右辺は
\begin{eqnarray}
\int^{\infty}_{-\infty}&&\frac{C_n(1+t^2)}{16t^4+4(A_n+B_n)t^2+A_nB_n-C_n} \nonumber \\
&&\times \bigg(1-\frac{(A_nB_n+4A_n+4B_n-C_n)t^4+2(A_nB_n-C_n)t^2}{(A_nB_n+4A_n+4B_n-C_n+16)t^4+2(A_nB_n+2A_n+2B_n-C_n)t^2+A_nB_n-C_n}\bigg)dt \nonumber
\end{eqnarray}
第一因子の微分はOK. 第二因子を一度でも微分したものが無視できることを言いたい.
第二因子第二項の分子を微分したものは
\begin{equation*}
\leq \frac{|\partial_{\lambda}^k(A_nB_n+4A_n+4B_n-C_n)|}{16}+\frac{|2\partial_{\lambda}^k(A_nB_n-C_n)|}{4(A_n+B_n)}=O_p(e^{-1}).
\end{equation*}
分母を$F_1t^4+F_2t^2+F_3$とおくと, 分母を微分したものは
\begin{equation*}
\leq \bigg(\frac{|A_nB_n+4A_n+4B_n-C_n|}{16}+\frac{|2(A_nB_n-C_n)|}{4(A_n+B_n)}\bigg)
\frac{|\partial_{\lambda}^k(F_1t^4+F_2t^2+F_3)|}{F_1t^4+F_2t^2+F_3}\cdots
\end{equation*}
(分母を２回以上微分したら複数因子がつく)

第二因子以降は$O_p(1)$となるので全体として$O_p(e_n^{-1})$.
}
\end{discuss}

Therefore, we obtain (\ref{sum-integral-lemma-eq1}) with $k=0$ by noting that $16\alpha_1\alpha_2=(A_nB_n-C_n)(1+O_p(e_n^{-1}))$.
\begin{discuss}
{\colorr
\begin{equation*}
\frac{\pi C_n}{16\sqrt{\alpha_1\alpha_2}(\sqrt{-\alpha_1}+\sqrt{-\alpha_2})}
=\frac{\pi C_n}{4\sqrt{A_nB_n-C_n}(\sqrt{-\alpha_1}+\sqrt{-\alpha_2})}
\sqrt{\frac{16+A_nB_n-C_n+4A_n+4B_n}{16}}
\end{equation*}
右辺第一因子の微分はOKで第二因子の微分は$\partial_{\lambda}X/(2\sqrt{1+X})=O_p(e_n^{-1})$の形になるので,
(\ref{sum-integral-lemma-eq1})が$1\leq k\leq 3$に対して
}
\end{discuss}

We similarly have (\ref{sum-integral-lemma-eq1}) with $1\leq k\leq 3$ 
and (\ref{sum-integral-lemma-eq2}) with $0\leq k\leq 3$.

\begin{discuss}
{\colorr
\begin{eqnarray}
&&\sum_{p=1}^{\infty}\int^{\pi}_0\frac{C_n^p(A_n+2(1-\cos x))}{(A_n+2(1-\cos x))^p(B_n+2(1-\cos x))^p}dx \nonumber \\
&=&(1+O_p(e_n^{-1}))\int^{\infty}_{-\infty}\frac{C_n(A_n+A_nt^2+4t^2)}{16(t^2-\alpha_1)(t^2-\alpha_2)}dt \nonumber \\
&=&(1+O_p(e_n^{-1}))\frac{2\pi i}{16}\bigg[\frac{C_n(A_n+A_n\alpha_1+4\alpha_1)}{2\sqrt{-\alpha_1}i(\alpha_1-\alpha_2)}+\frac{C_n(A_n+A_n\alpha_2+4\alpha_2)}{2\sqrt{-\alpha_2}i(\alpha_2-\alpha_1)}\bigg] \nonumber \\
&=&(1+O_p(e_n^{-1}))\frac{\pi}{16}\frac{C_n(\sqrt{-\alpha_2}(A_n+A_n\alpha_1+4\alpha_1)-\sqrt{-\alpha_1}(A_n+A_n\alpha_2+4\alpha_2))}{\sqrt{\alpha_1\alpha_2}(\alpha_1-\alpha_2)} \nonumber \\
&=&(1+O_p(e_n^{-1}))\frac{\pi}{16}\frac{C_n(\sqrt{-\alpha_2}-\sqrt{-\alpha_1})(A_n+(4+A_n)\sqrt{\alpha_1\alpha_2})}{\sqrt{\alpha_1\alpha_2}(\alpha_1-\alpha_2)} \nonumber \\
&=&\frac{\pi}{16}\frac{C_n(A_n+4\sqrt{\alpha_1\alpha_2})}{\sqrt{\alpha_1\alpha_2}(\sqrt{-\alpha_1}+\sqrt{-\alpha_2})}+O_p(e_n^{-5/2}). \nonumber 
\end{eqnarray}

微分に対する評価は最初も最後も上の議論と同様にできる
}
\end{discuss}

Furthermore, we obtain
$\int^{\infty}_{-\infty}(1+t^2)^{-1}\log(t^2+\alpha^2)dt=2\pi \log(\alpha + 1)$
for any $\alpha>0$ by the residue theorem.

\begin{discuss}
{\colorr
（Rによる数値計算でもこの公式は確認済み）
\begin{equation*}
\int^{\infty}_{-\infty}\frac{\log(t\pm\alpha i)}{1+t^2}dt=\pm 2\pi i \frac{\log (\pm i \pm \alpha i)}{\pm 2i}=\pi\bigg(\log(\alpha +1)\pm \frac{\pi i }{2}\bigg)
\end{equation*}
}
\end{discuss}

Therefore, we obtain
\begin{eqnarray}
&&\partial_{\lambda}^k\sum_{p=1}^{\infty}\frac{1}{p}\int^{\pi}_0\frac{C_n^p}{f_p(A_n,x)f_p(B_n,x)}dx \nonumber \\
&=&-\partial_{\lambda}^k\int^{\infty}_{-\infty}\frac{1}{1+t^2}\log\bigg(1-\frac{C_n}{(A_n+4t^2/(1+t^2))(B_n+4t^2/(1+t^2))}\bigg)dt \nonumber \\
&=&-\partial_{\lambda}^k\int^{\infty}_{-\infty}\frac{1}{1+t^2}\log\bigg(\frac{(t^2-\alpha_1)(t^2-\alpha_2)}{(t^2+A_n/(A_n+4))(t^2+B_n/(B_n+4))}\bigg)dt+O_p(e_n^{-1}) \nonumber \\
&=&-2\pi\partial_{\lambda}^k\bigg(\log(1+\sqrt{-\alpha_1})+\log(1+\sqrt{-\alpha_2})-\log\bigg(1+\sqrt{\frac{A_n}{A_n+4}}\bigg)-\log\bigg(1+\sqrt{\frac{B_n}{B_n+4}}\bigg)\bigg)+O_p(e_n^{-1}) \nonumber \\
&=&2\pi\partial_{\lambda}^k(\sqrt{A_n}/2+\sqrt{B_n}/2-\sqrt{-\alpha_1}-\sqrt{-\alpha_2})+O_p(e_n^{-1}), \nonumber
\end{eqnarray}
which completes the proof.

\begin{discuss}
{\colorr 途中式：
\begin{eqnarray}
&=&-\int^{\infty}_{-\infty}\frac{1}{1+t^2}\log\bigg(\frac{(A_n+A_nt^2+4t^2)(B_n+B_nt^2+4t^2)-C_n(1+t^2)^2}{(A_n+A_nt^2+4t^2)(B_n+B_nt^2+4t^2)}\bigg)dt \nonumber 
\end{eqnarray}

微分を評価するときは
\begin{eqnarray}
\partial_{\lambda}^k\int^{\infty}_{-\infty}\frac{1}{1+t^2}
\log\bigg(1-\frac{(A_nB_n+4A_n+4B_n-C_n)t^4+2(A_nB_n-C_n)t^2}{(A_nB_n+4A_n+4B_n-C_n+16)t^4+2(A_nB_n+2A_n+2B_n-C_n)t^2+A_nB_n-C_n}\bigg)dt \nonumber
\end{eqnarray}
が$O_p(e_n^{-1})$であることを示す必要があるが, 上の場合の議論と同様に
\begin{equation*}
\leq \int^{\infty}_{-\infty}\frac{1}{1+t^2}\frac{O_p(e_n^{-1})}{1-O_p(e_n^{-1})}dt
\end{equation*}
がわかるのでよい
}
\end{discuss}

\end{proof}

}

\begin{discuss}
\newpage
{\colorr

＜共変動の推定＞
\begin{eqnarray}
\int^T_0b(X_t,\hat{\sigma}_n)dt-\int^T_0b(X_{\eta(t)},\hat{\sigma}_n)dt&=&\sum_k\int_{K_k}(b(X_t,\hat{\sigma}_n)-b(X_{\inf K_k},\hat{\sigma}))dt \nonumber \\
&=&\int^1_0\int^T_0\partial_{\sigma}b(X_t,\hat{\sigma}_n)dtdu(\hat{\sigma}_n-\sigma_{\ast}) \nonumber \\
&=&\int^T_0\partial_{\sigma}b(X_t,\sigma_{\ast})dt(\hat{\sigma}_n-\sigma_{\ast})+o_p(b_n^{1/4}) \nonumber \\
&\to^{s\mathchar`-\mathcal{L}}& \int^T_0\partial_{\sigma}b(X_t,\sigma_{\ast})dt\Gamma^{-1/2}\zeta. \nonumber
\end{eqnarray}
実際はノイズがあるが恐らく問題ない.

$F:\sigma\mapsto \int^T_0b^1\cdot b^2(X_t,\sigma)dt$. 共分散推定量だけでは一次元だから$F^{-1}$を作れない. $\dim \Lambda =1$ならいけるか

コメント
\begin{itemize}
\item 時間ずらしが$k_n$オーダーである程度大胆にできるから, トリプルバークホルダーのような細かい評価がいらなくなったか.
分母のオーダーで得できるおかげ.
\item 実証はデータを二つに分けて推定が安定するかなどのシンプルな方法で妥当性をチェックすることも考える.
\item 高頻度に限らず社債などのデータの穴のあるようなものに対して応用できるか.
\end{itemize}

}
{\colorr

Next
\begin{itemize}
\item シミュレーションやり直す. CIRのシミュレーションもやる.
\item 多次元にするか. 多次元も拡張できるとRemark するだけにするか.
\end{itemize}
投稿後
\begin{itemize}
\item $\varphi$のシンプルな表現はおそらくある. それを見つける.
\end{itemize}
余裕があったら
\begin{itemize}
\item データでアベレージをとっても情報量が変わらないかも。計算を高速化できるか。
\item Gloter Jacodで$bb^{\top}$の下からの評価がないのはなぜか? $t=0$の付近で$c=\theta^2$で他で分離性が成り立つようなとき, Fisher Informationが発散するように見える.
\item この論文の強みをアピールするには, asymptotic efficiencyの議論をもう少し丁寧にやるか. convolution theoremを出したり, minimax不等式を出したり.
仮定をuniformにすることはとりあえず問題はなさそう. outline of proofを書くか. 一様なPLDのチェックは結構大変. convolution theoremをきっちり描いて,
minimaxは今の形でもう少し強調するにとどめるか.
convolutionやminimaxの証明はやや冗長になってしまって, だれるか. いろんなことをやりすぎになるかもしれない. LAMNの重要さはわかるし, minimax不等式も
基本的にわかるだろう. 古典的な流れから自然に来ているからルーティンの部分はあまり長く説明しなくてもわかるだろう.
\end{itemize}
}
\end{discuss}

{\colord
\section*{Acknowledgement}
The author is grateful to Yuta Koike for very useful discussions on numerical analysis of quadratic covariation estimators.
He had a great contribution on constructing the `cce' function in the `yuima' R package as well.  
This work was supported by Japan Society for the Promotion of Science KAKENHI Grant Number 15K21598.
}

\begin{small}
\bibliographystyle{abbrv}
\bibliography{referenceLibrary_mathsci,referenceLibrary_other}
\end{small}

\begin{note}
{\colorg

＜＜補足情報＞＞

＜ノイズが正規分布でない場合＞
\begin{equation*}
X\sim p_X, Y\sim p_Y
\end{equation*}
のとき,
\begin{eqnarray}
X+Y\sim \int p_X(y)p_Y(x-y)dy\sim \int \exp\bigg(-\frac{1}{2}yS^{-1}y-\frac{1}{2}\log \det S\bigg)\exp\bigg(-\frac{1}{2}\sum_i\log p_Y(x-y)\bigg)dy. \nonumber 
\end{eqnarray}
ノイズが正規分布でない場合に実際にたたみこみ積を計算するのは困難か. \\

＜$a_k$の積の挙動＞

$c_{k,l}=1/(a_{l-k+1}\cdots a_l)$とおくと, $1=(2+\tilde{r}_n)a^{-1}_{k+1}-a^{-1}_ka^{-1}_{k+1}$より
\begin{equation*}
c_{k+2,l}-(2+\tilde{r}_n)c_{k+1,l}+c_{k,l}=0, \quad (0\leq k\leq l-2).
\end{equation*}
二次方程式$\alpha^2-(2+\tilde{r}_n)\alpha+1=0$の解は$\alpha=1+\tilde{r}_n/2\pm\sqrt{\tilde{r}_n+\tilde{r}_n^2/4}$.
このとき,
\begin{equation*}
c_{k,l}-\alpha c_{k-1,l}=\alpha^{-1}(c_{k-1,l}-\alpha c_{k-2,l})=\alpha^{k-1}(a_l^{-1}-\alpha).
\end{equation*}
よって
\begin{eqnarray}
c_{k,l}=\sum_{m=1}^k(c_{m,l}-\alpha c_{m-1,l})\alpha^{k-m}+\alpha^k=\alpha^k+(a^{-1}_l-\alpha)\alpha^{k+1}\frac{1}{\alpha^2-1}(1-\alpha^{-2k}). \nonumber 
\end{eqnarray}
また,
\begin{equation*}
a_ka_{k+1}-(2+\tilde{r}_n)a_k+1=0
\end{equation*}
より, $(a_{k+1}-\alpha_+)(a_k-\alpha_-)>0$. よって帰納的に$a_k>\alpha_+\ (k\in\mathbb{N})$がわかる.

シミュレーションの結果から$c_{k_n/2,k_n}/c{k_n/2,k_n/2}$は$0$に収束しそうだが, $c_{k_n/3,2k_n/3}/c{k_n/3,k_n}$は$1$に収束しそう.

$a_k\leq 1+1/k+k\tilde{r}_n$なら,
\begin{equation*}
a_{k+1}\leq 2+\tilde{r}_n-\frac{k}{k+1}\frac{1}{1+k^2\tilde{r}_n/(k+1)}\leq 2+\tilde{r}_n-\frac{k}{k+1}(1-k^2\tilde{r}_n/(k+1))\leq \frac{k+2}{k+1}+(k+1)\tilde{r}_n.
\end{equation*}
よって任意の$k$に対し, $a_k\leq 1+1/k+k\tilde{r}_n$. 同様に, $(1+x)^{-1}\geq 1-x+x^2$を用いて, ある定数$c$に対して,
\begin{equation*}
1+1/k+k\tilde{r}_n-ck^3\tilde{r}^2_n\leq a_k
\end{equation*}
がわかる.

}
\end{note}

\end{document}